\documentclass[10pt]{article}
\usepackage{amsmath}
\usepackage{amsfonts}
\usepackage{amssymb}
\usepackage{amsthm}
\usepackage{enumerate}
\usepackage[totalwidth=450pt, totalheight=650pt]{geometry}

\begin{document}

\newcommand{\C}{{\mathbb{C}}}
\newcommand{\R}{{\mathbb{R}}}
\newcommand{\Z}{{\mathbb{Z}}}
\newcommand{\N}{{\mathbb{N}}}
\newcommand{\Q}{B}
\newcommand{\q}{\left}
\newcommand{\w}{\right}
\newcommand{\Vol}[1]{\mathrm{Vol}\q(#1\w)}
\newcommand{\B}[4]{B_{\q(#1,#2\w)}\q(#3,#4\w)}
\newcommand{\CjN}[3]{\q\|#1\w\|_{C^{#2}\q(#3\w)}}
\newcommand{\Cj}[2]{C^{#1}\q( #2\w)}
\newcommand{\grad}{\bigtriangledown}
\newcommand{\sI}[2]{\mathcal{I}\q(#1,#2 \w)}
\newcommand{\Det}[1]{\det_{#1\times #1}}
\newcommand{\sK}{\mathcal{K}}
\newcommand{\sKt}{\widetilde{\mathcal{K}}}
\newcommand{\sA}{\mathcal{A}}
\newcommand{\sB}{\mathcal{B}}
\newcommand{\sC}{\mathcal{C}}
\newcommand{\sD}{\mathcal{D}}
\newcommand{\sS}{\mathcal{S}}
\newcommand{\sF}{\mathcal{F}}
\newcommand{\sQ}{\mathcal{Q}}
\newcommand{\sV}{\mathcal{V}}
\newcommand{\cV}{\q( \sV\w)}
\newcommand{\vsig}{\varsigma}
\newcommand{\vsigt}{\widetilde{\vsig}}
\newcommand{\dil}[2]{#1^{\q(#2\w)}}
\newcommand{\lA}{-\log_2 \sA}
\newcommand{\eh}{\widehat{e}}
\newcommand{\Ho}{\mathbb{H}^1}
\newcommand{\sd}{\sum d}
\newcommand{\dt}{\tilde{d}}
\newcommand{\dhc}{\hat{d}}
\newcommand{\Span}[1]{\mathrm{span}\q\{ #1 \w\}}
\newcommand{\dspan}[1]{\dim \Span{#1}}
\newcommand{\K}{K_0}
\newcommand{\ad}[1]{\mathrm{ad}\q( #1 \w)}
\newcommand{\LtOpN}[1]{\q\|#1\w\|_{L^2\rightarrow L^2}}
\newcommand{\LpOpN}[2]{\q\|#2\w\|_{L^{#1}\rightarrow L^{#1}}}
\newcommand{\LpN}[2]{\q\|#2\w\|_{L^{#1}}}
\newcommand{\Jac}{\mathrm{Jac}\:}
\newcommand{\kapt}{\widetilde{\kappa}}
\newcommand{\gt}{\widetilde{\gamma}}
\newcommand{\gtt}{\widetilde{\widetilde{\gamma}}}
\newcommand{\gh}{\widehat{\gamma}}
\newcommand{\Sh}{\widehat{S}}
\newcommand{\Wh}{\widehat{W}}
\newcommand{\Ih}{\widehat{I}}
\newcommand{\Wt}{\widetilde{W}}
\newcommand{\Xt}{\widetilde{X}}
\newcommand{\Tt}{\widetilde{T}}
\newcommand{\Nt}{\widetilde{N}}
\newcommand{\Phit}{\widetilde{\Phi}}
\newcommand{\Vh}{\widehat{V}}
\newcommand{\Xh}{\widehat{X}}
\newcommand{\sSh}{\widehat{\mathcal{S}}}
\newcommand{\sFh}{\widehat{\mathcal{F}}}
\newcommand{\thetah}{\widehat{\theta}}
\newcommand{\ct}{\widetilde{c}}
\newcommand{\at}{\tilde{a}}
\newcommand{\bt}{\tilde{b}}
\newcommand{\fg}{\mathfrak{g}}
\newcommand{\cC}{\q( \sC\w)}
\newcommand{\cG}{\q( \sC_{\fg}\w)}
\newcommand{\cJ}{\q( \sC_{J}\w)}
\newcommand{\cY}{\q( \sC_{Y}\w)}
\newcommand{\cZ}{\q( \sC_{Z}\w)}
\newcommand{\cYu}{\q( \sC_{Y}^u\w)}
\newcommand{\cGu}{\q( \sC_{\fg}^u\w)}
\newcommand{\cJu}{\q( \sC_{J}^u\w)}
\newcommand{\cZu}{\q( \sC_{Z}^u\w)}
\newcommand{\cYs}{\q( \sC_{Y}^s\w)}
\newcommand{\cGs}{\q( \sC_{\fg}^s\w)}
\newcommand{\cJs}{\q( \sC_{J}^s\w)}
\newcommand{\cZs}{\q( \sC_{Z}^s\w)}
\newcommand{\Bb}{\overline{B}}
\newcommand{\Qb}{\overline{\Q}}
\newcommand{\sP}{\mathcal{P}}
\newcommand{\sY}{\mathcal{Y}}
\newcommand{\sN}{\mathcal{N}}
\newcommand{\cH}{\q(\mathcal{H}\w)}
\newcommand{\Omegat}{\widetilde{\Omega}}
\newcommand{\Lie}{\mathrm{Lie}}
\newcommand{\thetat}{\widetilde{\theta}}
\newcommand{\psit}{\widetilde{\psi}}
\newcommand{\supp}[1]{\mathrm{supp}\q(#1\w)}
\newcommand{\sM}{\mathcal{M}}

\newcommand{\Lpp}[2]{L^{#1}\q(#2\w)}
\newcommand{\Lppn}[3]{\q\|#3\w\|_{\Lpp{#1}{#2}}}

\newtheorem{thm}{Theorem}[section]
\newtheorem{cor}[thm]{Corollary}
\newtheorem{prop}[thm]{Proposition}
\newtheorem{lemma}[thm]{Lemma}
\newtheorem{conj}[thm]{Conjecture}

\theoremstyle{remark}
\newtheorem{rmk}[thm]{Remark}

\theoremstyle{definition}
\newtheorem{defn}[thm]{Definition}

\theoremstyle{remark}
\newtheorem{example}[thm]{Example}

\numberwithin{equation}{section}

\title{Multi-parameter singular Radon transforms I: the $L^2$ theory}
\author{Brian Street\footnote{The author was partially supported by NSF DMS-0802587.}}
\date{}

\maketitle

\begin{abstract}
The purpose of this paper is to study the $L^2$ boundedness of operators of the form
\[
f\mapsto \psi(x) \int f(\gamma_t(x)) K(t)\: dt,
\]
where $\gamma_t(x)$ is a $C^\infty$ function defined on a neighborhood
of the origin in $(t,x)\in \R^N\times \R^n$, satisfying
$\gamma_0(x)\equiv x$, $\psi$ is a $C^\infty$ cutoff function
supported on a small neighborhood of $0\in \R^n$, and $K$
is a ``multi-parameter singular kernel'' supported on a small neighborhood
of $0\in \R^N$.
The goal is, given an appropriate class of kernels $K$, to give
conditions on $\gamma$ such that every operator of the
above form is bounded on $L^2$.
The case when $K$ is a Calder\'on-Zygmund kernel was studied
by Christ, Nagel, Stein, and Wainger; we generalize their conditions to the case when $K$ has a ``multi-parameter''
structure.  For example, when $K$ is given by a ``product kernel.''
Even when $K$ is a Calder\'on-Zygmund kernel, our methods yield
some new results.
This is the first paper in a three part series, the later two
of which are joint with E. M. Stein.  The second paper
deals with the related question of $L^p$ boundedness, while the third
paper deals with the special case when $\gamma$ is real analytic.

\end{abstract}

\tableofcontents

\section{Introduction}\label{SectionIntro}
The purpose of this paper is to prove the $L^2$ boundedness
of multi-parameter singular Radon transforms.  These operators
generalize the well-studied (single-parameter) singular Radon
transforms.  The study of the $L^p$ boundedness of
these single-parameter singular Radon transforms cumulated
in the work of Christ, Nagel, Stein, and Wainger
\cite{ChristNagelSteinWaingerSingularAndMaximalRadonTransforms}.
They studied operators of the form,
\begin{equation}\label{EqnIntroOp}
T f\q( x\w) = \psi\q( x\w) \int f\q( \gamma_t\q( x\w)\w) K\q( t\w)\: dt,
\end{equation}
where $\psi$ is a $C_0^\infty$ cut-off function, supported near $0\in \R^n$,
$\gamma_t\q( x\w) = \gamma\q( t,x\w)$ is a $C^\infty$ map defined
in a neighborhood of the origin in $\R^N\times \R^n$ satisfying
$\gamma_0\q( x\w) \equiv x$, and $K$ is a standard Calder\'on-Zygmund
kernel on $\R^N$ supported for $t$ near $0$.
This means that $K$ satisfies
\begin{equation*}
\q|\partial_t^{\alpha} K\q( t\w) \w|\lesssim \q|t\w|^{-N-\q|\alpha\w|}, \quad t\ne 0,
\end{equation*}
along with a certain ``cancellation condition'' which we will make precise
later (see Definition \ref{DefnProdKer}).\footnote{Actually, 
\cite{ChristNagelSteinWaingerSingularAndMaximalRadonTransforms} 
restricts attention to homogeneous kernels $K$ (and replaces the
integral in \eqref{EqnIntroOp} with $\int_{\q|t\w|<a}$, where $a>0$
is some small number).  This restriction to homogeneous kernels
is not essential to their work.}
In \cite{ChristNagelSteinWaingerSingularAndMaximalRadonTransforms},
appropriate ``curvature conditions'' on
$\gamma$ were assumed to guarantee $T$ extends to a bounded map
$L^p\rightarrow L^p$ ($1<p<\infty$).
See Section \ref{SectionCNSW}  for a further discussion of these curvature conditions.
  The goal of
this paper is to replace $K$ in \eqref{EqnIntroOp} with an appropriate
multi-parameter kernel and develop conditions on $\gamma$
which will allow us to show $T$ is bounded on $L^2$.

For instance, one might consider the case when $K$ is a so-called ``product kernel''
supported near $0$.  To define this notion, suppose we have
decomposed $\R^N=\R^{N_1}\times \cdots \times \R^{N_\nu}$,
and write $t=\q( t_1,\ldots, t_\nu\w)\in \R^{N_1}\times \cdots \times \R^{N_\nu}$.
A product kernel
satisfies
\begin{equation*}
\q|\partial_{t_1}^{\alpha_1} \cdots \partial_{t_\nu}^{\alpha_\nu} K\q( t\w)\w| \lesssim \q|t_1\w|^{-N_1-\q|\alpha_1\w|}\cdots \q|t_\nu\w|^{-N_\nu-\q|\alpha_\nu\w|},
\end{equation*}
again along with certain ``cancellation conditions'' (see Definition \ref{DefnProdKer}).\footnote{The simplest example of a product kernel
is $K\q(t_1,\ldots, t_\nu\w)= K_1\q(t_1\w)\otimes \cdots\otimes K_\nu\q(t_\nu\w)$, where
$K_1,\ldots, K_\nu$ are Calder\'on-Zygmund kernels.}
We will develop conditions on $\gamma$ so that the operator
$T$ given by \eqref{EqnIntroOp} extends to a bounded operator
on $L^2$ for every such product kernel (relative to a fixed
decomposition of $\R^N$).
In addition, we will study kernels more general than these
product kernels.  When restricted to the single-parameter
case, our results will imply the $L^2$ results of
\cite{ChristNagelSteinWaingerSingularAndMaximalRadonTransforms}--and,
in fact, generalize the $L^2$ results of \cite{ChristNagelSteinWaingerSingularAndMaximalRadonTransforms},
even in the single-parameter case (see Section \ref{SectionSpecialCase} for a discussion of this).

This paper is the first in a three part series, the second
two of which are joint with Elias Stein.  The second paper,
\cite{SteinStreetMultiParameterSingRadonLp}
will
investigate the corresponding $L^p$ theory
($1<p<\infty$).  The third paper, \cite{StreetMultiParameterSingRadonAnal},
will discuss, in detail, the special case when $\gamma$ is a real
analytic function--a situation where more can be said.
We give an overview of the results of \cite{StreetMultiParameterSingRadonAnal}
in Section \ref{SectionOverview}, to give the reader a view
of this entire project.

The outline of this paper is as follows.  
In Section \ref{SectionOverview} we discuss
the sequel to this work \cite{StreetMultiParameterSingRadonAnal}.
In Section \ref{SectionSpecialCase}
we discuss some special cases which motivate our main definitions and theorem.
Sections \ref{SectionKernels}-\ref{SectionCurves} give all
the definitions and notation necessary to state our main theorem.
Section \ref{SectionResults} includes the statement of the main theorem.
Section \ref{SectionCNSW} gives an overview of some of the main
aspects of \cite{ChristNagelSteinWaingerSingularAndMaximalRadonTransforms}
we will use, along with some minor extensions that will be necessary
for our proofs.
Section \ref{SectionProofOutline} gives a high-level overview
of the proof.
Sections \ref{SectionCCII}-\ref{SectionL1delta} are devoted
to stating the technical results and definitions necessary for our proof.
In Section \ref{SectionGenThm} a general $L^2$ theorem is stated and proved,
which will imply our main theorem.  In Section \ref{SectionProof}
the general theorem from Section \ref{SectionGenThm} is used to prove
the main theorem of the paper.  In Section \ref{SectionKernelsII}
we offer a deeper study of the class of kernels we use.
Finally, in Section \ref{SectionExamples}, a number of examples are given
where the main theorem applies.

\par\noindent \textbf{Acknowledgements}
I would like to thank Eli Stein.  This project was born from
a discussion in his office, where
upon hearing the results of \cite{StreetMultiParameterCCBalls},
he 
suggested that perhaps they could be used
to answer questions like the ones investigated in this paper.
In addition, he explained to me some of the surrounding theory.
Once the main results for this paper were completed, he and
I combined efforts to complete the project
in \cite{SteinStreetMultiParameterSingRadonLp,StreetMultiParameterSingRadonAnal}.
I would also like to thank Alex Nagel who informed me of the paper
\cite{CarberyWaingerWrightSingularIntegralsAndTheNewtonDiagram},
and Steve Wainger who described it (and related results) to me in detail.
These discussions helped shape my thoughts on pure powers and non-pure
powers as described in this paper.
Finally, I found it helpful to know some of the recent 
theory
developed by Nagel, Ricci, Stein, and Wainger 
\cite{NagelRicciSteinWaingerSingularIntegralWithFlagKernelsOnHomogeneousGroupsI}
regarding
convolution operators given by flag kernels on nilpotent groups.
I thank all four of them for describing various aspects of it to me.

	\subsection{Informal statement of the main result}\label{SectionIntroRes}
	In this section, we informally state the special case of our main theorem
when $K\q(t_1,\ldots, t_\nu\w)$ is a product kernel
relative to the decomposition $\R^N=\R^{N_1}\times \cdots\times \R^{N_\nu}$ (see
the introduction and Definition \ref{DefnProdKer} for this notion).
We suppose we are given a $C^\infty$ function, $\gamma\q(t,x\w)=\gamma_t\q(x\w)\in \R^n$ defined
on a small neighborhood of the origin in $\q(t,x\w)\in \R^N\times \R^n$,
satisfying $\gamma_0\q(x\w)\equiv x$.
For $t$ sufficiently small, $\gamma_t$ is a diffeomorhpism onto
its image.  Thus, it makes sense to write $\gamma_t^{-1}$, the inverse mapping. 
We define the vector field
\begin{equation*}
W\q(t,x\w)= \frac{d}{d\epsilon}\bigg|_{\epsilon=1} \gamma_{\epsilon t}\circ \gamma_t^{-1} \q(x\w)\in T_x \R^n.
\end{equation*}

For a collection of vector fields $\sV$, let $\sD\q(\sV\w)$ denote
the involutive distribution generated by $\sV$.  I.e., the smallest
$C^\infty$ module containing $\sV$ and such that if $X,Y\in \sD\q(\sV\w)$
then $\q[X,Y\w]\in \sD\q(\sV\w)$.
For a multi-index $\alpha\in \N^N$, write $\alpha=\q(\alpha_1,\ldots, \alpha_\nu\w)$, with $\alpha_\mu\in \N^{N_\mu}$.

Decompose $W$ into a Taylor series in the $t$ variable,
\begin{equation*}
W\q(t,x\w)\sim \sum_{\alpha} t^{\alpha}X_\alpha.
\end{equation*}
We call $\alpha=\q(\alpha_1,\ldots, \alpha_\nu\w)\in \N^N=\N^{N_1}\times \cdots\times \N^{N_\nu}$ a pure
power if $\alpha_\mu\ne 0$ for precisely one $\mu$.  Otherwise,
we call it a non-pure power.

We assume that the following conditions hold ``uniformly'' for
$\delta=\q(\delta_1,\ldots, \delta_\nu\w)\in \q(0,1\w]^\nu$,
though we defer making this notion of uniform
precise to Section \ref{SectionCurves}.
\begin{itemize}
\item For every $\delta\in \q(0,1\w]^\nu$,
\begin{equation*}
\sD_{\delta}:=\sD\q(\q\{\delta_1^{\q|\alpha_1\w|}\cdots\delta_\nu^{\q|\alpha_\nu\w|}X_{\alpha_1,\ldots,\alpha_\nu} : \q(\alpha_1,\ldots,\alpha_\nu\w)\text{ is a pure power}\w\}\w)
\end{equation*}
is finitely generated as a $C^\infty$ module, uniformly in $\delta$.

\item For every $\delta\in \q(0,1\w]^\nu$,
\begin{equation*}
W\q(\delta_1 t_1,\ldots, \delta_\nu t_\nu\w)\in \sD_\delta,
\end{equation*}
uniformly in $\delta$.
\end{itemize}

\begin{rmk}
If it were not for the ``uniform'' aspect of the above assumptions, they
would be independent of $\delta$.  Thus it is the uniform part,
which we have not made precise, that is the heart of the above
assumptions.
\end{rmk}

Our main theorem is,
\begin{thm}
Under the above assumptions (which are made
precise in Section \ref{SectionCurves}), the operator given by
\begin{equation*}
f\mapsto \psi\q(x\w) \int f\q(\gamma_t\q(x\w)\w) K\q(t\w)\: dt,
\end{equation*}
is bounded on $L^2$, for every product kernel $K\q(t_1,\ldots, t_\nu\w)$,
with sufficiently small support, provided $\psi$ has sufficiently small support.
\end{thm}

The precise statement of our main result (where more general
kernels are considered) can be
found in Theorems \ref{ThmMainThmFirstPass} and \ref{ThmMainThmSecondPass}.

\section{Overview of the series}\label{SectionOverview}
One way to view this paper is as the foundation for the theory
developed in the sequels where (in joint works with E. M. Stein) the corresponding
$L^p$ theory is developed \cite{SteinStreetMultiParameterSingRadonLp},
and the special case when $\gamma$ is real analytic is
investigated \cite{StreetMultiParameterSingRadonAnal}.
In this section, we discuss the main results of \cite{StreetMultiParameterSingRadonAnal}, where $\gamma$ is assumed to be real analytic.
For the purposes of this overview, we restrict our attention to the case when $K$ is a ``product kernel,''
though more general cases will be covered in
the later sections and in
\cite{SteinStreetMultiParameterSingRadonLp,StreetMultiParameterSingRadonAnal}.
All of the results discussed here will be contained 
in \cite{StreetMultiParameterSingRadonAnal}.

The setting is as follows.  
We decompose $\R^N$ into $\nu$ factors,
$\R^N=\R^{N_1}\times\cdots\times \R^{N_\nu}$,
and we write $t=\q(t_1,\ldots, t_\nu\w)\in \R^N$.
We consider product kernels, $K\q(t_1,\ldots, t_\nu\w)$,
relative to this decomposition (see the introduction
and Definition \ref{DefnProdKer} for the notion of
a product kernel).  In particular, when $\nu=1$,
$K$ is standard Calder\'on-Zygmund kernel.
We suppose that we are given a germ
of a real analytic function defined on a neighborhood of the
origin in $\R^N\times \R^n$:\footnote{Here, we have used the notation
$f:\R^n_0\rightarrow \R^m$ to denote that $f$ is a germ of a function
defined on a neighborhood of $0\in \R^n$.}
\begin{equation*}
\gamma\q(t,x\w)=\gamma_t\q(x\w):\R^N_0\times \R^n_0\rightarrow \R^n,
\end{equation*}
satisfying $\gamma_0\q(x\w)\equiv x$.

There are two, closely related, operators which are of interest to us:
\begin{itemize}
\item $Tf\q(x\w)=\psi\q(x\w)\int f\q(\gamma_{t_1,\ldots, t_\nu}\q(x\w)\w) K\q(t_1,\ldots, t_\nu\w) \: dt_1\cdots dt_\nu$,
where $K$ product kernel supported near $t=0$ and $\psi$ is a $C^\infty$
cut-off function supported near $x=0$.
\item $\sM f\q(x\w) = \psi\q(x\w) \sup_{0<\delta_1,\ldots, \delta_\nu<<1} \int_{\q|t\w|\leq 1} \q|f\q(\gamma_{\delta_1 t_1,\ldots, \delta_\nu t_\nu }\q(x\w)\w)\w|\: dt_1\cdots dt_\nu$, where $\psi$ is a $C^\infty$ cut-off function supported near
$x=0$, and the supremum is taken over $\delta_1,\ldots, \delta_\nu$ small.
\end{itemize}
Note that the above operators only make sense if $K$ and $\psi$ are supported
sufficiently close to $0$, since $\gamma$ is only defined on a neighborhood
of $0$.
Our goal is two-fold:
\begin{enumerate}[(I)]
\item\label{ItemTGoal} Give conditions on $\gamma$ such that $T$ is bounded $L^p$ ($1<p<\infty$)
for {\it every} such product kernel.
\item\label{ItemMGoal} Give conditions on $\gamma$ such that $\sM$ is bounded on $L^p$ ($1<p\leq \infty$).
\end{enumerate}
This paper is only concerned with Goal \ref{ItemTGoal} in the case $p=2$, though
Goal \ref{ItemMGoal} is an essential part of 
\cite{SteinStreetMultiParameterSingRadonLp,StreetMultiParameterSingRadonAnal}.

In this paper and \cite{SteinStreetMultiParameterSingRadonLp}, we consider
the more general case when $\gamma$ is $C^\infty$.  When $\gamma$ is $C^\infty$, we
put two types\footnote{We do not make the dichotomy between these
two types of conditions explicit in this paper:  it is only
particularly relevant when $\gamma$ is real analytic, and
we defer a precise discussion of these
issues to
\cite{StreetMultiParameterSingRadonAnal}.}
 of assumptions on $\gamma$ to ensure that $T$ and $\sM$ are bounded
on $L^p$:
\begin{itemize}
\item A ``finite type'' condition.  This condition can be seen
as a generalization of the condition in \cite{ChristNagelSteinWaingerSingularAndMaximalRadonTransforms}, where the case $\nu=1$ (i.e., when $K$ is a Calder\'on-Zygmund kernel) is treated.
\item An ``algebraic'' condition.  This condition holds
automatically in the case $\nu=1$ (i.e., when $K$ is a Calder\'on-Zygmund kernel).
\end{itemize}
For a discussion of these two conditions in an easy to understand
special case, we refer the reader to Section \ref{SectionSpecialCase}.
When we restrict attention to the case
that $\gamma$ is real analytic in \cite{StreetMultiParameterSingRadonAnal},
we will use a Weierstrass-type preparation theorem (from \cite{GalligoTheoremeDeDivisionEtStabiliteEnGeometrieAnalytiqueLocale}) to help show:
\begin{itemize}
\item The ``finite-type'' condition holds automatically when $\gamma$
is real analytic.
\item The ``algebraic'' condition is not necessary for the $L^p$
boundedness of $\sM$ when $\gamma$ is real analytic (though it is necessary, in some cases, for
the $L^p$ boundedness of $T$).
\end{itemize}
Because of this, our results take a much simpler form when
$\gamma$ is assumed to be real analytic.
We state them here.

\begin{thm}[\cite{StreetMultiParameterSingRadonAnal}]\label{ThmRealAnalM}
When $\gamma$ is real analytic, $\sM$ is bounded on $L^p$ ($1<p\leq \infty$)
under no additional assumptions. 
\end{thm}

\begin{rmk}
There are at least two special cases of Theorem \ref{ThmRealAnalM}
in the literature.
The first is due to Bourgain \cite{BourgainARemarkOnTheMaximalFunctionAssociatedToAnAnalyticVectorField}
in the case $\nu=1$.  He studied the case
$\gamma\q(t,x\w):\R^1_0\times \R^2_0\rightarrow \R^2$ given by
$\gamma\q(t,x\w)= x-tv\q(x\w)$, where $v:\R^2_0\rightarrow \R^2$ is a real
analytic vector field.
Another special case of Theorem \ref{ThmRealAnalT} is due to Christ
\cite{ChristTheStrongMaximalFunctionOnANilpotentGroup}, where 
left
invariant maximal functions (of a certain form) on a nilpotent Lie group 
are studied.
\end{rmk}

We now turn to $T$.  We begin by stating the corollary of our general
result in the special case $\nu=1$,

\begin{cor}[\cite{StreetMultiParameterSingRadonAnal}]\label{CorRealAnalT}
When $\nu=1$ (i.e., when $K$ is a Calder\'on-Zygmund kernel) and
when $\gamma$ is real analytic, $T:L^p\rightarrow L^p$ ($1<p<\infty$)
under no additional assumptions.
This result is a special case of Theorem \ref{ThmRealAnalT} below.
\end{cor}

Unfortunately, when we move to the case $\nu>1$, it is necessary
to put additional hypotheses on $\gamma$ to ensure that $T$ is bounded
on $L^p$ (or even $L^2$), {\it even when $\gamma$ is real analytic}.
Examples demonstrating this can be found in Section \ref{SectionExTransInv}.

To explain our assumptions on $\gamma$, we must first introduce some notation.
In \cite{ChristNagelSteinWaingerSingularAndMaximalRadonTransforms}
it was shown that every $\gamma$ could be written in the form,
\begin{equation*}
\gamma_{t_1,\ldots, t_\nu}\q(x\w) \sim \exp\q(\sum_{\q|\alpha_1\w|+\cdots+\q|\alpha_\nu\w|>0} t_1^{\alpha_1}\cdots t_\nu^{\alpha_\nu} X_{\alpha_1,\ldots, \alpha_\nu} \w)x,
\end{equation*}
where the $X_{\alpha_1,\ldots, \alpha_{\nu}}$ are vector fields
and we write
$\gamma_t\q(x\w)\sim \exp\q(\sum_{\alpha} t^{\alpha} X_\alpha\w)x$ if
$\gamma_t\q(x\w) = \exp\q(\sum_{\q|\alpha\w|<L} t^{\alpha} X_\alpha\w)x+O\q(\q|t\w|^L\w)$.

To each $X_{\alpha_1,\ldots, \alpha_\nu}$ we assign the formal degree
$d_{\alpha_1,\ldots, \alpha_\nu}= \q(\q|\alpha_1\w|,\ldots, \q|\alpha_\nu\w|\w)$.
We say that $d_{\alpha}$ is a {\it pure power} if $d_{\alpha}$
is nonzero in precisely one component, otherwise we say it is a {\it non-pure power}.
Define 
two sets
$$\sP=\q\{\q(X_\alpha, d_\alpha\w): d_\alpha\text{ is a pure power}\w\},$$
$$\sN=\q\{\q(X_\alpha, d_\alpha\w): d_\alpha\text{ is a non-pure power}\w\},$$
Let $\sS$ be the smallest set of vector fields with formal degrees such that,
\begin{itemize}
\item $\sP\subseteq \sS$,
\item If $\q(X_1,d_1\w),\q(X_2,d_2\w)\in \sS$, then $\q(\q[X_1,X_2\w],d_1+d_2\w)\in \sS$.
\end{itemize}

\begin{thm}[\cite{StreetMultiParameterSingRadonAnal}]\label{ThmRealAnalT}
Let $\gamma$ be real analytic.  Suppose that, for every $\q(Y,e\w)\in \sN$,
there is a neighborhood $U=U\q(Y,e\w)$ of $0$ 
and vector fields
$\q(X_1,d_1\w),\ldots, \q(X_q,d_q\w)\in \sS$
such that for every $\delta\in \q[0,1\w]^\nu$,
we have\footnote{We write $\delta^{d_j}=\prod_{\mu=1}^\nu \delta_{\mu}^{d_j^{\mu}}$.}
\begin{equation*}
\delta^e Y\big|_U=\sum_{j=1}^q c_j^\delta \delta^{d_j} X_j\big|_U,
\end{equation*}
where $\q\{c_j^\delta:\delta\in \q[0,1\w]^\nu\w\}\subset C^\infty\q(U\w)$
is a bounded set.
Then $T$ is bounded on $L^p$ ($1<p<\infty$) for every product
kernel $K$ (with sufficiently small support), provided
$\psi$ has sufficiently small support.
\end{thm}

\begin{rmk}\label{RmkMoreGeneralA}
Actually, the result in \cite{StreetMultiParameterSingRadonAnal}
requires less than is outlined in Theorem \ref{ThmRealAnalT}:
the set $U$ can depend on $\delta$ in a specific way,
and one requires less than $\q\{c_j^\delta\w\}$ forming a bounded set.
However, while this stronger result is important for applications,
it requires a good deal of notation to state precisely.  We, therefore,
defer the discussion to \cite{StreetMultiParameterSingRadonAnal}.
\end{rmk}

\begin{rmk}
Note that if $Y=0$ for every $\q(Y,e\w)\in \sN$, then $T$ is bounded on
$L^p$, as the conditions of Theorem \ref{ThmRealAnalT} hold trivially.
In particular, in the case $\nu=1$, $\sN=\emptyset$ and 
Corollary \ref{CorRealAnalT} follows.
\end{rmk}

\begin{rmk}
One of the main aspects of \cite{StreetMultiParameterSingRadonAnal}
is showing that Theorem \ref{ThmRealAnalT} (and its more general
analog alluded to in Remark \ref{RmkMoreGeneralA})
are a special case of Theorem \ref{ThmMainThmFirstPass}, below (at
least when $p=2$; more general $p$ are studied in \cite{SteinStreetMultiParameterSingRadonLp}).
\end{rmk}

\section{A Special Case}\label{SectionSpecialCase}
Before we enter into the rather lengthy statement of our main result,
we discuss a special case.  Namely, the case when,
\begin{equation}\label{EqnSpecialCaseGamma}
\gamma_t\q( x\w) = e^{\sum_{0<\q|\alpha\w|\leq M} t^{\alpha} X_\alpha} x,
\end{equation}
where each $X_\alpha$ is a $C^\infty$ vector field--note that this
is an exponential of a {\it finite} sum of vector fields.
Later in the paper, we will deal with more general $\gamma$;
see Remark \ref{RmkMoreGeneralGamma} for some comments on how we will
do this.
The goal in this section is to discuss conditions we may place
on the $X_\alpha$ so that the operator given by
\eqref{EqnIntroOp} is bounded on $L^2$.  The conditions
will depend on the type of kernel $K$ which is used
in \eqref{EqnIntroOp}.

We begin by discussing the work in 
\cite{ChristNagelSteinWaingerSingularAndMaximalRadonTransforms}
in this special case.  Then, we discuss a way in which
our results generalize the results in
\cite{ChristNagelSteinWaingerSingularAndMaximalRadonTransforms}
in the single-parameter setting.  We then close by informally
discussing a special case of the multi-parameter setting.

In \cite{ChristNagelSteinWaingerSingularAndMaximalRadonTransforms}, operators
of the form,
\begin{equation}\label{EqnSpecilCaseCNSW}
f\mapsto \psi\q( x\w) \int f\q( \gamma_t\q( x\w) \w) K\q( t\w) \: dt,
\end{equation}
were studied,\footnote{\cite{ChristNagelSteinWaingerSingularAndMaximalRadonTransforms} did not restrict attention to $\gamma$ of the form \eqref{EqnSpecialCaseGamma}; nor will we.  However, it is easier to understand the basic
ideas in this special case.  
For the connection between \cite{ChristNagelSteinWaingerSingularAndMaximalRadonTransforms}
and ours when
 $\gamma$ is more general,
see Section \ref{SectionEXCNSW} where a there is a
proof of how the framework of \cite{ChristNagelSteinWaingerSingularAndMaximalRadonTransforms} fits into our more general setting.}
where $\psi$ is an appropriate cut-off functions,
and $K$ is a Calder\'on-Zygmund kernel with small support
(how small depends on $\gamma$).
One of the main aspects of \cite{ChristNagelSteinWaingerSingularAndMaximalRadonTransforms}
was exhibiting conditions on $\gamma$ for which the operator
given by \eqref{EqnSpecilCaseCNSW} is bounded on $L^p$ ($1<p<\infty$).
In the special case when $\gamma$
is given by \eqref{EqnSpecialCaseGamma}, this condition
is that $\q\{ X_\alpha\w\}$ satisfy H\"ormander's condition.
I.e., that the Lie algebra generated by the vector fields $\q\{X_\alpha\w\}$
span the tangent space at every point of the support of $\psi$.

Now suppose we are no longer in the situation where $\q\{X_\alpha\w\}$
satisfy H\"ormander's condition.  Instead, let $\sD$ be the
involutive distribution generated by the $X_\alpha$.
I.e., let $\sD$ be the $C^\infty$ module generated by the $X_\alpha$
and the commutators of the $X_\alpha$ of all orders.
Suppose that $\sD$ is finitely generated\footnote{Usually, one
sees the condition that $\sD$ is {\it locally} finitely generated.  However,
our results are local in nature, and so it suffices to assume that $\sD$
is finitely generated.}
 as a $C^\infty$ module.\footnote{This is a generalization of H\"ormander's
condition.  Indeed, H\"ormander's condition states that the involutive
distribution generated by the $X_\alpha$ is the entire space of vector fields
(on a neighborhood of the support of $\psi$).  This distribution
is clearly finitely generated:  it is generated by the coordinate
vector fields.}
Note, if $\sD_x$ is the vector space given by evaluating all of
the vector fields in $\sD$ at the point $x$, then we are
not assuming $\dim \sD_x$ is constant in $x$.
The classical Frobenius theorem applies to show that $\R^n$ is locally
foliated into leaves, where the space of smooth sections of the
tangent bundle of each leaf is given by $\sD$.
Since the dimension of $\dim \sD_x$ is not necessarily
constant, this could be a singular foliation; i.e., the dimension
of the leaf may vary from point to point.  
It is clear that,
when restricted to each leaf, $\q\{ X_\alpha\w\}$ satisfy H\"ormander's
condition.

One might hope, given the above discussion, that one could apply
the theory
of \cite{ChristNagelSteinWaingerSingularAndMaximalRadonTransforms}
to each leaf, and then put this all together to obtain
the $L^p$ (or, in our case, $L^2$) boundedness of
\eqref{EqnSpecilCaseCNSW}.  Unfortunately, one
runs into
a technical difficulty:  near a singular\footnote{A point $x$ is said
to be singular, if $\dim \sD_x$ is not constant on any neighborhood of $x$.} 
point of the involutive distribution $\sD$, the classical proofs
of the Frobenius theorem do not yield uniform
enough control of the coordinate charts defining the leaves to be able
to apply the theory of \cite{ChristNagelSteinWaingerSingularAndMaximalRadonTransforms}
uniformly on each leaf.

Fortunately, one {\it can} obtain uniform control of these coordinate charts.
Indeed, this was one of the main aspects of
\cite{StreetMultiParameterCCBalls}.
In fact, the theory in \cite{StreetMultiParameterCCBalls}
will be the main technical tool on which we will base all of
our ``scaling'' arguments (see Sections \ref{SectionCC} and \ref{SectionCCII} for
more on \cite{StreetMultiParameterCCBalls}).\footnote{In particular, the results
in \cite{StreetMultiParameterCCBalls} allow us to use multi-parameter Carnot-Carth\'eodory
balls, without resorting to the weakly-comparable hypotheses in Section 4 of
\cite{TaoWrightLpImprovingBoundsForAverages}, by introducing
some extra assumptions on the relevant vector fields.  This
is an essential point of our analysis, and might also be useful in
problems related to the ones in \cite{TaoWrightLpImprovingBoundsForAverages}.}
Because of this, in the above
case, we will obtain the $L^2$ boundedness of \eqref{EqnSpecilCaseCNSW}
in this paper.\footnote{The corresponding $L^p$ ($1<p<\infty$) boundedness
holds as well, and will be covered in the sequel to this
paper \cite{SteinStreetMultiParameterSingRadonLp}.}

\begin{rmk}
In the special case when the $X_\alpha$ are real analytic,
the distribution $\sD$ is automatically finitely generated as
a $C^\infty$ module (when restricted to a sufficiently small neighborhood).\footnote{This is a result of Lobry \cite{LobryControlabiliteDesSystemesNonLinearies},
and can be seen as a consequence of the Weierstrass preparation theorem.
See \cite{LobryControlabiliteDesSystemesNonLinearies,StreetMultiParameterSingRadonAnal} for a further discussion.}
This gives intuition as to why Corollary \ref{CorRealAnalT} is true.
\end{rmk}

Let us rephrase the above single-parameter theory in a slightly
more complicated way, which will
lead us naturally to the multi-parameter theory.
In what follows, we take the standard dilations on $t=\q( t_1,\ldots, t_N\w)\in \R^N$:
\begin{equation*}
\delta t = \q( \delta t_1,\ldots, \delta t_N\w), \quad \delta\in \q(0,1\w].
\end{equation*}
In later sections of the paper, we will raise $\delta$ to different
powers in each coordinate, but we ignore such generalizations for the moment.
One of the main aspects of Calder\'on-Zygmund kernels is that the class
of kernels is dilation invariant:
\begin{equation*}
\delta^{-N} K\q( \delta^{-1} t\w)
\end{equation*}
is again a Calder\'on-Zygmund kernel if $K$ is (and this is true uniformly in
$\delta$).
To take advantage of this dilation invariance, we wish to restate the
above condition that the involutive distribution, $\sD$, generated
by $\q\{ X_\alpha\w\}$ is finitely
generated as a $C^\infty$ module, in a naturally dilation invariant way.

Note that,
\begin{equation*}
\gamma_{\delta t}\q( x\w) = e^{\sum_{0<\q|\alpha\w|\leq M} t^{\alpha} \delta^{\q|\alpha\w|} X_\alpha}x.
\end{equation*}
Our dilation invariant version of the above assumption is that the
involutive distribution generated by
$\q\{\delta^{\q|\alpha\w|} X_\alpha\w\}$ is finitely
generated ``uniformly'' for $\delta\in\q(0,1\w]$ (of
course for every $\delta$ the involutive distribution so generated is always
equal to $\sD$, so it is the uniform aspect that is the point).
We make this notion of uniform precise in a moment.
We will see, in the single-parameter case (which we are presently considering),
this uniformity follows for free (Lemma \ref{LemmaSpecialCaseEquivFiniteGen}), but this is not the case
in the multi-parameter setting (see Section \ref{SectionMultiParamDifficult}).

Assign to each vector field $X_\alpha$ the formal degree $\q|\alpha\w|$.
Recursively, assign formal degrees as follows:
if $X_1$ has formal degree $d_1$ and $X_2$ has formal degree $d_2$,
then we assign $\q[X_1,X_2\w]$ the formal degree $d_1+d_2$ (it is possible
that the same vector field may have more than one formal degree).
Scaling $X_\alpha$ by $\delta^{\q|\alpha\w|}$ induces a
scaling $\delta^{d_1}X_1$ where $X_1$ has formal degree $d_1$.

Our ``uniform'' assumption in $\delta$ is as follows.  We assume
that there is a finite list of the above vector fields
$X_1,\ldots, X_q$, $X_j$ having formal degree $d_j$,
each $\q( X_\alpha, \q|\alpha\w|\w)$ appearing as some $\q( X_j,d_j\w)$,
and this list satisfying,
\begin{equation}\label{EqnSpecialCaseInteg}
\q[ \delta^{d_j} X_j, \delta^{d_k} X_k\w] = \sum_{l=1}^q c_{j,k}^{l,\delta} \delta^{d_l} X_l, \quad \delta\in \q(0,1\w],
\end{equation}
where $c_{j,k}^{l,\delta}$ is a $C^\infty$ function\footnote{When
we present our assumptions on the vector fields in more detail
(see Section \ref{SectionCC}) we will only require \eqref{EqnSpecialCaseInteg}
to hold locally near each point in a precise sense, and will
require less than the functions being $C^{\infty}$.  We defer
such details to Section \ref{SectionCurves}.}
uniformly in $\delta$.  Note that, by taking $\delta=1$, \eqref{EqnSpecialCaseInteg}
implies that $\sD$ is a finitely generated distribution (it is generated
by $X_1,\ldots, X_q$ as a $C^\infty$ module).  We also have the converse,
\begin{lemma}\label{LemmaSpecialCaseEquivFiniteGen}
One may choose $\q(X_1,d_1\w),\ldots, \q(X_q,d_q\w)$ as above, so that \eqref{EqnSpecialCaseInteg}
holds, if and only if $\sD$ is finitely generated as a $C^\infty$ module.
\end{lemma}
\begin{proof}
As noted above, the only if part is clear.  To prove the converse,
we will prove a stronger form of \eqref{EqnSpecialCaseInteg}.
Indeed, we will show that we may select
$\q( X_1,d_1\w),\ldots, \q( X_q,d_q\w)$ as above so that,
\begin{equation}\label{EqnSpecialCaseCNSWInteg}
\q[ X_j, X_k\w] = \sum_{d_{l}\leq d_j+d_k} c_{j,k}^l X_l,
\end{equation}
where $c_{j,k}^l$ are $C^\infty$ functions.
Note that \eqref{EqnSpecialCaseCNSWInteg} implies \eqref{EqnSpecialCaseInteg},
as one may take $c_{j,k}^{l,\delta}= \delta^{d_j+d_k-d_l} c_{j,k}^l$,
when $d_j+d_k\geq d_l$, and $0$ otherwise.
Incidentally, \eqref{EqnSpecialCaseCNSWInteg} was first introduced
in \cite{NagelSteinWaingerBallsAndMetricsDefinedByVectorFields}
to study vector fields which satisfy H\"ormander's condition.
Suppose $X_1,\ldots, X_r$ are generators for $\sD$
as a $C^{\infty}$ module.  We may assume that each $X_j$
arises as an iterated commutator of the $X_\alpha$.
Thus, each $X_j$, $j=1,\ldots, r$, has assigned to it
a formal degree $d_j$.  We may also assume that each
$X_\alpha$ appears in the list $X_1,\ldots, X_r$.
We let $\q( X_1,d_1\w),\ldots, \q( X_q,d_q\w)$ be an enumeration
of all vector fields with formal degree $d_k$ such that
$d_k\leq \max_{1\leq l\leq r} d_l$ (here we have taken
$X_1,\ldots, X_r$ to be the first $r$ elements of $X_1,\ldots, X_q$).
We claim that this list: $\q( X_1,d_1\w),\ldots, \q( X_q,d_q\w)$
satisfies \eqref{EqnSpecialCaseCNSWInteg}.  Indeed, if we take
the commutator $\q[X_j,X_k\w]$ ($1\leq j,k\leq q$), there
are two possibilities.  The first possibility is that
$d_j+d_k\leq \max_{1\leq l\leq r} d_l$.  In this case
$\q( \q[X_j,X_k\w], d_j+d_k\w)$ already appears in the list
$\q( X_1,d_1\w),\ldots, \q( X_q,d_q\w)$ and so \eqref{EqnSpecialCaseCNSWInteg}
is trivial.
On the other hand, if $d_j+d_k> \max_{1\leq l \leq r} d_l$, then we use the fact
that,
\begin{equation*}
\q[X_j,X_k\w] = \sum_{l=1}^r c_{j,k}^l X_l,
\end{equation*}
as this immediately follows from the fact that $X_1,\ldots, X_r$
generate $\sD$ as a $C^\infty$ module.
Now \eqref{EqnSpecialCaseCNSWInteg} follows immediately.
\end{proof}

Now that we have stated our single-parameter assumptions
in a dilation invariant manner, we are prepared to
state a multi-parameter result.
For the purposes of this introduction, we restrict our
attention to the two-parameter situation, but
we will see later in the paper that all of these
ideas extend to any number of parameters.

We will be considering operators of the form
\begin{equation}\label{EqnSpecialCaseMulti}
f\mapsto \psi\q( x\w) \int f\q( \gamma_{\q( s,t\w)}\q( x\w)\w) K\q( s,t\w)\: ds\: dt,
\end{equation}
where $\q( s,t\w) \in \R^{N_1}\times \R^{N_2}$.  Here $K\q( s,t\w)$ is a
product kernel\footnote{For our main theorem, we will discuss kernels
more general than product kernels.  See Sections \ref{SectionKernels}
and \ref{SectionKernelsII}.} (supported near $s=t=0$).
That is, $K$ satisfies,
\begin{equation*}
\q|\partial_s^{\alpha} \partial_t^{\beta} K\q( s,t\w)\w| \lesssim \q|s\w|^{-N_1-\q|\alpha\w|}\q|t\w|^{-N_2-\q|\beta\w|},
\end{equation*}
along with certain cancellation conditions.  See Section \ref{SectionKernelsII}
for more precise details.  The relevant dilations in this situation
are two-parameter.  For $\q( \delta_1,\delta_2\w) \in \q[0,1\w]^2$,
we define,
\begin{equation*}
\q( \delta_1, \delta_2\w) \q( s,t\w) = \q( \delta_1 s, \delta_2 t\w).
\end{equation*}

We consider $\gamma$ of the form,
\begin{equation*}
\gamma_{\q( s,t\w)}\q( x\w) = e^{\sum_{0<\q|\alpha\w|+\q|\beta\w|\leq M} s^{\alpha} t^{\beta} X_{\alpha,\beta}} x,
\end{equation*}
where, as before, the $X_{\alpha,\beta}$ are $C^\infty$ vector fields.
Dilating gives,
\begin{equation*}
\gamma_{\q( \delta_1 s, \delta_2 t\w)} \q( x\w) = e^{\sum_{0<\q|\alpha\w|+\q|\beta\w|\leq M} \delta_1^{\q|\alpha\w|}\delta_2^{\q|\beta\w|} s^{\alpha} t^{\beta} X_{\alpha,\beta} } x.
\end{equation*}
In analogy with the single parameter case, we assign to the vector field
$X_{\alpha,\beta}$ the two-parameter formal degree $d_{\alpha,\beta}=\q( \q|\alpha\w|,\q|\beta\w|\w)$.  Thus, scaling $X_{\alpha, \beta}$ by $\delta_1^{\q|\alpha\w|}\delta_2^{\q|\beta\w|}$ is the same as scaling it by $\q( \delta_1,\delta_2\w)^{d_{\alpha,\beta}}$, where we think of $d_{\alpha,\beta}$ as a 
multi-index.

In the two-parameter situation, a new phenomenon appears.  We must separate
our vector fields into two sets and view these two sets differently.
To define this, we call $d_{\alpha,\beta}$ a {\it pure power} if either
$\q|\alpha\w|$ or $\q|\beta\w|$ is $0$.  Otherwise we call it
a non-pure power.  We define two sets,
\begin{equation*}
\begin{split}
\sP&=\q\{\q(X_{\alpha,\beta},d_{\alpha,\beta} \w): d_{\alpha,\beta}\text{ is a pure power}\w\},\\
\sN&=\q\{\q(X_{\alpha,\beta},d_{\alpha,\beta} \w): d_{\alpha,\beta}\text{ is a non-pure power}\w\}.
\end{split}
\end{equation*}
In the single-parameter case, every power was a pure power, but in the
multi-parameter case, we must deal with the set $\sN$ in a different manner.

First we discuss the case when $X_{\alpha,\beta}=0$ for every $\q(X_{\alpha,\beta},d_{\alpha,\beta}\w)\in\sN$.  This case works in
much the same way as the single-parameter case.  We recursively
define formal degrees
on the iterated commutators of the elements of $\sP$ as follows.
If $X_1$ has formal degree $d_1\in \N^2$ and $X_2$ has formal
degree $d_2\in \N^2$, we assign to $\q[X_1,X_2\w]$ the formal degree
$d_1+d_2$.  Our assumption is that there is a finite
list
$\q( X_1,d_1\w),\ldots, \q( X_q,d_q\w)$ of these vector fields,
containing the set $\sP$, such that for every
$\delta\in \q[0,1\w]^2$, we can write,
\begin{equation}\label{EqnSpecialCaseMultiInteg}
\q[\delta^{d_j}X_j,\delta^{d_k}X_k\w]=\sum_{l=1}^q c_{j,k}^{l,\delta} \delta^{d_l} X_l,
\end{equation}
where $c_{j,k}^{l,\delta}\in C^\infty$ uniformly\footnote{Again,
we have just stated a special case of our assumptions
on
$c_{j,k}^{l,\delta}$ for simplicity.}
in $\delta$.
Note, by taking $\delta_2=0$, \eqref{EqnSpecialCaseMultiInteg} implies
that the vector fields $\q\{X_{\alpha,0}\w\}$ satisfy
the single-parameter assumptions.  Similarly, by
taking $\delta_1=0$, \eqref{EqnSpecialCaseMultiInteg} implies that
the vector fields $\q\{X_{0,\beta}\w\}$ also satisfy the single
parameter assumptions.  Moreover, \eqref{EqnSpecialCaseMultiInteg}
implies more:  it implies a condition on the commutators between
the $\q\{X_{\alpha,0}\w\}$ and the $\q\{ X_{0,\beta}\w\}$.

More generally, when there is some nonzero vector field in $\sN$,
we proceed as above, obtaining
the list $\q( X_1,d_1\w) ,\ldots, \q( X_q,d_q\w)$ from
$\sP$ and assuming it satisfies \eqref{EqnSpecialCaseMultiInteg}.
We further assume that for every $\q( Y,e\w)\in \sN$
and every $\delta\in \q[0,1\w]^2$, we have,
\begin{equation}\label{EqnSpecialCaseNonPureControl}
\delta^e Y=\sum_{l=1}^q c_Y^{l,\delta} \delta^{d_l} X_l,
\end{equation}
with $c_Y^{l,\delta}\in C^\infty$ uniformly in $\delta$.
See Section \ref{SectionExTransInv} for an example of how
$L^2$ boundedness can fail if we do not have
\eqref{EqnSpecialCaseNonPureControl} and Remarks \ref{RmkCommentsPurePowers} and \ref{RmkPureIsUsed} for how
\eqref{EqnSpecialCaseNonPureControl} is used in the proof.

Our main theorem is a generalization of the following:  under the
above hypotheses, the operator given by
\eqref{EqnSpecialCaseMulti} is bounded on $L^2$.

\begin{rmk}
It may not be obvious that our main theorem generalizes the above.
The fact that it does is discussed in Section \ref{SectionExExpOfVect}.
\end{rmk}

\begin{rmk}
As mentioned before, the obvious analog of Lemma \ref{LemmaSpecialCaseEquivFiniteGen}
in the multi-parameter setting does not hold.  This is
discussed in Section \ref{SectionMultiParamDifficult}.
\end{rmk}

\begin{rmk}
Our assumptions above on the set $\sP$ (i.e., the existence of $\q(X_1,d_1\w),\ldots, \q(X_q,d_q\w)$) can be thought of as the ``finite type''
assumptions discussed in Section \ref{SectionOverview}.
The assumptions on the set $\sN$ can be thought of as the
``algebraic'' assumptions.
As mentioned in Section \ref{SectionOverview}, the ``finite type''
assumptions hold automatically when the $X_\alpha$
are real analytic.  This is proven in
\cite{StreetMultiParameterSingRadonAnal}.
\end{rmk}

\begin{rmk}\label{RmkMoreGeneralGamma}
In the rest of the paper, we work with more general functions
$\gamma_t\q(x\w)$, not necessarily assuming
that $\gamma_t\q(x\w)=e^{\sum t^{\alpha} X_\alpha}x$.
If one were to assume that
$\gamma_t\q(x\w) = e^{A\q(t\w)}x$ for some $C^\infty$ vector field
depending smoothly on $t$ with $A\q(0\w)=0$, then it would not be difficult
the generalize the above assumptions.  Unfortunately, not every $\gamma$
is of this form.  Fortunately, there is a convenient alternative.
Since $\gamma_0\q(x\w)\equiv x$, $\gamma_t$ is a diffeomorphism
onto its image, for $t$ sufficiently small.  Thus, we may define,
\begin{equation*}
W\q(t,x\w)=\frac{d}{d\epsilon}\bigg|_{\epsilon=1} \gamma_{\epsilon t}\circ \gamma_t^{-1}\q(x\w)\in T_x\R^n.
\end{equation*}
Note that $W\q(t\w)$ is a vector field, depending smoothly on $t$,
satisfying $W\q(0\w)=0$.  Moreover, it turns out that the map
$\gamma\mapsto W$ is a {\it bijection} (see Proposition \ref{PropGetGammaFromW}).
The inverse mapping is given by solving an ODE.  This ODE is similar
to the one which is used to define the exponential map.  We will see
that everywhere we wish to use the exponential map, it can be replaced
by the inverse mapping $W\mapsto \gamma$.  Thus, in what follows,
we will exhibit analogs of the above conditions applied
to the vector field $W\q(t\w)$.
\end{rmk}


\section{Kernels}\label{SectionKernels}
In this section, we define the kernels $K$ for which we will
study operators of the form \eqref{EqnIntroOp}.
The kernels we study will be distributions on $\R^N$, and they
will be supported in $\Q^N\q( a\w)=\q\{x\in \R^N: \q|x\w|<a\w\}$, where $a>0$ is a small
number to be chosen later.\footnote{For one thing, we will
always choose $a>0$ so small that for $\q|t\w|\leq a$, $\gamma_t$
is a diffeomorphism onto its image; so that $\gamma_t^{-1}$
makes sense.}
Fix $\nu \in \N$.  We will be studying $\nu$-parameter
operators.
Fix a subset $\sA\subseteq \q[0,1\w]^\nu$ such that
if $\delta_1,\delta_2\in \sA$, then $\delta_1\vee \delta_2\in \sA$.

We suppose that we are given $\nu$-parameter dilations on $\R^N$.
That is, we are given $e=\q( e_1,\ldots, e_N\w)$ with each
$0\ne e_j=\q(e_j^1,\ldots, e_j^\nu\w)\in \q[0,\infty\w)^\nu$.  For $\delta\in \q[0,\infty\w)^\nu$
and $t=\q( t_1,\ldots, t_N\w)\in \R^N$,
we define
\begin{equation}\label{EqnDefndeltat}
\delta t = \q(\delta^{e_1}t_1,\ldots, \delta^{e_N}t_N \w),
\end{equation}
thereby obtaining $\nu$ parameter dilations on $\R^N$.

The class of distributions we will define will depend on
$N$, $a$, $\sA$, and $e$.  At a first reading, the reader may
wish to take $\sA=\q[0,1\w]^\nu$, as this case contains all the
main ideas.

We will make two passes at the definition of the class of kernels.
First, we will define a class of kernels which we will denote by
$\sKt=\sKt\q(N,e,a,\sA \w)$.  This class will be simple to understand,
and contains all of the main ideas.  Our proof, however, will
works for a slightly larger class of kernels $\sK\q( N,e,a,\sA\w)$.
This class of kernels is important for applications, and so
we state our results in this context.  On a first pass,
the reader may wish to just understand the proof
in the case of $\sKt$, however.

\begin{rmk}
Actually, the class $\sK$ is much simpler to understand in the
case when $\sA=\q[0,1\w]^\nu$.  See Lemma \ref{LemmaCancelInFullCase} for details.
\end{rmk}

	\subsection{The class $\sKt$}\label{SectionClasssKt}
	For reach $\mu$, $1\leq \mu\leq \nu$, and given $t=\q( t_1,\ldots, t_N\w)\in \R^N$, let
$t^\mu$ denote those coordinates $t_j$ of $t$ such that $e_j^\mu\ne 0$. 

Given a function $\vsig$ on $\R^N$, and $j\in \N^\nu$, define
\begin{equation*}
\dil{\vsig}{2^j}\q( t\w) = 2^{j\cdot e_1+ \cdots + j\cdot e_N} \vsig\q( 2^jt\w).
\end{equation*}
Note that $\dil{\vsig}{2^j}$ is defined in such a way that
\begin{equation*}
\int \dil{\vsig}{2^j} \q( t\w) \: dt= \int \vsig\q( t\w) \: dt.
\end{equation*}
Let $\lA=\q\{j\in \N^\nu:2^{-j}\in \sA\w\}$.  Note that if
$\sA=\q[0,1\w]^\nu$, then $\lA=\N^\nu$.

\begin{defn}
We define $\sKt=\sKt\q( N,e,a,\sA\w)$ to be the set of all distributions, $K$,
of the form
\begin{equation}\label{EqnDefnOfsKt}
K=\sum_{j\in \lA} \dil{\vsig_j}{2^j}
\end{equation}
where $\q\{\vsig_j\w\}\subset C_0^\infty\q( \Q^N\q( a\w)\w)$ is a bounded set,
satisfying
\begin{equation}\label{EqnCancelsKt}
\int \vsig_j\q( t\w) \: dt^\mu =0.
\end{equation}
The convergence in \eqref{EqnDefnOfsKt} is taken in the sense
of distributions.  We will see in Lemma \ref{LemmaSumsConvg} that every such sum converges
in distribution.
\end{defn}

\begin{rmk}\label{RmkBiggerA}
Given $N$, $e$, and $\sA$, we will define a class of functions $\gamma_t$
such that the operator given by \eqref{EqnIntroOp} is bounded on
$L^2$ for every $K\in \sKt\q( N,e,a,\sA\w)$ for $a>0$ sufficiently small.
Note that if $\sA\subseteq \sB$, then $\sKt\q( N,e,a,\sA\w)\subseteq \sKt\q(N,e,a,\sB\w)$.
However, our theorem will apply to a larger class of $\gamma$ corresponding
to $\sA$ than it will corresponding to $\sB$.
\end{rmk}

	\subsection{The class $\sK$}\label{SectionClasssK}
	While the class $\sKt$ contains all of the main ideas, in many applications
the kernels that arise are in a slightly larger class, which we denote
$\sK\q( N,e,a,\sA\w)$.  Indeed, we will see the benefits of working
with this class in Section \ref{SectionKernelsII}.
The main point, here, is that in our proofs we do not use the full
cancellation condition \eqref{EqnCancelsKt}.
Thus, the definition of the class $\sK$ will be the same
as that of $\sKt$, but we will weaken the condition \eqref{EqnCancelsKt}.

Before we can define $\sK$, we need some new notation.
Fix a constant $C\geq 1$, and $j=\q( j_1,\ldots, j_\nu\w)\in \lA$.
For $1\leq \mu\leq \nu$, we say $\mu$ is $j,\sA$-minimal, if there does 
not exist a $k\in \lA$ such that $j_\mu>k_\mu$.
For $1\leq \mu_1,\mu_2\leq \nu$, we write
$$ \mu_1\preceq_{j,C,\sA} \mu_2 $$
if for every $k=\q( k_1,\ldots, k_\nu\w)\in \lA$ with $j_{\mu_1}>k_{\mu_1}$,
we have $j_{\mu_1}-k_{\mu_1}\leq C\q( j_{\mu_2}-k_{\mu_2}\w)$.
We define
\begin{equation*}
\q[\mu_1\w]_{j,C,\sA} = \q\{\mu_2: \mu_1\preceq_{j,C,\sA} \mu_2 \text{ and } \mu_2 \preceq_{j,C,\sA} \mu_1\w\}.
\end{equation*}
Note that $\mu_1\in \q[\mu_1\w]_{j,C,\sA}$.
Also, we define $t_1^{\q[\mu_1\w]_{j,C,\sA}} = \q(t^\mu\w)_{\mu\in \q[\mu_1\w]_{j,C,\sA}}$, where repeated coordinates are only included once.
I.e., $t_1^{\q[\mu_1\w]_{j,C,\sA}}$ consists of those coordinates $t_j$ such
that $e^\mu_j\ne 0$ for some $\mu\in \q[\mu_1\w]_{j,C,\sA}$.
We let $t_2^{\q[\mu_1\w]_{j,C,\sA}}$ be the rest of the coordinates, so that
$t=\q(t_1^{\q[\mu_1\w]_{j,C,\sA}},t_2^{\q[\mu_1\w]_{j,C,\sA}} \w)$.
Finally, we say $\mu_1\prec_{j,C,\sA} \mu_2$ if $\mu_1\preceq_{j,C,\sA} \mu_2$
but we do not have $\mu_2\preceq_{j,C,\sA} \mu_1$.

\begin{defn}\label{DefnsK}
We define $\sK=\sK\q( N,e,a,\sA\w)$ to be the set of all distributions, $K$,
such that there exists a $C\geq 1$ such that $K$ can be written in the form
\begin{equation}\label{EqnDefsK}
K=\sum_{j\in \lA} \dil{\vsig_j}{2^j},
\end{equation}
where $\q\{\vsig_j\w\}\subset C_0^\infty\q( \Q^N\q( a\w)\w)$ is a bounded set,
satisfying the following cancellation conditions:  for every $j\in \lA$,
and every $\mu$ such that $\mu$ is not $j,\sA$-minimal, and there does
not exist a $\mu'$ with $\mu\prec_{j,C,\sA} \mu'$, then we assume
\begin{equation}\label{EqnDefnsKCancel}
\int \vsig_j\q( t\w) \: dt_1^{\q[\mu\w]_{j,C,\sA}} =0;
\end{equation}
otherwise, we allow the integral to be possibly non-zero.
\end{defn}

\begin{rmk}
Note that it is evident that $\sKt\q( N,e,a,\sA\w)\subseteq \sK\q( N,e,a,\sA\w)$.
\end{rmk}

In Section \ref{SectionKernelsII} we further discuss the 
class $\sK$, though the results there will not be used
in the rest of the paper.  In particular, we discuss
following facts:
\begin{enumerate}
\item Every sum of the form \eqref{EqnDefsK} converges in the sense of
distributions.
\item $\delta_0\in \sK\q( N,e,a,\q[0,1\w]^\nu\w)$.
\item Product kernels (in particular Calder\'on-Zygmund kernels) and flag kernels (see Definitions \ref{DefnProdKer} and \ref{DefnFlagKer} and 
\cite{NagelRicciSteinSingularIntegralsWithFlagKernels}) are special
cases of $\sK$.  Indeed, product kernels (with support in $\Q^N\q(a\w)$) are the case
when each $e_j$ is nonzero in precisely one component, and $\sA=\q[0,1\w]^\nu$
(this is a result of \cite{NagelRicciSteinSingularIntegralsWithFlagKernels}).
Similarly, flag kernels are the kernels with the same $e$ but with $\sA=\q\{\delta\in \q[0,1\w]^\nu: \delta_1\geq \delta_2\geq\cdots\geq \delta_\nu\w\}$.
Note that, as in Remark \ref{RmkBiggerA}, our main theorem will apply
to a larger class of $\gamma$ when considering $K$ a flag kernel,
than it will when considering $K$ a product kernel.  See
Section \ref{SectionFlagVsProduct} for more details on this.
\item There are interesting kernels which are neither
product kernels nor flag kernels which are of the
form $\sK\q( N,e,a,\sA\w)$--see the end of Section \ref{SectionKernelsII}
for an example.
\end{enumerate}

\section{Multi-parameter Carnot-Carath\'eodory geometry}\label{SectionCC}
At the heart of the 
conditions we will assume on $\gamma$
lies multi-parameter Carnot-Carath\'eodory
geometry.  Thus, before we can even define the class
of $\gamma$ we will study, it is necessary to review the relevant
definitions of multi-parameter Carnot-Carath\'eodory balls.
We defer the theorems we will use to deal with these balls
to Section \ref{SectionCCII}.
Our main reference for Carnot-Carath\'eodory geometry is \cite{StreetMultiParameterCCBalls},
and we refer the reader there for a more detailed discussion.

Let $\Omega\subseteq \R^n$ be a fixed open set, and suppose
$X_1,\ldots, X_q$ are $C^\infty$ vector fields on $\Omega$.
We define the Carnot-Carath\'eodory ball of unit radius, centered
at $x_0\in \Omega$, with respect to the list $X$ by
\begin{equation*}
\begin{split}
B_X\q( x_0\w):=\bigg\{y\in \Omega \:\bigg|\: &\exists \gamma:\q[0,1\w]\rightarrow \Omega, \gamma\q( 0\w) =x_0, \gamma\q( 1\w) =y,  \\
&\gamma'\q( t\w) = \sum_{j=1}^q a_j\q( t\w) X_j\q( \gamma\q( t\w)\w),
 a_j\in L^\infty\q( \q[0,1\w]\w), \\
&\Lppn{\infty}{\q[0,1\w]}{\q(\sum_{1\leq j\leq q} \q|a_j\w|^2\w)^{\frac{1}{2}}}<1\bigg\}.
\end{split}
\end{equation*}
Now that we have the definition of balls with unit radius, we may define
(multi-parameter) balls of any radius merely by scaling the vector fields.
To do so, we assign to each vector field, $X_j$, a (multi-parameter) formal
degree $0\ne d_j=\q(d_j^1,\ldots, d_j^\nu \w)\in \q[0,\infty\w)^\nu$.  For $\delta=\q(\delta_1,\ldots, \delta_\nu \w)\in \q[0,\infty\w)^\nu$,
we define the list of vector fields $\delta X$ to be the list
$\q( \delta^{d_1} X_1,\ldots, \delta^{d_q}X_q\w)$.  Here, $\delta^{d_j}$
is defined by the standard multi-index notation:  $\delta^{d_j} =\prod_{\mu=1}^\nu \delta_\mu^{d_j^\mu}$.
We define the ball of radius $\delta$ centered at $x_0\in \Omega$ by
$$\B{X}{d}{x_0}{\delta} := B_{\delta X}\q( x_0\w). $$

At times, it will be convenient to assume that the ball $\B{X}{d}{x_0}{\delta}$
lies ``inside'' of $\Omega$.  To this end, we make the following
definition.
\begin{defn}
Given $x_0\in \Omega$ and $\Omega'\subseteq \Omega$, we say the list of vector fields $X$ satisfies $\sC\q( x_0,\Omega'\w)$
if for every $a=\q( a_1,\ldots, a_q\w) \in \q( L^{\infty}\q( \q[0,1\w]\w)\w)^q$,
with
$$\Lppn{\infty}{\q[0,1\w]}{\q|a\w|}=\Lppn{\infty}{\q[0,1\w]}{\q(\sum_{j=1}^q \q|a_j\w|\w)^{\frac{1}{2}}}<1,$$
there exists a solution $\gamma:\q[0,1\w]\rightarrow \Omega'$ to the
ODE
\begin{equation*}
\gamma'\q( t\w) =\sum_{j=1}^q a_j\q( t\w) X_j\q( \gamma\q( t\w)\w), \quad \gamma\q( 0\w) =x_0.
\end{equation*}
Note, by Gronwall's inequality, when this solution exists, it is unique.
Similarly, we say $\q( X,d\w)$ satisfies $\sC\q( x_0,\delta,\Omega'\w)$
if $\delta X$ satisfies $\sC\q(x_0,\Omega'\w)$.
\end{defn}

One of the main points of \cite{StreetMultiParameterCCBalls} was to
provide a detailed study of the balls $\B{X}{d}{x_0}{\delta}$, under
appropriate conditions on the list $\q( X,d\w)$.  To do this, we first
need to pick a subset $\sA\subseteq \q[0,1\w]^\nu$, and
a compact set $\K\Subset \Omega$.\footnote{One should think of $\K$
as the closure of a small open set, on which the operator of
study is supported; i.e., the Schwartz kernel of the operator
we are studying will be supported in $\K\times\K$.}  
We will
(essentially) be restricting our attention to those balls
$\B{X}{d}{x_0}{\delta}$ such that $x_0\in \K$ and $\delta\in \sA$ (this $\sA$ is the
{\it same} set as in Section \ref{SectionKernels}).

\begin{defn}\label{DefnsD}
We say $\q( X,d\w)$ satisfies $\sD\q(\K,\sA\w)$ if the following holds:
\begin{itemize}
\item Take
$\Omega'$ with $\K\Subset \Omega'\Subset \Omega$ and
$\xi>0$ such that for every $\delta\in \sA$ and $x\in \K$, $\q( X,d\w)$ satisfies
$\sC\q( x, \xi\delta, \Omega'\w)$.  

\item For every $\delta\in \sA$
and $x\in \K$, we assume
\begin{equation*}
\q[\delta^{d_i} X_i, \delta^{d_j} X_j\w] = \sum_k c_{i,j}^{k,\delta,x} \delta^{d_k} X_k, \text{ on } \B{X}{d}{x}{\xi\delta}.
\end{equation*}

\item For every ordered multi-index $\alpha$ we assume\footnote{We write $\CjN{f}{0}{U}:=\sup_{x\in U}\q|f\q(x\w)\w|$, and if we say the norm is finite,
we mean (in addition) that $f$ is continuous on $U$.}
\begin{equation*}
\sup_{\substack{\delta\in \sA\\ x\in \K } } \CjN{\q(\delta X\w)^\alpha c_{i,j}^{k,x,\delta}}{0}{\B{X}{d}{x}{\xi\delta}}<\infty.
\end{equation*}
\end{itemize}
If we wish to be explicit about $\Omega'$ and $\xi$, we write $\sD\q( \K,\sA, \Omega', \xi\w)$.
\end{defn}

It is under condition $\sD\q( \K, \sA\w)$ that the balls $\B{X}{d}{x}{\delta}$
were studied in \cite{StreetMultiParameterCCBalls}.
We refer the reader to Section \ref{SectionCCII} for an overview of
the theorems from \cite{StreetMultiParameterCCBalls} that
we shall use.

In what follows, we will not be directly given a list of
vector fields with formal degrees satisfying $\sD\q( \K, \sA\w)$,
\begin{equation*}
\q( X_1,d_1\w), \ldots, \q( X_q,d_q\w),
\end{equation*}
but, rather,  we will be given a list of $C^\infty$ vector fields with
formal degrees which we will assume to ``generate'' such a list.

To understand this, let $\q( X_1,d_1\w),\ldots, \q( X_r,d_r\w)$ be
$C^\infty$ vector fields with associated formal degrees
$0\ne d_j\in \q[0,\infty\w)^\nu$.
For a list $L=\q( l_1,\ldots, l_m\w)$ where $1\leq l_j\leq r$, we define,
\begin{equation*}
\begin{split}
X_L &= \ad{X_{l_1}}\ad{X_{l_2}}\cdots \ad{X_{l_{m-1}}} X_{l_m},\\
d_L &= d_{l_1}+d_{l_2}+\cdots +d_{l_m}.
\end{split}
\end{equation*}
We define 
$\sS=\q\{\q( X_L,d_L\w) : L \text{ is any such list}\w\}.$

\begin{defn}\label{DefnGeneratesAFiniteList}
We say $\sS$ is {\it finitely generated}
 or that $\q( X_1,d_1\w), \ldots, \q( X_r,d_r\w)$ {\it generates a finite list} if there exists finite subset,
$\sF\subseteq \sS$, such that $\sF$ satisfies $\sD\q( \K,\sA\w)$\footnote{Here,
we are thinking of $\K$ and $\sA$ fixed.}
 and
\begin{equation*}
\q( X_j, d_j\w)\in \sF, \quad 1\leq j\leq r.
\end{equation*}
If we enumerate the vector fields in $\sF$,
\begin{equation*}
\sF=\q\{\q(X_1,d_1\w),\ldots, \q( X_q,d_q\w) \w\},
\end{equation*}
we say that $\q( X_1,d_1\w),\ldots, \q( X_r,d_r\w)$ {\it generates
the finite list} $\q( X_1,d_1\w) ,\ldots, \q( X_q,d_q\w)$.
Note that, if $\sS$ is finitely generated, $\q( X_1,d_1\w),\ldots \q( X_r,d_r\w)$ could generate
many different finite lists.
However, if we let $\q( X,d\w)$ and $\q( X',d'\w)$ be two different such lists
then
either choice will work for our purposes.
In fact, it is easy to see that $\q( X,d\w)$ and $\q( X',d'\w)$ are
{\it equivalent} in the sense discussed in Section \ref{SectionControlOfVFs}. 
It follows from the discussion in Section \ref{SectionControlOfVFs}, in every place we use these notions, it will not make
a difference which finite list we use.
Thus, we will unambiguously
say ``$\q( X_1,d_1\w) ,\ldots, \q(X_r,d_r\w)$ generates
the finite list $\q( X_1,d_1\w) ,\ldots, \q( X_q,d_q\w)$,''
to mean that $\q( X_1,d_1\w), \ldots, \q(X_r,d_r\w)$ generates
a finite list and $\q( X_1,d_1\w),\ldots, \q( X_q,d_q\w)$ can
be any such list.
\end{defn}

\begin{rmk}
We can rephrase the above in the following way.  We iteratively
take commutators of the $X_j$s and assign formal degrees as above.
The assumption that $\q(X_1,d_1\w),\ldots, \q(X_r,d_r\w)$ generates
a
finite list says that after a finite number of steps,
one ``does not get anything new.''
For an example of how a list $\q( X_1,d_1\w),\ldots, \q( X_r,d_r\w)$
might not generate a finite list, see Section \ref{SectionFlagVsProduct}.
\end{rmk}

\section{Surfaces}\label{SectionCurves}
In this section, we define the class $\gamma$ for which
we will study operators of the form \eqref{EqnIntroOp}.
The functions we study generalize the ones studied in
\cite{ChristNagelSteinWaingerSingularAndMaximalRadonTransforms}:
even in the single parameter case, when $K$ is a Calder\'on-Zygmund
kernel, our class of $\gamma$ is somewhat wider that those
studied in \cite{ChristNagelSteinWaingerSingularAndMaximalRadonTransforms}
(cf. the special case in Section \ref{SectionSpecialCase}).

In \cite{ChristNagelSteinWaingerSingularAndMaximalRadonTransforms}
a number of equivalent ``curvature conditions'' on $\gamma$ are defined under
which the single-parameter version of operators
of the form \eqref{EqnIntroOp} are studied.
We discuss some of these conditions in Section \ref{SectionCNSW}.
Unfortunately, the most convenient way (for our purposes) to look
at these conditions was not directly studied in
\cite{ChristNagelSteinWaingerSingularAndMaximalRadonTransforms}.
In Section \ref{SectionCNSW} we define and discuss a new curvature condition,
$\cZ$, which is equivalent to the other curvature conditions
in \cite{ChristNagelSteinWaingerSingularAndMaximalRadonTransforms}.
In this section, we jump straight to the generalization of $\cZ$
to the multi-parameter situation.\footnote{Strictly speaking,
what follows is not a straight generalization of the curvature
conditions in \cite{ChristNagelSteinWaingerSingularAndMaximalRadonTransforms}.
As mentioned in Section \ref{SectionSpecialCase}, our results
will hold for a larger class of $\gamma$ in the single-parameter case.
However, it is reasonable to think of what follows
as the multi-parameter ``generalization'' of the curvature
conditions in \cite{ChristNagelSteinWaingerSingularAndMaximalRadonTransforms}.}

We assume that we are given an open subset $\Omega\subseteq \R^n$
and a subset $\sA\subseteq \q[0,1\w]^\nu$ (such that
if $\delta_1,\delta_2\in \sA$, then $\delta_1\vee \delta_2\in \sA$),
and $\nu$-parameter dilations $e$ as in Section \ref{SectionKernels}.

\begin{defn}
Given a multi-index $\alpha\in \N^N$, we define,
\begin{equation*}
\deg\q( \alpha\w) = \sum_{j=1}^N \alpha_j e_j\in \q[0,\infty\w)^\nu.
\end{equation*}
\end{defn}

Let $\K\Subset \Omega'\Subset \Omega''\Subset \Omega$ be subsets of $\Omega$
with $\K$ compact and $\Omega'$ and $\Omega''$ open by relatively
compact in $\Omega$.  Our goal in this section is to define
a class of $C^\infty$ functions
\begin{equation*}
\gamma\q( t,x\w):\Q^N\q( \rho\w)\times \Omega''\rightarrow \Omega
\end{equation*}
such that $\gamma\q( 0,x\w) \equiv x$.  Here $\rho>0$ is a small
number.
This class of functions will depend on $\sA$, $N$, and $e$ (it 
will also depend on $\K$, $\Omega$, $\Omega'$, and $\Omega''$).
This class will be such that if $\psi$ is a $C_0^\infty$
function supported on the interior of $\K$, then there is
an $a>0$, sufficiently small, such that the operator
given by \eqref{EqnIntroOp} is bounded on $L^2$ for every
$K\in \sK\q( N,e,a,\sA\w)$.

As in Section \ref{SectionIntro}, we will write $\gamma_t\q(x\w) = \gamma\q( t,x\w)$.  Note, by possibly shrinking $\rho>0$ we may assume that for
every $t\in \Q^N\q( \rho\w)$, $\gamma_t$ is a diffeomorphism onto
its image.  From now on we assume this, so that it makes
sense to write $\gamma_t^{-1}$ throughout the paper.

Unlike the work in \cite{ChristNagelSteinWaingerSingularAndMaximalRadonTransforms},
we separate the conditions on $\gamma_t$ into two aspects.
For the first, suppose we are given a list of $C^\infty$ vector
fields on $\Omega''$, $X_1,\ldots, X_q$
with associated $\nu$-parameter formal degrees $d_1,\ldots,d_q$
satisfying $\sD\q( \K,\sA,\Omega',\xi\w)$ for some $\xi>0$ (we
will see later where these vector fields come from).

\begin{defn}\label{DefnControlWEveryScale}
Suppose we are given a $C^\infty$ vector field on $\Omega''$,
depending smoothly on $t\in \Q^N\q( \rho\w)$,
$W\q( t,x\w)\in T_x \Omega''$.  We say that $\q( X,d\w)$
{\it controls} $W\q( t,x\w)$ if there exists $0<\rho_1\leq \rho$
and $0<\tau_1\leq \xi$ such that
for every $x_0\in \K$, $\delta\in \sA$, there exists
functions $c_l^{x_0,\delta}$ on $\Q^{N}\q( \rho_1\w)\times \B{X}{d}{x_0}{\tau_1 \delta}$ satisfying,
\begin{itemize}
\item $W\q( \delta t, x\w) = \sum_{l=1}^q c_l^{x_0,\delta} \q( t,x\w) \delta^{d_l} X_l\q( x\w)$ on $\Q^{N}\q( \rho_1\w)\times \B{X}{d}{x_0}{\tau_1\delta}$, where $\delta t$ is defined as in \eqref{EqnDefndeltat}.
\item $\sup_{\substack{x_0\in \K \\ \delta\in \sA}} \sum_{\q|\alpha\w|+\q|\beta\w|\leq m} \CjN{\q(\delta X \w)^\alpha \partial_t^{\beta} c_l^{x_0,\delta} }{0}{\Q^{N}\q( \rho_1\w) \times \B{X}{d}{x_0}{\tau_1\delta}}<\infty,$ for every $m$.
\end{itemize}
\end{defn}

\begin{defn}\label{DefnControlEveryScale}
We say $\q( X,d\w)$ {\it controls} $\gamma_t\q( x\w)$ if $\q( X,d\w)$
controls $W\q( t,x\w)$, where,
\begin{equation*}
W\q( t,x\w) = \frac{d}{d\epsilon}\bigg|_{\epsilon=1} \gamma_{\epsilon t}\circ \gamma_t^{-1}\q( x\w);
\end{equation*}
here, $\epsilon t=\q( \epsilon t_1,\ldots, \epsilon t_n\w)$ and so
has nothing to do with the dilations $e$.
\end{defn}

Part of our assumption on $\gamma$ will be that a particular list of vector
fields $\q( X,d\w)$ controls $\gamma_t$.  Where these vector fields
come from constitutes the other part of our assumption.

Let $W$ be as in Definition \ref{DefnControlEveryScale}.
Let $X_\alpha$ be the Taylor coefficients of $W$
when the Taylor series is taken in the $t$ variable:
\begin{equation}\label{EqnDefnXalpha}
W\q( t\w) \sim \sum_{\alpha} t^{\alpha} X_\alpha,
\end{equation}
so that $X_\alpha$ is a $C^\infty$ vector field on $\Omega''$.

Our assumption on $\gamma$ is that if we take the set
of vector fields with formal degrees:
\begin{equation}\label{EqnCurvesDefnS}
\sV=\q\{\q(X_\alpha, \deg\q( \alpha\w) \w): \deg\q( \alpha\w)\text{ is nonzero in only one component}\w\},
\end{equation}
then there is a finite subset $\sF\subseteq \sV$ such that $\sF$
generates a finite list $\q( X,d\w)=\q( X_1,d_1\w),\ldots, \q(X_q,d_q\w)$,
and this finite list controls $\gamma$.

\begin{rmk}
The list of vector fields $\q( X,d\w)$ depends on a few choices we have made in the above:
it depends on the chosen subset $\sF$ and it depends on
the chosen list generated by $\sF$.  However, none of these
choices affects $\q( X,d\w)$ in an essential way.
This is discussed in Section \ref{SectionMoreOnXd}.
\end{rmk}

\begin{rmk}
Note that the above assumptions are local.  Indeed,
if $C$ is a large compact set, and if for each point
$x_0\in C$, $\gamma$ satisfies the above assumptions
with $\K$ some compact neighborhood of $x_0$, then
$\gamma$ satisfies the above assumptions
with $\K$ replaced by $C$.
Thus, one should think of $\K$
as a small neighborhood of a point.  However, we find it more
straightforward to fix $\K$ in what follows, as opposed to letting
it (possibly) shrink from line to line.
We took the opposite perspective in Sections \ref{SectionIntroRes} and \ref{SectionOverview}.
\end{rmk}

\begin{rmk}
The ``curvature conditions'' in
\cite{ChristNagelSteinWaingerSingularAndMaximalRadonTransforms}
are equivalent to saying $\q\{ X_\alpha: \alpha\in \N^{N}\w\}$ satisfies H\"ormander's
condition (see $\cZ$ in Section \ref{SectionCNSW}).  We will
see in Section \ref{SectionEXCNSW} that this is a special case
of the above assumptions, when $\nu=1$.
\end{rmk}

\begin{rmk}
Note that our conditions on $\gamma$ only involve the
vector field $W$.  It is not hard to see, though, that
$W$ uniquely determines $\gamma$.  This is discussed
further in Section \ref{SectionCurvesII}.
\end{rmk}

\begin{rmk}
Notice the scale invariance of Definition \ref{DefnControlEveryScale}.
Indeed, for $\delta_1,\delta_2\in \sA$, the condition
imposed on $W\q( \delta_1 t, x\w)$ is formally the same
as the condition imposed on $W\q( \delta_2 t, x\w)$ provided
one replaces $\delta_1 X$ with $\delta_2 X$.  This scale invariance
will play an essential role in our proofs.  This also
explains the entrance of Carnot-Carath\'eodory geometry
in our assumptions:  the Carnot-Carath\'eodory balls
scale in the same way that our other assumptions scale.
\end{rmk}

\begin{rmk}
Unlike in Section \ref{SectionOverview}, we have not separated our
assumptions on $\gamma$ into a ``finite type'' condition
and an ``algebraic'' condition.  This distinction only has
real significance when $\gamma$ is real analytic,
and we, therefore, differ discussion of it
to 
\cite{StreetMultiParameterSingRadonAnal}.
\end{rmk}

\section{Statement of results}\label{SectionResults}
Fix $\Omega\subseteq \R^n$ open, and $\K\Subset\Omega'\Subset\Omega''\Subset \Omega$ with $\K$ compact (with nonempty interior) and $\Omega'$ and $\Omega''$ open but relatively compact
in $\Omega$.
Let
$$\gamma\q( t,x\w) : \Q^N\q( \rho\w)\times \Omega''\rightarrow \Omega$$
be a $C^\infty$ function such that $\gamma\q( 0,x\w)\equiv x$.  Here,
$\rho>0$ is a small number.

Fix $\nu\in \N$ positive, and $\sA\subseteq \q[0,1\w]^\nu$ such that
for $\delta_1,\delta_2\in \sA$, $\delta_1\vee \delta_2\in \sA$.
Furthermore, let $e=\q( e_1,\ldots, e_N\w)$ be given,
with $0\ne e_j\in \q[0,\infty\w)^\nu$.
We suppose $\gamma$ satisfies the assumptions of
Section \ref{SectionCurves} with this $\K$, $\sA$, and $e$.

\begin{thm}\label{ThmMainThmFirstPass}
For every $\psi\in C_0^\infty\q( \R^n\w)$ supported in the
interior of $\K$, there
exists $a>0$ such that
for every $K\in \sK\q( N,e,a,\sA\w)$
the operator
\begin{equation}\label{EqnMainThmFirst}
Tf\q(x\w)= \psi\q( x\w) \int f\q( \gamma_t\q( x\w)\w)K\q( t\w) \: dt
\end{equation}
extends to a bounded operator $L^2\q( \R^n\w)\rightarrow L^2\q( \R^n\w)$.
\end{thm}

Actually, Theorem \ref{ThmMainThmFirstPass} follows directly
from the following, slightly more general theorem.
\begin{thm}\label{ThmMainThmSecondPass}
There exists $a>0$ such that for every $\psi_1,\psi_2\in C_0^\infty\q( \R^n\w)$
supported on the interior of $\K$, every $K\in \sK\q( N,e,a,\sA\w)$ and
every $C^\infty$ function
$$\kappa\q( t,x\w): \Q^N\q( a\w)\times \Omega''\rightarrow \C$$
the operator
\begin{equation}\label{EqnMainThmSecond}
T\q(f\w) \q( x\w)= \psi_1\q( x\w) \int f\q( \gamma_t\q( x\w)\w)\psi_2\q( \gamma_t\q( x\w)\w)\kappa\q( t,x\w)K\q( t\w) \: dt
\end{equation}
extends to a bounded operator $L^2\q( \R^n\w) \rightarrow L^2\q( \R^n\w)$.
\end{thm}

\begin{proof}[Proof of Theorem \ref{ThmMainThmFirstPass} given Theorem \ref{ThmMainThmSecondPass}]
Given $\psi$ take $\psi_2$ to be equal to $1$ on a neighborhood of the support of $\psi$.
It is easy to see that if $K$ has sufficiently small support, then
$\psi_2\q( \gamma_t\q( x\w)\w)=1$ on the domain of integration
in \eqref{EqnMainThmSecond} (with $\psi$ in place of $\psi_1$).
Recall that $K$ is supported in $\Q^N\q( a\w)$, and so this
may be achieved by shrinking $a$.
Taking $\kappa\equiv 1$ yields Theorem \ref{ThmMainThmFirstPass}.
\end{proof}

\begin{rmk}
The reason we work with Theorem \ref{ThmMainThmSecondPass} is that
the class of operators given by \eqref{EqnMainThmSecond} is
closed under adjoints, while the class of operators
given by \eqref{EqnMainThmFirst} is not.
This may not be immediately obvious from our assumptions.
See Section \ref{SectionAdjoint} for details.
\end{rmk}

\begin{rmk}
In the second paper in this series (joint with Elias Stein \cite{SteinStreetMultiParameterSingRadonLp}), we will
prove a corresponding $L^p$ theorem ($1<p<\infty$).  The operators
for the $L^p$ theorem will not be as general as those covered by
Theorem \ref{ThmMainThmSecondPass}:  we will need to assume
a special form for the set $\sA$.  However, $\sA=\q[0,1\w]^\nu$
is of this form, and so in that special case, the results
of Theorem \ref{ThmMainThmSecondPass} do extend to
$L^p$.
\end{rmk}

\section{Basic Notation}\label{SectionNotation}
Throughout the paper, for $v=\q(v_1,\ldots, v_n\w)\in \R^n$, we write 
$\q|v\w|$ for $\q( \sum_j \q|v_j\w|^2 \w)^{\frac{1}{2}}$, and we write
$\q|v\w|_\infty$ for $\sup_j \q|v_j\w|$.
$B^{n}\q( \eta\w)$ will denote
the ball of radius $\eta>0$ in the $\q|\cdot\w|$ norm.
For two numbers $a,b\in \R$ we write
$a\vee b$ for the maximum of $a$ and $b$ and $a\wedge b$ for the minimum.
If instead, $a=\q( a_1,\ldots, a_n\w), b=\q( b_1,\ldots, b_n\w)\in \R^n$, we write $a\vee b$ (respectively, $a\wedge b$)
for $\q(a_1\vee b_1,\ldots, a_n \vee b_n \w)$ (respectively, $\q(a_1\wedge b_1,\ldots, a_n \wedge b_n\w)$).

For a vectors $\delta=\q(  \delta_1,\ldots, \delta_\nu\w), d=\q( d_1,\ldots, d_\nu\w)\in \R^\nu$, we define $\delta^d$ by the standard multi-index notation.
I.e., $\delta^d=\prod_{\mu=1}^{\nu} \delta_\mu^{d_\mu}$.  Also we will
write $2^d = \q(2^{d_1},\ldots,2^{d_\nu}\w)$.  

Given a, possibly arbitrary, set $U\subseteq \R^n$ and a continuous
function $f$ defined on a neighborhood of $U$, we write
$$\CjN{f}{j}{U} = \sum_{\q|\alpha\w|\leq j}\sup_{x\in U} \q|\partial_x^\alpha f\q( x\w)\w|,$$
and if we state that $\CjN{f}{j}{U}$ is finite, we mean that the partial
derivatives up to order $j$ of $f$ exist on $U$, are continuous, and the above norm
is finite.
If $f$ is replaced by a vector field $Y=\sum_k a_k\q( x\w)\partial_{x_k}$, then we write,
$$\CjN{Y}{j}{U} = \sum_{k}\CjN{a_k}{j}{U}.$$

Given two integers $1\leq m\leq n$, we let $\sI{m}{n}$ denote the set
of all lists of integers $\q( i_1,\ldots, i_m\w)$ such that
$$1\leq i_1<i_2<\cdots<i_m\leq n.$$
Furthermore, suppose $A$ is an $n\times q$ matrix, and suppose
$1\leq n_0\leq n\wedge q$.
For $I\in \sI{n_0}{n}$, $J\in \sI{n_0}{q}$ we let $A_{I,J}$ denote
the $n_0\times n_0$ matrix given by taking the rows from $A$ which
are listed in $I$ and the columns from $A$ which are listen in $J$.
We define
$$\Det{n_0} A = \q( \det A_{I,J} \w)_{\substack{I\in \sI{n_0}{n}\\J\in \sI{n_0}{q}}},$$
so that, in particular, $\Det{n_0} A$ is a {\it vector} (it will not
be important to us in which order the coordinates are arranged).
$\Det{n_0} A$ comes up when one changes variables.  Indeed, suppose
$\Phi$ is a $C^1$ diffeomorphism from an open subset $U\subset \R^{n_0}$
mapping to an $n_0$ dimensional submanifold of $\R^n$, where
this submanifold is given the induced Lebesgue measure $dx$.  Then, we have
$$\int_{\Phi\q( U\w)} f\q( x\w) \: dx = \int_U f\q( \Phi\q( t\w)\w) \q|\Det{n_0} d\Phi\q( t\w)\w|\: dt.$$

If $A=\q( A_1,\ldots, A_q\w)$ is a list of, possibly non-commuting, operators,
we will use ordered multi-index notation to define $A^\alpha$, where
$\alpha$ is a list of numbers $1,\ldots, q$.  $\q|\alpha\w|$ will
denote the length of the list.  For instance, if
$\alpha=\q( 1,4,4,2,1\w)$, then $\q| \alpha\w|=5$ and $A^\alpha = A_1A_4A_4A_2A_1$.  Thus, if $A_1,\ldots A_q$ are vector fields, then $A^\alpha$ is an
$\q|\alpha\w|$ order partial differential operator.

If $f:\R^{n}\rightarrow \R^m$ is a map, then we write,
\begin{equation*}
d f\q( x\w) \q( \frac{\partial}{\partial x_j}\w),
\end{equation*}
for the differential of $f$ at the point $x$ applied to the vector field
$\frac{\partial}{\partial x_j}$.  If $f$ is a function of two variables,
$f\q( t,x\w): \R^N\times \R^n\rightarrow \R^m$, and we wish to view $d f$ as a linear transformation
acting on the vector space spanned by $\frac{\partial}{\partial_{t_j}}$ ($1\leq j\leq N$),
then we instead write,
\begin{equation*}
\frac{\partial f}{\partial t} \q( t,x\w)
\end{equation*}
to denote this linear transformation.  Hence, it makes sense to write,
\begin{equation*}
\det_{n_0\times n_0} 
\frac{\partial f}{\partial t} \q( t,x\w),
\end{equation*}
where $n_0\leq m\wedge N$.

Finally, we will devote a good deal of notation to multi-parameter
Carnot-Carath\'eodory geometry. See 
Sections \ref{SectionCC} and \ref{SectionCCII}.

\section{The work of Christ, Nagel, Stein, and Wainger}\label{SectionCNSW}
In this section we review some of the aspects
of the fundamental work
\cite{ChristNagelSteinWaingerSingularAndMaximalRadonTransforms},
and we also discuss some further results that follow easily
from that work, which will be useful in the multi-parameter
setting.
The results of this section will be of use to us in two ways.
First, the work in \cite{ChristNagelSteinWaingerSingularAndMaximalRadonTransforms}
offered the primary inspiration for the results in this paper, and a review
of these results will offer the reader
much better context in which to understand this paper.
Second, we will be able to save quite a bit of effort by
``transferring'' some of the single parameter results
from \cite{ChristNagelSteinWaingerSingularAndMaximalRadonTransforms}
to the multi-parameter setting.  For an example of
this ``transferring'' in a simpler context,
see Section 5.2.4 of \cite{StreetMultiParameterCCBalls}.

In \cite{ChristNagelSteinWaingerSingularAndMaximalRadonTransforms},
operators of the form,
\begin{equation}\label{EqnCNSWOp}
T f\q( x\w) = \psi\q( x\w) \int f\q( \gamma_t\q( x\w)\w) K\q( t\w) \: dt
\end{equation}
were studied, where $\psi$ is a cutoff function supported near
a fixed point $x_0\in \R^n$, $K$ is a Calder\'on-Zygmund kernel
supported near $t=0$, and $\gamma: \R^N\times \R^n\rightarrow \R^n$
is a $C^\infty$ function\footnote{Actually, we only need $\gamma$ defined
near $0\in \R^N$ and $x_0\in \R^n$.}
with $\gamma_0\q( x\w) \equiv x$ and satisfying an additional condition,
denoted by $\cC$ in \cite{ChristNagelSteinWaingerSingularAndMaximalRadonTransforms}.

Moreover, \cite{ChristNagelSteinWaingerSingularAndMaximalRadonTransforms}
showed that $\cC$ could be viewed as any one of a number of
equivalent conditions.  Three of their conditions are useful for our
purposes.  We review these three below, along with a fourth condition
(not directly studied in \cite{ChristNagelSteinWaingerSingularAndMaximalRadonTransforms}) which is also equivalent to the other three.

\begin{itemize}
\item $\cG$:  \cite{ChristNagelSteinWaingerSingularAndMaximalRadonTransforms}
showed that $\gamma_t$ can be written asymptoticly as:\footnote{\eqref{EqnExpRepn} means that, for every $M$, $\gamma_t\q( x\w) = \exp\q(\sum_{0<\q|\alpha\w|<M} t^{\alpha} \Xh_\alpha \w)+ O\q( \q|t\w|^M\w)$, as $t\rightarrow 0$.}
\begin{equation}\label{EqnExpRepn}
\gamma_t\q( x\w) \sim \exp\q(\sum_{\q|\alpha\w|>0} t^{\alpha} \Xh_\alpha \w)x,
\end{equation}
where each $\Xh_\alpha$ is a $C^\infty$ vector field on $\R^n$ (see
Theorem 8.5 of \cite{ChristNagelSteinWaingerSingularAndMaximalRadonTransforms}).
Condition $\cG$ then states that the vector fields $\Xh_\alpha$ satisfy
H\"ormander's condition.  I.e., that the Lie algebra generated by the $\Xh_\alpha$
spans the tangent space to $\R^n$ at $x_0$.

\item $\cJ$:  One defines $\Gamma: \R^{Nn}\times \R^{n}\rightarrow \R^n$ by
\begin{equation}\label{EqnDefnGammacJ}
\Gamma\q( \tau, x\w) = \gamma_{t^1}\circ\gamma_{t^2}\circ\cdots \circ\gamma_{t^n}\q( x\w),
\end{equation}
where $\tau=\q( t^1,\ldots, t^n\w)\in \R^{Nn}$.  $\cJ$ then states
that for some multi-index $\beta$, we have,
\begin{equation*}
\q| \q(\frac{\partial}{\partial\tau} \w)^\beta \det_{n\times n} \frac{\partial \Gamma }{\partial\tau} \q( \tau,x_0\w)\bigg|_{\tau=0} \w|\ne 0.
\end{equation*}

\item $\cY$:\footnote{See the bottom of page 521 of \cite{ChristNagelSteinWaingerSingularAndMaximalRadonTransforms}
to find $\cY$.}  For $1\leq j\leq N$, define the $C^\infty$ vector fields,
\begin{equation*}
W_j\q( t,x\w) = d\gamma\q( t,\gamma_t^{-1}\q( x\w)\w) \q(\frac{\partial}{\partial t_j}\w) = \frac{\partial}{\partial s_j}\bigg|_{s=0} \gamma_{t+s}\circ \gamma_t^{-1}\q( x\w).
\end{equation*}
We now express $W_j$ as a Taylor series in the $t$ variable,
\begin{equation*}
W_j\q( t\w) \sim \sum_{\alpha} t^{\alpha} \frac{X_{\alpha,j}}{\alpha!},
\end{equation*}
where $X_{\alpha,j}$ is a $C^\infty$ vector field.  $\cY$ then
states that the vector fields $\q\{X_{\alpha,j}: \q|\alpha\w|\geq 0, 1\leq j\leq N \w\}$ satisfy H\"ormander's condition at $x_0$.

\item $\cZ$:  Define the $C^\infty$ vector field,\footnote{Once again,
$\epsilon t=\q( \epsilon t_1,\ldots, \epsilon t_n\w)$ and so does not
use any non-standard dilations as in \eqref{EqnDefndeltat}.}
\begin{equation*}
W\q( t,x\w) = \frac{\partial}{\partial \epsilon}\bigg|_{\epsilon=1} \gamma_{\epsilon t}\circ\gamma_t^{-1}\q( x\w).
\end{equation*}
Express $W$ as a Taylor series in the $t$ variable,
\begin{equation*}
W\q( t\w) \sim \sum_{\alpha} t^{\alpha} X_\alpha,
\end{equation*}
where $X_\alpha$ is a $C^\infty$ vector field.  $\cZ$ then states that
the vector fields $\q\{X_\alpha\w\}$ satisfy H\"ormander's condition
at $x_0$.
\end{itemize}
The following result is mostly contained in
\cite{ChristNagelSteinWaingerSingularAndMaximalRadonTransforms}:
\begin{thm}\label{ThmEquivConds}
$\cG\Leftrightarrow \cJ\Leftrightarrow \cY\Leftrightarrow \cZ$.
We denote by $\cC$ any of the above equivalent conditions.
\end{thm}
\begin{proof}
$\cG\Leftrightarrow \cJ$ is contained in Theorem 8.8
of \cite{ChristNagelSteinWaingerSingularAndMaximalRadonTransforms}.
$\cG\Leftrightarrow \cY$ is Proposition 9.6 of
\cite{ChristNagelSteinWaingerSingularAndMaximalRadonTransforms}.
We are therefore left with showing only $\cY\Leftrightarrow \cZ$
and the proof will be complete.

For a set of $C^\infty$ vector fields $\sV$, let 
$\Lie\q(\sV\w)$ denote the Lie algebra generated by $\sV$.
Let $W_j$, $X_{\alpha,j}$, $W$, and $X_\alpha$ be as in
the definitions of $\cY$ and $\cZ$.
We will show,
\begin{equation*}
\Lie\q\{X_{\alpha,j}: 1\leq j\leq N, \alpha\in \N^{N}\w\}=\Lie \q\{X_\alpha: \alpha\in \N^N\w\},
\end{equation*}
which will complete the proof.  In fact, we will show for every $M\in \N$,
\begin{equation}\label{EqnToShowLie}
\Lie\q\{X_{\alpha,j}: 1\leq j\leq N, \q|\alpha\w|\leq M\w\}=\Lie \q\{X_\alpha: \q|\alpha\w|\leq M+1\w\}.
\end{equation}

We proceed by induction on $M$.
It follows directly from the definition of $W$ and $W_j$ that,
\begin{equation}\label{EqnRelWWj}
W\q( t,x\w) = \sum_{j=1}^N t_j W_j\q( t,x\w).
\end{equation}
In light of \eqref{EqnRelWWj}, the base case ($M=0$) of \eqref{EqnToShowLie}
is trivial.
We assume we have \eqref{EqnToShowLie} for $M-1$ and prove it for $M$.
From \eqref{EqnRelWWj}, the containment $\supseteq$ is trivial,
and so we prove only the containment $\subseteq$.
To do so, we use the fact that $\gamma$ satisfies the ODE in each variable,
\begin{equation}\label{EqnWjDiffEq}
\frac{\partial}{\partial t_j} \gamma_t\q( x\w) = W_j \q( t, \gamma_t\q( x\w)\w).
\end{equation}
Hence, the $W_j$ must satisfy the integrability condition:
\begin{equation}\label{EqnWjInteg}
\frac{\partial W_j}{\partial t_k} - \frac{\partial W_k}{\partial t_j} = \q[W_j,W_k\w].
\end{equation}
See, for example,\footnote{That \eqref{EqnWjInteg} must hold follows
simply by using the equality $\partial_{t_k} \partial_{t_j} \gamma_t = \partial_{t_j} \partial_{t_k} \gamma_t$ and applying \eqref{EqnWjDiffEq}.} 
Theorem 10.9.4 of \cite{DieudonneFoundationsOfModernAnalysis}.
Let $f_j$ be the multi-index which is $1$ in the $j$th coordinate
and $0$ in all other coordinates.  From, \eqref{EqnWjInteg},
we obtain,
\begin{equation*}
\frac{X_{\alpha+f_k,j}}{\alpha!}- \frac{X_{\alpha+f_j,k}}{\alpha!} = \sum_{\gamma+\delta=\alpha} \q[\frac{X_{\gamma,j}}{\gamma!}, \frac{X_{\delta,k}}{\delta!}\w].
\end{equation*}
As a consequence, if $\q|\alpha\w|=M-1$, our inductive hypothesis shows,
\begin{equation}\label{Eqn0ModLie}
X_{\alpha+f_k,j} - X_{\alpha+f_j,k} \equiv 0 \mod \Lie\q\{X_\alpha: \q|\alpha\w|\leq M\w\}.
\end{equation}
Fix $\beta$ with $\q|\beta\w|=M$.  We wish to show $X_{\beta,j}\in \Lie\q\{X_\alpha: \q|\alpha\w|\leq M+1\w\}$.
For every $k$ such that $\beta-f_k\in \N^\nu$, we have
by \eqref{Eqn0ModLie}
 (since $\q|\beta-f_k\w|=M-1$),
\begin{equation*}
X_{\beta+f_j-f_k,k}\equiv X_{\beta,j} \mod \Lie\q\{X_\alpha: \q|\alpha\w|\leq M\w\}.
\end{equation*}

Applying \eqref{EqnRelWWj}, we see that the coefficient of $t^{\beta+f_j}$
in the Taylor series for $W$ is given by,
\begin{equation*}
X_{\beta+f_j}=\sum_{\substack{1\leq k\leq N\\ \beta-f_k\in \N^N}} \frac{X_{\beta+f_j-f_k,k} }{\q( \beta+f_j-f_k\w)!} \equiv C X_{\beta,j}\mod \Lie\q\{X_\alpha: \q|\alpha\w|\leq M\w\},
\end{equation*}
where $C$ is a nonzero constant.  Hence,
$X_{\beta,j}\in \Lie \q\{X_\alpha: \q|\alpha\w|\leq M+1\w\}$,
completing the proof.
\end{proof}

One of the main theorems of
\cite{ChristNagelSteinWaingerSingularAndMaximalRadonTransforms}
is the following:
\begin{thm}[Theorem 11.1 of \cite{ChristNagelSteinWaingerSingularAndMaximalRadonTransforms}]\label{ThmCNSWBound}
Suppose $\gamma$ satisfies $\cC$ at some point $x_0$, and
suppose $\psi$ is a $C_0^\infty$ function supported on a
sufficiently small neighborhood of $x_0$ and $K$ is
a standard Calder\'on-Zygmund kernel supported
sufficiently close to $0$.  Then,
the operator given by \eqref{EqnCNSWOp} extends
to a bounded operator on $L^p$ ($1<p<\infty$).
\end{thm}


\begin{rmk}
We will use multi-parameter analogs of the concepts
in the above four conditions.  For example,
the assumptions placed on $\gamma$ in Section \ref{SectionCurves}
is the relevant multi-parameter analog of $\cZ$.
We will discuss multi-parameter analogs of the other
conditions described above, and a multi-parameter
analog of Theorem \ref{ThmEquivConds},
in Section \ref{SectionMultiParamCurvature}.
\end{rmk}

As is discussed in Section 5.2.4 of \cite{StreetMultiParameterCCBalls},
results in the single-parameter setting (like Theorem \ref{ThmEquivConds})
can often be lifted to the multi-parameter setting,
provided the results are uniform, in an appropriate sense.
Fortunately, in this instance, the desired uniformity
follows from Theorem \ref{ThmEquivConds}
via a compactness argument.
We now state and prove the main consequence 
of Theorem \ref{ThmEquivConds} which we will use.

For notational convenience, we take $x_0=0$ in what follows.
We wish to define ``uniform'' versions of $\cG$, $\cJ$, $\cY$, and $\cZ$.
Fix $\rho>0$, $\eta>0$, and let $\sS$ be a set of $C^\infty$
functions $\gamma:\Q^N\q( \rho\w) \times B^{n}\q( \eta\w)\rightarrow \R^n$,
satisfying $\gamma_0\q( x\w)\equiv x$.  We wish to define
the notion of 
the above conditions holding
uniformly for $\gamma\in \sS$.  We denote this by
$\cGu$, $\cJu$, $\cYu$, and $\cZu$ respectively.

\begin{itemize}
\item $\cGu$: For $\gamma\in \sS$, define $\Xh_\alpha$
as in $\cG$.  $\cGu$ states that there exists $M\in \N$,
independent of $\gamma\in \sS$,
such that $\q\{\Xh_\alpha: \q|\alpha\w|\leq M\w\}$
satisfies H\"ormander's condition at $0$, uniformly
for $\gamma\in \sS$.  More precisely, that there exists
$M'\in \N$, $c>0$, independent of $\gamma\in \sS$
such that if we let $V_1,\ldots, V_L$ denote the list
of vector fields containing $\q\{\Xh_\alpha: \q|\alpha\w|\leq M\w\}$,
along with all commutators of the vector fields in $\q\{\Xh_\alpha: \q|\alpha\w|\leq M\w\}$
up to order $M'$, then we have,
\begin{equation*}
\q|\det_{n\times n} V\q( 0\w)\w|\geq c,
\end{equation*}
where we have written $V$ to denote the matrix whose columns are
$V_1,\ldots, V_L$.
\item $\cJu$: For $\gamma\in \sS$, define $\Gamma$ as in
$\cJ$.  $\cJu$ then states that there is an $M\in \N$ and
a $c>0$, both independent of $\gamma\in \sS$ such that
for every $\gamma\in \sS$ there exists $\beta$ with $\q|\beta\w|\leq M$,
and
\begin{equation*}
\q| \q(\frac{\partial}{\partial\tau} \w)^\beta \det_{n\times n} \frac{\partial \Gamma }{\partial\tau} \q( \tau,0\w)\bigg|_{\tau=0} \w|\geq c.
\end{equation*}

\item $\cYu$: For $\gamma\in \sS$, define $X_{\alpha,j}$ as
in $\cY$.  $\cYu$ states that there exists $M\in \N$, independent
of $\gamma\in\sS$ such that
$\q\{X_{\alpha,j}:\q|\alpha\w|\leq M, 1\leq j\leq N\w\}$
satisfies H\"ormander's condition at $0$, uniformly
for $\gamma \in \sS$ (as in $\cGu$).

\item $\cZu$: For $\gamma\in \sS$, define the vector fields $\q\{X_\alpha\w\}$
as in $\cZ$.  $\cZu$ states that there exist $M\in \N$, independent
of $\gamma\in \sS$, such 
$\q\{X_\alpha: \q|\alpha\w|\leq M\w\}$
satisfies H\"ormander's condition at $0$, uniformly for $\gamma\in \sS$ (as in
$\cGu$).
\end{itemize}

\begin{thm}\label{ThmUnifEquivCond}
Let $\sS$ be as above, and suppose $\sS\subset C^{\infty}\q(\Q^N\q( \rho\w)\times B^n\q( \eta\w); \R^n \w)$ is a bounded set.
Then, $\cGu\Leftrightarrow\cJu\Leftrightarrow \cYu\Leftrightarrow \cZu$.
\end{thm}
\begin{proof}
We prove only $\cZu\Rightarrow \cJu$ as that is the most important implication
for our purposes.
All of the other implications follow in a completely analogous manner.
Suppose $\sS$ satisfies $\cZu$ but not $\cJu$.
Let $\gamma^j\in \sS$ be a sequence satisfying the following.
Define $\Gamma^j$ in terms of $\gamma^j$ as in \eqref{EqnDefnGammacJ}.
We choose $\gamma^j$ so that for every multi-index $\beta$
with $\q|\beta\w|\leq j$, we have,
\begin{equation*}
\q|\q(\frac{\partial}{\partial \tau} \w)^\beta \det_{n\times n} \frac{\partial \Gamma^j}{\partial \tau} \q( \tau ,0\w)\w|\leq \frac{1}{j}.
\end{equation*}
This is clearly possible since $\sS$ does not satisfy $\cJu$.
Since bounded subsets of $C^\infty$ are precompact, we may
select a convergent subsequence $\gamma^j\rightarrow \gamma^{\infty}$.
It is easy to verify (since $\sS$ satisfies $\cZu$) that $\gamma^{\infty}$
satisfies $\cZ$.  It is also easy to see, by our choice of $\gamma^j$,
that $\gamma^\infty$ does not satisfy $\cJ$.  This contradicts
Theorem \ref{ThmEquivConds} and completes the proof.
\end{proof}


\begin{rmk}
The reader might suspect that one could improve
Theorem \ref{ThmUnifEquivCond} by tracing through
the proof in \cite{ChristNagelSteinWaingerSingularAndMaximalRadonTransforms},
and determining in precisely which way $M$ and $c$ (from $\cJu$)
depend on the $C^j$ norms of the elements of $\sS$
and $M_1$, $M_2$, and $c$ (from $\cZu$).  In fact,
this is the case (at least if one replaces
$\cJ$ with $\cJ'$ as defined in Definition 10.9 of
\cite{ChristNagelSteinWaingerSingularAndMaximalRadonTransforms}).
One can use this to improve the way in which various constants depend
on each other in this paper.  However, this improvement
would not change the main thrust of our results.
Moreover, even if one does proceed in this manner,
all of our results would still be far from optimal in this sense.
We, therefore, leave such details to the interested reader.
\end{rmk}

\section{Informal outline of the proof}\label{SectionProofOutline}
Before we begin the rigorous proof of
Theorem \ref{ThmMainThmSecondPass}, in this section we offer
an informal outline to help the reader see the big picture of
what follows.
We also highlight the difficulties that arise
in the present work, that did not arise in previous
works, such as \cite{ChristNagelSteinWaingerSingularAndMaximalRadonTransforms}.
We take all the same notation as in Theorem \ref{ThmMainThmSecondPass}.
Furthermore, for the purposes of this section, we imagine that
the vector fields $X_1,\ldots, X_q$ span the tangent space at every point.\footnote{These are the vector fields from Section \ref{SectionCurves}, obtained
from the map $\gamma$.}
This will not be used in later sections, but it makes the
informal description of the argument simpler.
Letting $K\in \sK\q( N,e,a,\sA\w)$ we may decompose $K$,
\begin{equation*}
K=\sum_{j\in \lA} \dil{\vsig_j}{2^j}.
\end{equation*}
Defining,
\begin{equation}\label{EqnDefnTj}
T_j f\q( x\w) = \psi_1\q( x\w) \int f\q( \gamma_t\q( x\w) \w) \psi_2\q( \gamma_t\q( x\w)\w) \kappa\q( t,x\w)\dil{\vsig_j}{2^j}\q( t\w) \: dt,
\end{equation}
we wish to study the operator,
\begin{equation}\label{EqnCommentsSumOp}
\sum_{j\in \lA} T_j.
\end{equation}

The goal is to show that \eqref{EqnCommentsSumOp} converges in the
strong operator topology, as bounded operators on $L^2$.
In fact, we will show that the family $\q\{T_j\w\}_{j\in \sA}$
is an almost orthogonal family, and the result will then
follow from the Cotlar-Stein lemma.
More specifically, we will show,
\begin{equation}\label{EqnCommentsAlmostOrtho}
\LpOpN{2}{T_k^{*} T_j}, \LpOpN{2}{T_j T_k^{*}}\lesssim 2^{-\epsilon\q|j-k\w|},
\end{equation}
for some $\epsilon>0$, for every $j,k\in \lA$.
Theorem \ref{ThmMainThmSecondPass} follows directly from \eqref{EqnCommentsAlmostOrtho}.
In this section, we focus just on
\begin{equation}\label{EqnCommentsAlmostOrtho2}
\LpOpN{2}{T_k^{*} T_j}\lesssim 2^{-\epsilon\q|j-k\w|},
\end{equation}
the other inequality being similar.  In what follows, we let 
$\epsilon>0$ be a number that may change from line to line,
but will always be independent of $j,k\in \lA$.

\eqref{EqnCommentsAlmostOrtho2} follows directly from the inequality,
\begin{equation}\label{EqnCommentsAlmostOrtho3}
\LpOpN{2}{\q( T_j^{*} T_k T_k^{*} T_j\w)^{n+1}}\lesssim 2^{-\epsilon\q|j-k\w|},
\end{equation}
where $n$ is the dimension of the $x$-space.
One has the trivial inequality,
\begin{equation}\label{EqnCommentsAlmostOrtho4}
\LpOpN{1}{\q( T_j^{*} T_k T_k^{*} T_j\w)^{n+1}}\lesssim 1, 
\end{equation}
and so interpolation with \eqref{EqnCommentsAlmostOrtho4}
shows that to prove \eqref{EqnCommentsAlmostOrtho3} we need only
prove,
\begin{equation}\label{EqnCommentsAlmostOrtho5}
\LpOpN{\infty}{\q( T_j^{*} T_k T_k^{*} T_j\w)^{n+1}}\lesssim 2^{-\epsilon\q|j-k\w|}.
\end{equation}

We will see, for a fixed point $x_0$,
\begin{equation}\label{EqnCommentsAlmostOrtho6}
\q( T_j^{*} T_k T_k^{*} T_j\w)^{n+1} f\q( x_0\w)
\end{equation}
depends only on the values of $f$ on a small neighborhood, $U=U\q( j,k,x_0\w)$, of $x_0$ (later in the paper,
we will have to be precise about how small a neighborhood).\footnote{If the
vector fields $X_1,\ldots, X_q$ do not span the tangent space
at $x_0$, $U$ would instead be a small neighborhood of $x_0$ on the
leaf passing through $x_0$, generated by $X_1,\ldots, X_q$.}
Because of this, given $j$, $k$, and $x_0$, we will be able to construct
a diffeomorphism,
\begin{equation*}
\Phi=\Phi_{j,k,x_0}: B^{n}\q( \eta\w) \rightarrow U,
\end{equation*}
with $\Phi\q( 0\w) = x_0$ and $\eta$ is independent of $j$, $k$, and $x_0$.  
The point of this diffeomorphism will be to ``rescale'' 
\eqref{EqnCommentsAlmostOrtho5} near the point $x_0$ so that it
will follow from elementary considerations.

Let $\Phi^{\#} f \q( u\w) = f\q( \Phi\q( u\w)\w)$.
Then, to prove \eqref{EqnCommentsAlmostOrtho5}, it suffices to show,
\begin{equation}\label{EqnCommentsAlmostOrtho7}
\q| \Phi^{\#} \q( T_j^{*} T_k T_k^{*} T_j\w)^{n+1} \q( \Phi^{\#} \w)^{-1} g\q( 0\w) \w| \lesssim 2^{-\epsilon \q|j-k\w|} \LpN{\infty}{g},
\end{equation}
where $g=\Phi^{\#} f$.  Here, the implicit constants in \eqref{EqnCommentsAlmostOrtho7}
cannot depend on the point $x_0$.

\begin{rmk}
Later in the paper, we will not explicitly conjugate our operators by
the pullback $\Phi^{\#}$.  However, the way in which we proceed is equivalent
to the above.
\end{rmk}

The map $\Phi$ will be defined in such a way that the effect
of conjugating $T_j$ by $\Phi^{\#}$ leaves one with an operator
of the same form as $T_j$ (though, with a different $\gamma$) but with $j$ replaced by
$j'$, where $j'=j-j\wedge k$.  With an abuse of notation,
we call this new operator $T_{j'}$.  Similarly, $T_k$ is replaced by
$T_{k'}$ where $k'=k-j\wedge k$.\footnote{Moreover, even in the case
when $X_1,\ldots, X_q$ do not span the tangent space, the analogous vector fields
associated to $T_{j'}$ and $T_{k'}$ 
do span the tangent space (uniformly in any relevant parameters).}

In this way, we can reduce
\eqref{EqnCommentsAlmostOrtho5} to the case when $j\wedge k=0$ (provided
one has good enough control of the implicit constant
in \eqref{EqnCommentsAlmostOrtho5}).
More precisely, if we let,\footnote{$\Phi$ will be defined in such a
way that the map $\theta$ defined above is $C^\infty$, uniformly
in any relevant parameters.}
\begin{equation*}
\theta_t\q( u \w) = \Phi^{-1}\circ \gamma_{2^{-j\wedge k} t}\circ \Phi\q( u\w),
\end{equation*}
then we are interested in the operator,
\begin{equation*}
T_{j'} g \q( u\w) = \int g\q( \theta_{2^{-j'}t}\q( u\w)\w) \kappa\q( t,u\w) \vsig_j\q( t\w)\: dt,
\end{equation*}
where $\kappa$ is a $C^\infty$ function (uniformly in any relevant parameters)
and we have suppressed $\psi_1$ and $\psi_2$ since we are working locally.
A similar formula holds for $T_{k'}$.

The diffeomorphism $\Phi$ will be defined in such a way that if we
formally write,
\begin{equation*}
\theta_t\q( u\w) \sim e^{\sum_{0<\q|\alpha\w|} t^{\alpha} Y_\alpha}u,
\end{equation*}
(see Section \ref{SectionCNSW}) then the vector fields
\begin{equation}\label{EqnCommentsHormander1}
\q\{Y_\alpha: \deg\q( \alpha\w)\text{ is nonzero in precisely one component}\w\}
\end{equation}
satisfy H\"ormander's condition uniformly in any relevant parameters.
Rewriting \eqref{EqnCommentsHormander1} using the fact that
$j'\wedge k'=0$, we see,
\begin{equation}\label{EqnCommentsHormander2}
\q\{Y_{\alpha} : j'\cdot\deg\q( \alpha\w)=0 \text{ or } k'\cdot \deg\q( \alpha\w) =0\w\}
\end{equation}
satisfy H\"ormander's condition uniformly in any relevant parameters.
Hence, writing,
\begin{equation}\label{EqnCommentsHormader3}
\begin{split}
\theta_{2^{-j'}t}\q( u\w) &\sim e^{\sum_{0<\q|\alpha\w|} t^{\alpha} 2^{-j'\cdot \deg\q( \alpha\w)} Y_\alpha} u,\\
\theta_{2^{-k'}t}\q( u\w) &\sim e^{\sum_{0<\q|\alpha\w|} t^{\alpha} 2^{-k'\cdot \deg\q( \alpha\w)} Y_\alpha} u,
\end{split}
\end{equation}
we see the set of vector fields in \eqref{EqnCommentsHormader3} which
are not scaled satisfy H\"ormander's condition.
Because of this, it will be possible to apply the methods
of \cite{ChristNagelSteinWaingerSingularAndMaximalRadonTransforms}
to the operator,
\begin{equation*}
T_{j'}^{*}T_{k'}T_{k'}^{*}T_{j'},
\end{equation*}
to prove \eqref{EqnCommentsAlmostOrtho7} and complete the proof.

\begin{rmk}
In short, the above says that the map $\theta_{2^{-k'}t_1}\circ\theta_{2^{-j'}t_2}$ satisfies
$\cG$ uniformly in any relevant parameters.  When we
turn to the rigorous proof, it will be more convenient
to use the (equivalent) fact that as $j,k$ vary, the set of all
such $\theta_{2^{-k'}t_1}\circ \theta_{2^{-j't_2}}$ satisfies
$\cZu$.  We will then be able to apply Theorem \ref{ThmUnifEquivCond}.
\end{rmk}

\begin{rmk}\label{RmkCommentsPurePowers}
Note that we have used, in an essential way, that we are only
considering those vector fields $Y_\alpha$ such that
$j'\cdot \deg\q( \alpha\w)=0$ or $k'\cdot \deg\q( \alpha\w)=0$.  This highlights the difference
between the pure powers and non-pure powers
discussed in Section \ref{SectionSpecialCase}.
\end{rmk}

We close this section by discussing some difficulties that
arise in the multi-parameter situation, that
did not arise in the single parameter
situation of 
\cite{ChristNagelSteinWaingerSingularAndMaximalRadonTransforms}.

In \cite{ChristNagelSteinWaingerSingularAndMaximalRadonTransforms}
the Rothschild-Stein lifting procedure
(see \cite{RothschildSteinHypoellipticDifferentialOperatorsAndNilpotentGroups})
was used to ``lift'' the problem to a higher dimensional
setting; in fact the problem was lifted to a high dimensional stratified
Lie group.  In this higher dimensional setting, the scaling
map $\Phi$ had a very simple form:  it was given by the standard
dilations on this stratified Lie group.
In the multi-parameter setting, this lifting argument creates
more difficulties than it alleviates.  In particular, it is unclear
how to consider the non-pure powers, nor how to naturally
deal with the multi-parameter dilations.
Because of this, we must work directly with a more complicated scaling
map $\Phi$.  This map was defined and studied in \cite{StreetMultiParameterCCBalls},
and we review the relevant theory in Section \ref{SectionCCII}.
As a consequence, the lifting procedure is not necessary to prove the
results of \cite{ChristNagelSteinWaingerSingularAndMaximalRadonTransforms}.
The single-parameter version of the results in
\cite{StreetMultiParameterCCBalls}
were first obtained in \cite{NagelSteinWaingerBallsAndMetricsDefinedByVectorFields}.
Using the scaling maps from \cite{NagelSteinWaingerBallsAndMetricsDefinedByVectorFields},
one can (with some reordering of the proof) recreate the entire theory
of \cite{ChristNagelSteinWaingerSingularAndMaximalRadonTransforms}
without resorting to the lifting procedure.

In the single-parameter case, one is given two numbers $j,k\in \N$,
and one wishes to show, for instance,
\begin{equation*}
\LpOpN{2}{T_k^{*}T_j}\lesssim 2^{-\epsilon\q|j-k\w|}.
\end{equation*}
When, for instance, $k\leq j$, it suffices to show,
\begin{equation*}
\LpOpN{2}{\q(T_kT_k^{*}\w)^M T_j}\lesssim 2^{-\epsilon\q|j-k\w|};
\end{equation*}
for some large, but fixed, $M$.  In the multi-parameter situation,
however, it could be that some coordinates of
$j$ are greater than those of $k$, while the opposite
is true for other coordinates.  Because of this, we must instead
show,
\begin{equation*}
\LpOpN{2}{\q(T_j^{*} T_kT_k^{*}T_j\w)^M}\lesssim 2^{-\epsilon\q|j-k\w|}.
\end{equation*}
This turns out to be more of a notational difficulty than anything else.

	\subsection{Diffeomorphism invariance}
	The assumptions of Theorem \ref{ThmMainThmSecondPass} are
invariant under diffeomorphisms.
More precisely, if we let $T$ be the operator in Theorem \ref{ThmMainThmSecondPass},
and if we let $\Psi:\Omega\rightarrow \Omegat$ be a diffeomorphism,
then the operator $\q(\Psi^{\#}\w)^{-1}T\Psi^{\#}$ satisfies
the assumptions of Theorem \ref{ThmMainThmSecondPass}.
This is straightforward to check.

It turns out that more is true.  
Denote by $\Tt$ the map $\q(\Psi^{\#}\w)^{-1}T\Psi^{\#}$,
and by $\Tt_j$ the map $\q(\Psi^{\#}\w)^{-1}T_j\Psi^{\#}$,
where $T_j$ is as in Section \ref{SectionProofOutline}.
Note that $\Tt=\sum_{j\in \lA} \Tt_j$, and in fact if we had
applied the proof in Section \ref{SectionProofOutline} to $\Tt$,
this is the decomposition of $\Tt$ we would have used.
Fix $j$, $k$, and $x_0$
as in Section \ref{SectionProofOutline}.
Let $\Phi$ be the scaling map associated to $T_j, T_k$ at $x_0$,
and let $\Phit$ denote the scaling map associated to $\Tt_j, \Tt_k$
at $\Psi\q(x_0\w)$.  It is not hard to see, from the definition
of $\Phi$ in Section \ref{SectionCCII} that,
\begin{equation*}
\Phit = \Psi\circ \Phi.
\end{equation*}

Thus, we have,
\begin{equation*}
\Phit^{\#} \Tt_j \q(\Phit^{\#}\w)^{-1} = \Phi^{\#} \Psi^{\#} \q( \Psi^{\#}\w)^{-1} T_j \Psi^{\#}\q(\Psi^{\#} \w)^{-1} \q(\Phi^{\#}\w)^{-1} = \Phi^{\#} T_j \q(\Phi^{\#}\w)^{-1};
\end{equation*}
with a similar result for $T_k$ and $\Tt_k$.  Thus, one obtains
the {\it same} operators when pulling back $T_j$ and $T_k$
via $\Phi$ as when pulling back $\Tt_j$ and $\Tt_k$ via $\Phit$.
Hence, not only are the assumptions of Theorem \ref{ThmMainThmSecondPass}
invariant under diffeomorphisms,
but a large part of the proof actually remains completely unchanged when
conjugated by a diffeomorphism.

\section{Multi-parameter Carnot-Carath\'eodory geometry revisited}\label{SectionCCII}
	In this section, we present the results that allow us to deal
with Carnot-Carath\'eodory geometry.  The results we
outline here are contained in Section 4 of
\cite{StreetMultiParameterCCBalls}.
The heart of this theory is the ability to ``rescale'' vector fields.
This rescaling is obtained by 
pulling back via
a particular diffeomorphism,
which will be denoted by $\Phi$ in what follows.

Before we can enter into details, we must explain the connection
between multi-parameter balls and single-parameter balls.
We assume we are given $C^\infty$ vector fields
$X_1,\ldots, X_q$ with associated formal degrees $0\ne d_1,\ldots,d_q\in \q[0,\infty\w)^\nu$.  Here, $\nu\in \N$ is the number of parameters.
Given the multi-parameter degrees, we obtain corresponding
single parameter degrees, which we denote
by $\sd$ and are defined by $\q( \sd\w)_j:=\sum_{\mu=1}^\nu d_j^\mu=\q|d_j\w|_1$.
Let $\delta\in \q[0,\infty\w)^\nu$, and suppose we wish to study
the ball
\begin{equation*}
\B{X}{d}{x_0}{\delta}.
\end{equation*}
Decompose $\delta=\delta_0\delta_1$ where $\delta_0\in \q[0,\infty\w)$
and $\delta_1\in\q[0,\infty\w)^\nu$ (of course this decomposition
is not unique).
Then, directly from the definition, we obtain:
\begin{equation*}
\B{X}{d}{x_0}{\delta}=\B{\delta_1 X}{\sd}{x_0}{\delta_0}=\B{\delta X}{\sd}{x_0}{1}.
\end{equation*}
Thus, studying a ball of radius $\delta$ corresponding to $\q( X,d\w)$
is the same as studying a ball of radius $1$ corresponding
to $\q( \delta X, \sd\w)$.
For this reason, taking $\K$ and $\sA$ as in Section \ref{SectionCC}
and assuming $\q( X,d\w)$ satisfies $\sD\q( \K,\sA\w)$,
we will fix $x_0\in \K$ and $\delta\in \sA$ and study balls
of radius $\approx 1$ centered at $x_0$ corresponding to the
vector fields with single-parameter formal degrees $\q( \delta X, \sd\w)$.
In what follows, it will be important that all of the implicit constants
are independent of $x_0\in \K$ and $\delta\in \sA$.

We now turn to stating a theorem about a list of $C^\infty$ vector fields
$Z_1,\ldots, Z_q$ defined on an open set $\Omega\subseteq \R^n$, with associated single parameter formal degrees
$\dt_1,\ldots, \dt_q\in \q(0,\infty\w)$.  The special case
we are interested in is the case when
$\q( Z,\dt\w) = \q( \delta X, \sd\w)$; i.e., when
$Z_j=\delta^{d_j} X_j$ and $\dt_j=\q|d_j\w|_1$.

Fix $x_0\in \Omega$ and $1\geq \xi>0$.\footnote{In our primary example,
one takes $x_0\in \K$ and $\xi$ as in $\sD\q( \K,\sA\w)$.}
Let $n_0=\dspan{Z_1\q( x_0\w), \ldots, Z_q\q( x_0\w)}$.
For $J=\q( j_1,\ldots, j_{n_0}\w)\in \sI{n_0}{q}$, let $Z_J$ denote the list of vector fields
$Z_{j_1},\ldots, Z_{j_{n_0}}$.
Fix $J_0\in \sI{n_0}{q}$ such that
\begin{equation*}
\q| \det_{n_0\times n_0} Z_{J_0}\q( x_0\w)\w|_\infty = \q|\det_{n_0\times n_0} Z\q( x_0\w) \w|_\infty,
\end{equation*}
where we have identified $Z\q( x_0\w)$ with the $n\times q$ matrix
whose columns are given by $Z_1\q( x_0\w),\ldots, Z_q\q( x_0\w)$ and
similarly for $Z_{J_0}\q( x_0\w)$.
We assume $\q(Z,\dt\w)$ satisfies $\sC\q( x_0,\xi,\Omega\w)$.
In addition we suppose that there are functions $c_{i,j}^k$ on
$\B{Z}{\dt}{x_0}{\xi}$ such that
\begin{equation*}
\q[Z_i,Z_j\w]=\sum c_{i,j}^k Z_k, \text{ on }\B{Z}{\dt}{x_0}{\xi}.
\end{equation*}
We assume that:
\begin{itemize}
\item $\CjN{Z_j}{m}{\B{Z}{\dt}{x_0}{\xi}}<\infty$ for every $m$.
\item $\sum_{\q|\alpha\w|\leq m} \CjN{ Z^\alpha c_{i,j}^k}{0}{\B{Z}{\dt}{x_0}{\xi}}<\infty,$
for every $m$ and every $i,j,k$.
\end{itemize}
We say that $C$ is an $m$-admissible constant if $C$ can be chosen to
depend only on upper bounds for the above two quantities (for that particular 
choice of $m$),
$m$,
upper and lower bounds for $\dt_1,\ldots, \dt_{q}$, an
upper bound for $n$ and $q$, and a lower bound for
$\xi$.
Note that, in our primary example $\q( Z,\dt\w)=\q( \delta X, \sd\w)$, 
$m$-admissible constants can be chosen independent of
$x_0\in \K$ and $\delta\in \sA$.
We write $A\lesssim_m B$ if $A\leq C B$, where $C$ is an $m$-admissible
constant, and we write $A\approx_m B$ if $A\lesssim_m B$ and $B\lesssim_m A$.
Finally, we say $\tau=\tau\q( \kappa\w)$ is an $m$-admissible constant
if $\tau$ can be chosen to depend on all the parameters an $m$
admissible constant may depend on, and $\tau$ may also depend on $\kappa$.
The next result is contained in Section 4 of
\cite{StreetMultiParameterCCBalls}.

\begin{thm}\label{ThmMainCCThm}
There exist $2$-admissible constants $\eta_1,\xi_1>0$ such that
if the map $\Phi:B^{n_0}\q( \eta_1\w)\rightarrow \B{Z}{\dt}{x_0}{\xi}$
is defined by
\begin{equation*}
\Phi\q( u\w) = e^{u\cdot Z_{J_0}} x_0,
\end{equation*}
we have
\begin{itemize}
\item $\Phi:B^{n_0}\q( \eta_1\w) \rightarrow  \B{Z}{\dt}{x_0}{\xi}$ is injective.
\item $\B{Z}{\dt}{x_0}{\xi_1}\subseteq \Phi\q( B^{n_0}\q( \eta_1\w)\w)$.
\end{itemize}
Furthermore, if we let $Y_j$ be the pullback of $Z_j$ under the map
$\Phi$, then we have, for $m\geq 0$,
\begin{equation}\label{EqnRescaledSmooth}
\CjN{Y_j}{m}{B^{n_0}\q( \eta_1\w)}  \lesssim_{m\vee 2} 1,
\end{equation}
\begin{equation}\label{EqnRescaledCjN}
\CjN{f}{m}{B^{n_0}\q(\eta_1\w)} \approx_{\q(m-1\w)\vee 2} \sum_{\q|\alpha\w|\leq m} \CjN{Y^\alpha f}{0}{B^{n_0}\q( \eta_1\w)}.
\end{equation}
Finally,
\begin{equation}\label{EqnRescaledSpan}
\q|\det_{n_0\times n_0} Y\q( u\w)\w|\approx_{2} 1,\quad \forall u\in B^{n_0}\q( \eta_1\w).
\end{equation}
\end{thm}

Note that, in light of \eqref{EqnRescaledSmooth} and \eqref{EqnRescaledSpan},
pulling back by the map $\Phi$ allows us to rescale the vector fields
$Z$ in such a way that the rescaled vector fields, $Y$, are smooth
and span the tangent space (uniformly in any relevant parameters).
We will also need the following 
technical result.

\begin{prop}\label{PropExtraCCStuff}
Suppose $\xi_2,\eta_2>0$ are given.  Then there exist $2$-admissible constants
$\eta'=\eta'\q( \xi_2\w)>0$, $\xi'=\xi'\q( \eta_2\w)>0$ such that,
\begin{equation*}
\Phi\q( B^{n_0}\q( \eta'\w)\w) \subseteq \B{Z}{\dt}{x_0}{\xi_2},
\end{equation*}
\begin{equation*}
\B{Z}{\dt}{x_0}{\xi'}\subseteq \Phi\q( B^{n_0}\q( \eta_2\w)\w).
\end{equation*}
\end{prop}
\begin{proof}
The existence of $\eta'$ can be seen by 
applying Theorem \ref{ThmMainCCThm} with $\xi$ replaced by $\xi\wedge \xi_2$.
The existence of $\xi'$ can be shown by
combining the proof
of Proposition 3.21 of \cite{StreetMultiParameterCCBalls} with
the proof of Proposition 4.16 of \cite{StreetMultiParameterCCBalls}.
\end{proof}

\begin{rmk}
With a slight abuse of notation, when we say $m$-admissible constant,
where $m<2$, we will take that to mean a $2$-admissible constant.
Using this new notation, the $\vee$ in \eqref{EqnRescaledSmooth}
and \eqref{EqnRescaledCjN} may be removed.
\end{rmk}

\begin{rmk}
One way to think about Theorem \ref{ThmMainCCThm} is from the
perspective of the classical Frobenius theorem.  Indeed,
the main assumption 
in Theorem \ref{ThmMainCCThm}
is essentially that the distribution generated by 
by $Z_1,\ldots, Z_q$ 
is involutive.
One can view the map $\Phi$ as a coordinate chart on the leaf
passing through $x_0$ generated by $Z_1,\ldots, Z_q$.
Theorem \ref{ThmMainCCThm} gives quantitative control
of this coordinate chart.  This is discussed
in more detail in \cite{StreetMultiParameterCCBalls}.
\end{rmk}

\begin{rmk}
It is not hard to see that the single-parameter formal degrees
$\dt$ do not play an essential role in the above
(see Remark 3.3 of \cite{StreetMultiParameterCCBalls}).
In fact, one could state Theorem \ref{ThmMainCCThm},
taking all the formals degrees $\dt_j=1$ and that would
be sufficient for our purposes.  Moreover, in every place
we use the single-parameter formal degrees $\dt$,
they are inessential.
We have stated the result as above, though, to allow
us to transfer seamlessly between the vector fields
$\q( X,d\w)$ and $\q( Z,\dt\w)$, without any hand-waving
about the formal degrees.
\end{rmk}

	\subsection{Control of vector fields}\label{SectionControlOfVFs}
	In this section, we introduce a notion (closely
related to Definition \ref{DefnControlEveryScale}) to discuss
when the list of vector fields with formal
degrees ``controls'' a vector field with a formal degree
$\q( X_0,d_0\w)$.
Informally, this means that if we were to add $\q(X_0,d_0\w)$
to the list $\q( X,d\w)$, we would not ``get anything new.''
In particular, under the conditions we will define, we will 
not significantly increase
the size of the balls $\B{X}{d}{x_0}{\delta}$ (we would, instead,
obtain comparable balls), and Theorem \ref{ThmMainCCThm}
will hold with the same choice of the map $\Phi$ (when Theorem
\ref{ThmMainCCThm} is applied
to $\q( \delta X,\sd\w)$).
Many of the results discussed in this section are discussed in 
more detail, and proved, in Sections 4.1 and 5.3 of
\cite{StreetMultiParameterCCBalls}, and we refer the
reader there for more details.

Following in the philosophy of Section \ref{SectionCCII}, we will
state our definitions and results for
single parameter vector fields $\q( Z,\dt\w)$ near a fixed
point $x_0\in \Omega$, where one should
think of the case $\q( Z,\dt\w) = \q( \delta X, \sd\w)$,
where
$\delta\in \sA$ and $x_0\in \K$ (and all of the results
and definitions hold uniformly in $\delta$ and $x_0$).
We maintain all the same assumptions on $\q( Z,\dt\w)$
(and associated notation) as in Section \ref{SectionCCII}.
In particular, Theorem \ref{ThmMainCCThm} applies and
we obtain $2$-admissible constants $\eta_1$ and $\xi_1$
and a map $\Phi$ as in Theorem \ref{ThmMainCCThm}.

Suppose we are given a $C^\infty$ vector field $Z_0$ with
an associated single parameter formal degree $\dt_0$.
Let $\q( Z',\dt'\w)$ denote the list of vector fields with
formal degrees
\begin{equation*}
\q( Z_0,\dt_0\w),\q( Z_1,\dt_1\w),\ldots, \q( Z_q,\dt_q\w):
\end{equation*}
the list of vector fields $\q( Z,\dt\w)$ with $\q( Z_0,\dt_0\w)$
added.
We introduce the following two conditions on $\q( Z_0,\dt_0\w)$
which turn out to be equivalent.
All parameters that follow are assumed to be strictly positive real numbers:
\begin{enumerate}
\item $\sP_1\q( \tau_1,\q\{\sigma_1^m\w\}_{m\in \N}\w)$ ($\tau_1\leq \xi_1$):
There exist $c_j\in C^0\q( \B{Z}{\dt}{x_0}{\tau_1}\w)$ such that: 
\begin{itemize}
\item $Z_0=\sum_{j=1}^q c_j Z_j,$ on $\B{Z}{\dt}{x_0}{\tau_1}$.
\item $\sum_{\q|\alpha\w|\leq m} \CjN{\q( Z\w)^\alpha c_j}{0}{\B{Z}{d}{x_0}{\tau_1}}\leq \sigma_1^m$.
\end{itemize}
\item $\sP_2\q( \eta_2,\q\{\sigma_2^m\w\}_{m\in \N}\w)$ ($\eta_2\leq \eta_1$):  $Z_0$
is tangent to the leaf passing through $x_0$ generated by $Z_1,\ldots, Z_q$,
and moreover, if we let $Y_0$ be the pullback of $Z_0$ via the map
$\Phi$ from Theorem \ref{ThmMainCCThm}, we have,
\begin{equation*}
\CjN{Y_0}{m}{B^{n_0}\q( \eta_2\w)} \leq \sigma_2^m.
\end{equation*}
\end{enumerate}

\begin{prop}\label{PropsP1equivsP2}
$\sP_1\Leftrightarrow \sP_2$ in the following sense.
\begin{itemize}
\item $\sP_1\q( \tau_1,\q\{\sigma_1^m\w\}_{m\in \N}\w)\Rightarrow$ there
exists a $2$-admissible $\eta_2=\eta_2\q( \tau_1\w)>0$ and $m$-admissible
constants $\sigma_2^m=\sigma_2^m\q( \sigma_1^m \w)$ such that
$\sP_2\q( \eta_2, \q\{\sigma_2^m\w\}_{m\in \N}\w)$ holds.
\item $\sP_2\q( \eta_2,\q\{\sigma_2^m\w\}_{m\in \N} \w)\Rightarrow$ there
exists a $2$-admissible constant $\tau_1=\tau_1\q( \eta_2\w)>0$
and $m$-admissible constants $\sigma_1^m=\sigma_1^m\q( \sigma_2^m\w)$
such that $\sP_1\q( \tau_1,\q\{\sigma_1^m\w\}_{m\in \N}\w)$ holds.
\end{itemize}
Moreover, in the case that $\sP_1\q(\tau_1,\q\{\sigma_1^m\w\}_{m\in \N} \w)$ holds,
we have the following.
\begin{itemize}
\item There is a $2$-admissible constant $\tau'=\tau'\q( \dt_0, \tau_1,\sigma_1^2\w)$, such that
\begin{equation*}
\B{Z}{\dt}{x_0}{\tau'}\subseteq \B{Z'}{\dt'}{x_0}{\tau'}\subseteq \B{Z}{\dt}{x_0}{\tau_1}.
\end{equation*}
\item The same map $\Phi$ satisfies all the conclusions
of Theorem \ref{ThmMainCCThm} with $\q( Z,\dt\w)$ replaced
by $\q( Z',\dt'\w)$.
\end{itemize}
\end{prop}
\begin{proof}
$\sP_1\Rightarrow \sP_2$ follows just as in Proposition 4.19 of
\cite{StreetMultiParameterCCBalls} (here $\sP_1$ is related
to $\sP_3$ in \cite{StreetMultiParameterCCBalls}).

We now turn to $\sP_2\Rightarrow \sP_1$.  
Take $\tau_1=\tau_1\q( \eta_2\w)$ to be a $2$-admissible constant
so small that,
\begin{equation*}
\B{Z}{\dt}{x_0}{\tau_1}\subseteq \Phi\q( B^{n_0}\q( \eta_2\w)\w),
\end{equation*}
by applying Proposition \ref{PropExtraCCStuff}.
Let
$Y_1,\ldots, Y_q$ be the pullbacks of $Z_1,\ldots, Z_q$ via
the map $\Phi$. 
Using \eqref{EqnRescaledSmooth} and \eqref{EqnRescaledSpan}, we may write,
\begin{equation}\label{EqnControlVectorPullBack}
Y_0 =\sum_{j=1}^q c_j Y_j,\text{ on }B^{n_0}\q( \eta_2\w),
\end{equation}
with,
\begin{equation}\label{EqnControlVectorPullBack2}
\sum_{\q|\alpha\w|\leq m} \CjN{\q( Y\w)^\alpha c_j}{0}{B^{n_0}\q( \eta_2\w)}\lesssim_m
\CjN{c_j}{m}{B^{n_0}\q( \eta_2\w)}\lesssim \sigma_2^m.
\end{equation}
Pushing \eqref{EqnControlVectorPullBack} forward via the map
$\Phi$ and combining it with the push-forward of \eqref{EqnControlVectorPullBack2}
yields $\sP_1$.

We now turn to the second part of the proposition:  when $\sP_1$ (or,
equivalently, $\sP_2$) holds, the above two conclusions hold.
The first follows just as in Proposition 4.18 of
\cite{StreetMultiParameterCCBalls}.
For the second, note that in Theorem \ref{ThmMainCCThm},
we chose $J_0$ such that
\begin{equation*}
\q| \det_{n_0\times n_0} Z_{J_0}\q( x_0\w) \w|_\infty = \q|\det_{n_0\times n_0} Z\q( x_0\w) \w|_\infty;
\end{equation*}
however, the theory from \cite{StreetMultiParameterCCBalls}
shows that we only need
\begin{equation*}
\q| \det_{n_0\times n_0} Z_{J_0}\q( x_0\w) \w|_\infty \gtrsim \q|\det_{n_0\times n_0} Z\q( x_0\w) \w|_\infty.
\end{equation*}
Thus to show that the same map $\Phi$ can be used for $\q( Z',\dt'\w)$,
we need only show,
\begin{equation}\label{EqnControlVectorsToShowDet}
\begin{split}
&n_0 = \dspan{Z_0\q( x_0\w),\ldots, Z_q\q( x_0\w)},\\
&\q|\det_{n_0\times n_0} Z\q( x_0\w)\w|_\infty \gtrsim \q|\det_{n_0\times n_0} Z'\q( x_0\w)\w|_\infty.
\end{split}
\end{equation}
\eqref{EqnControlVectorsToShowDet} follows from
Theorem 4.17 of \cite{StreetMultiParameterCCBalls} (the
$\sP_3^0\Rightarrow \sP_1^0$ part of Theorem 4.17).
\end{proof}

\begin{defn}
We say $\q( Z,\dt\w)$ controls $\q( Z_0,\dt_0\w)$ at the unit scale near
$x_0$ if either of the above two equivalent conditions ($\sP_1$ or $\sP_2$)
holds.
\end{defn}

\begin{defn}
Let $\q( Z,\dt\w)=\q( Z_1,\dt_1\w),\ldots, \q( Z_{q},\dt_q\w)$ and
$\q( Z',\dt'\w)=\q( Z_1',\dt_1'\w),\ldots, \q(Z_{q'}',\dt_{q'}'\w)$
be two lists of vector fields with single parameter formal
degrees as above.  We say $\q( Z,\dt\w)$ controls $\q( Z',\dt'\w)$
at the unit scale near $x_0$
if $\q( Z,\dt\w)$ controls $\q( Z_j', \dt_j'\w)$ at
the unit scale near $x_0$ for every
$1\leq j\leq q'$.  We say $\q( Z,\dt\w)$ and $\q( Z',\dt'\w)$
are equivalent at the unit scale near $x_0$ if
$\q( Z,\dt\w)$ controls $\q( Z',\dt'\w)$ at the unit scale near $x_0$
and $\q( Z',\dt'\w)$ controls $\q( Z,\dt\w)$ at the unit scale near $x_0$.
\end{defn}

\begin{prop}
If $\q( Z,\dt\w)$ and $\q( Z',\dt'\w)$ are equivalent at the unit
scale near $x_0$, and if $\Phi$ is the map obtained
when applying Theorem \ref{ThmMainCCThm} to the list $\q( Z,\dt\w)$,
then the same map $\Phi$ satisfies all the conclusions
of Theorem \ref{ThmMainCCThm} with $\q( Z',\dt'\w)$ in place
of $\q( Z,\dt\w)$ (with possibly different constants).
\end{prop}
\begin{proof}
The proof is an easy addition to Proposition \ref{PropsP1equivsP2},
and we leave the details to the reader.
\end{proof}

\begin{defn}\label{DefnControlVectorFieldEveryScale}
Let $\K\Subset\Omega$ and $\sA\subseteq\q[0,1\w]^\nu$ be as in 
Section \ref{SectionCC}, and let $\q( X,d\w)$ be a list
of $C^\infty$ vector fields with $\nu$-parameter formal degrees
as in Section \ref{SectionCC}, satisfying $\sD\q( \K, \sA\w)$.
Let $X_0$ be another $C^\infty$ vector field with an associated
formal degree $0\ne d_0\in \q[0,\infty\w)^\nu$.
We say $\q( X,d\w)$ {\it controls} $\q( X_0,d_0\w)$ if
for every $\delta\in \sA$, $x_0\in \K$, $\q( \delta X,\sd\w)$
controls $\q( \delta^{d_0} X_0, \q|d_0\w|_1\w)$ at the unit sale
near $x_0$ in such a way that the 
parameters of $\sP_1$ (equivalently $\sP_2$)
can be chosen to be independent of $x_0\in \sK$, $\delta\in \sA$.
\end{defn}

\begin{rmk}
Note that $\sD\q( \K, \sA\w)$ is equivalent to saying that
$\q( X,d\w)$ controls $\q( \q[X_i,X_j\w], d_i+d_j\w)$, for
every $1\leq i,j\leq q$.
\end{rmk}

Definition \ref{DefnControlVectorFieldEveryScale} is closely related
to Definition \ref{DefnControlEveryScale}.  Indeed, we have the following
proposition, whose proof we defer to Section \ref{SectionProofOfPropControlVectsControlCurves}.

\begin{prop}\label{PropControlVectsControlCurve}
Let $\q( X,d\w)$ be as in Definition \ref{DefnControlVectorFieldEveryScale}.
Define,
\begin{equation*}
\gamma_t\q( x\w) = e^{\sum_{0<\q|\alpha\w|\leq L} t^\alpha V_\alpha}x,
\end{equation*}
where each $V_\alpha$ is a $C^\infty$ vector field.  Then,
$\gamma$ is controlled by $\q( X,d\w)$ if and only if
$\q( V_\alpha, \deg\q(\alpha\w)\w)$ is controlled by
$\q( X,d\w)$ for every $\alpha$.
\end{prop}

\begin{defn}\label{DefnEquivLists}
Let $\q( X,d\w)=\q( X_1,d_1\w),\ldots, \q( X_1,d_q\w)$ and $\q( X',d'\w)=\q( X_1',d_1'\w),\ldots, \q(X_{q'}',d_{q'}'\w)$ be two lists of vector fields
with $\nu$-parameter formal degrees satisfying $\sD\q( \K,\sA\w)$.
We say $\q( X,d\w)$ {\it controls} $\q( X',d'\w)$ if $\q( X,d\w)$
controls $\q( X_j', d_j'\w)$ for every $1\leq j\leq q'$.
We say that $\q( X,d\w)$ and $\q( X',d'\w)$ are {\it equivalent}
if $\q( X,d\w)$ controls $\q( X',d'\w)$ and $\q( X',d'\w)$ controls
$\q( X,d\w)$.
\end{defn}

\begin{rmk}\label{RmkEquivLists}
In Section \ref{SectionCurves}, we singled out a list of vector fields
with formal degrees $\q( X,d\w)$.  Throughout, we will use these
vector fields to study $\gamma$.  In everything that
follows, and because of the results in this section, it would
work just as well to replace $\q( X,d\w)$ with an equivalent
list $\q( X',d'\w)$ in our applications of Theorem \ref{ThmMainCCThm}.
\end{rmk}

	\subsection{More on $\q( X,d\w)$}\label{SectionMoreOnXd}
	Our assumptions in Section \ref{SectionCurves}
above might make it seem that we have made some
choices which affected the list $\q( X,d\w)$.  Specifically,
we supposed that there existed a finite set $\sF\subseteq \sV$
such that $\sF$ generated a finite list $\q( X,d\w)$ and this
list controlled $\gamma$ (here we are taking
the notation from Section \ref{SectionCurves}; in particular,
$\sV$ is given by \eqref{EqnCurvesDefnS}).  One might expect that $\q( X,d\w)$
depends on this choice of $\sF$.  However, this is not the case:
given such a $\sF$ and given $\sF'\subseteq \sV$ finite 
with $\sF\subseteq \sF'$, then any finite list $\q( X',d'\w)$
generated by $\sF'$ is equivalent to $\q( X,d\w)$.\footnote{As
already noted in Definition \ref{DefnGeneratesAFiniteList}, given
$\sF$, all choices of $\q( X,d\w)$ are equivalent.}
This follows from the following lemma, and the
theory in Section \ref{SectionControlOfVFs}.

\begin{lemma}\label{LemmaAllCoefsControled}
Suppose $\q( X,d\w)$ controls $\gamma$.  Then, if $X_{\alpha}$ is
as in \eqref{EqnDefnXalpha},
$\q( X,d\w)$ controls $\q( X_{\alpha}, \deg\q( \alpha\w)\w)$. 
\end{lemma}
\begin{proof}
This follows easily from the definitions.
\end{proof}
As is discussed in Section \ref{SectionControlOfVFs},
in light of Lemma \ref{LemmaAllCoefsControled},
one does not ``get anything new'' by adding $\q( X_{\alpha}, \deg\q( \alpha\w)\w)$
to the list $\q( X,d\w)$.

Because of the above, there is a more direct way to construct the
list $\q( X,d\w)$.  Fix $M$ large.  One considers the set
\begin{equation}\label{EqnDefnsF}
\sF =\bigg\{\q(X_{\alpha}, \deg\q( \alpha\w)\w) : \deg\q( \alpha\w)\text{ is non-zero in only one component}, \q|\alpha\w|\leq M\bigg\}.
\end{equation}
One then recursively assigns formal
degrees to commutators of vector fields of $\sF$ as follows:
if $X_1$ has formal degree $d_1$ and $X_2$ has formal degree $d_2$,
then $\q[X_1,X_2\w]$ is assigned formal degree $d_1+d_2$.
One then takes iterated commutators of the elements of $\sF$,
and we assume that, after taking all iterated commutators
up to order $M$, the resulting finite list of vector fields
with formal degrees
satisfies $\sD\q( \K, \sA\w)$, and controls $\gamma$.
The above discussion shows that if this is true for $M$,
it is true for any larger choice of $M$, and the resulting
finite lists of vector fields with formal
degrees are equivalent.

\section{Surfaces revisited}\label{SectionCurvesII}
In this section, we introduce the technical results necessary
to deal with our assumptions on $\gamma$
in Section \ref{SectionCurves}.
Before we enter into details, we point out the formal connection
between $\gamma$ and $W$.
Recall, from Section \ref{SectionCurves}, $W\q( t,x\w)$
is defined by,
\begin{equation}\label{EqnCurvesIIW}
W\q( t,x\w) = \frac{d}{d\epsilon}\bigg|_{\epsilon=1} \gamma_{\epsilon t}\circ \gamma_t^{-1}\q( x\w).
\end{equation}
From this we see that, given $\gamma$, we obtain a $C^\infty$ vector
field $W\q( t\w)$, depending smoothly on $t$, satisfying $W\q( 0\w) =0$.
The reverse is true as well:

\begin{prop}\label{PropGetGammaFromW}
The map $\gamma\mapsto W$ is a bijective map from smooth functions $\gamma_t$, with $\gamma_0\q( x\w) \equiv x$ (thought of as germs in the $t$ variable), to smooth
vector fields depending on $t$, $W\q( t\w)$, with $W\q( 0\w)=0$ (also thought
of as germs in the $t$ variable).
\end{prop}
\begin{proof}
Let $\omega\q( \epsilon, t, x\w)$ be the (unique\footnote{It is easy
to see that the solution is unique, using that $W\q( 0\w)=0$, $W$ is smooth,
and the integral form of Gronwall's inequality.}) solution
to the ODE:
\begin{equation}\label{EqnODEForgamma}
\frac{d}{d \epsilon} \omega\q( \epsilon, t, x\w) = \frac{1}{\epsilon} W\q( \epsilon t, \omega\q( \epsilon,t,x\w)\w), \quad \omega\q( 0,t,x\w) =x.
\end{equation}
We will show that the (two-sided) inverse to the map $\gamma\mapsto W$ is given by $\gamma_t\q( x\w):= \omega\q( 1,t,x\w)$.\footnote{It is easy to see,
via the contraction mapping principle, using the fact that $W\q( 0\w)=0$
and $W\q( t\w)$ is smooth, that the solution of \eqref{EqnODEForgamma}
exists up to $\epsilon=1$ for $t$ sufficiently small.}
First we show that $\gamma_{\epsilon t}\q( x\w) = \omega\q( \epsilon,t,x\w)$, where $\omega$ corresponds to $W$ obtained by \eqref{EqnCurvesIIW} from $\gamma$.
Indeed, consider,
\begin{equation*}
\frac{d}{d\epsilon}\bigg|_{\epsilon=\epsilon_0} \gamma_{\epsilon t}\q( x\w) = \frac{d}{d\epsilon}\bigg|_{\epsilon=1} \frac{1}{\epsilon_0} \gamma_{\epsilon\q( \epsilon_0 t\w)} \q( x\w) = \frac{1}{\epsilon_0} W\q( \epsilon_0 t, \gamma_{\epsilon_0 t}\q( x\w)\w),
\end{equation*}
where in the last equality, we used the definition of $W$.
Thus $\gamma_{\epsilon t}\q( x\w)$ satisfies the ODE \eqref{EqnODEForgamma}.
Hence, $\gamma_t\q( x\w) = \omega\q( 1,t,x\w)$, and we may, therefore, reconstruct $\gamma$, given $W$.
We therefore have constructed a left inverse to the map $\gamma\mapsto W$.

Now we must show that if $W$ is as in the statement of the proposition,
and if we let $\omega$ be the solution to the ODE \eqref{EqnODEForgamma},
then,
\begin{equation}\label{EqnToShowRightInv}
\frac{d}{d\epsilon}\bigg|_{\epsilon=1} \omega\q( 1, \epsilon t, x\w) = W\q(  t, \omega\q( 1,t,x\w)\w).
\end{equation}
I.e., we must show our left inverse is also a right inverse.
\eqref{EqnToShowRightInv} will follow from \eqref{EqnODEForgamma} if we can show,
$\omega\q( 1,\epsilon t, x\w) = \omega\q( \epsilon, t, x\w)$.
This, in turn, will follow if we can show for every $\epsilon_0$,
$\omega\q( \epsilon_0 \epsilon, t, x\w) = \omega\q( \epsilon, \epsilon_0 t, x\w)$.
We already know that $\omega\q( \epsilon, \epsilon_0 t, x\w)$
satisfies the ODE \eqref{EqnODEForgamma} with $t$ replaced by
$\epsilon_0 t$.  Thus, we need only show the
same is true for $\omega\q( \epsilon_0 \epsilon, t, x\w)$.
However,
\begin{equation*}
\frac{d}{d\epsilon} \omega\q( \epsilon \epsilon_0, t, x\w) = \epsilon_0 \frac{1}{\epsilon \epsilon_0} W\q( \epsilon\epsilon_0 t, \omega\q( \epsilon\epsilon_0, t,x\w)\w) = \frac{1}{\epsilon} W\q( \epsilon \epsilon_0 t, \omega\q( \epsilon\epsilon_0,t,x\w)\w).
\end{equation*}
Since $\omega\q( 0\epsilon_0, t, x\w) =x$, by definition, this completes
the proof.
\end{proof}
The correspondence between $\gamma$ and $W$ given by the proof
of Proposition \ref{PropGetGammaFromW} will be an essential
point in what follows.

Now we turn back to the perspective used in Section \ref{SectionCCII}:
fixing $\delta$ and $x_0$ and working with $\q( \delta X,\sd\w)$ near $x_0$.
Let $\Omega\subseteq \R^n$ be open, and let $\Omega'\Subset\Omega''\Subset\Omega$
be open, relatively compact, subsets of $\Omega$.
We assume we are given a list of $C^\infty$ vector fields
with single-parameter formal degrees
$\q( Z,\dt\w)=\q( Z_1,\dt_1\w),\ldots,\q( Z_q,\dt_q\w)$,
$\dt_j\in \q( 0,\infty\w)$.
We fix $x_0\in \Omega'$ and assume that $\q( Z,\dt\w)$ satisfies
all of the assumptions
of the vector fields and formal degrees of
the same name in Section \ref{SectionCCII} (and
therefore we are given some $1\geq \xi>0$).  We also assume
that we are given a $C^\infty$ function
\begin{equation*}
\gh\q( t,x\w): \Q^N\q( \rho\w)\times \Omega''\rightarrow \Omega,
\end{equation*}
where $\rho>0$ is some fixed positive number, and $\gh\q( 0,x\w) \equiv x$.
Here $\rho>0$ is small enough that for $t\in \Q^{N}\q(\rho\w)$,
$\gh_t^{-1}$ exists.

For the primary example where we will apply the results
from this section,
take the same setup as in Section \ref{SectionProofOutline}.
Let $j,k\in \lA$ and set $j_0=j\wedge k$.  We will
take $\q( Z,\dt\w) = \q( 2^{-j_0} X, \sd\w)$ and
$\gh_t=\gamma_{2^{-j_0}t}$.
It is essential that all of the constants in this section
can be chosen to be independent of $x_0\in \K$ and
$j,k\in \lA$.
Note that our assumption on the list $\q( X,d\w)$ in
Section \ref{SectionCurves} implies that $\q( Z,\dt\w)= \q( 2^{-j_0}X,\sd\w)$
satisfies all of the assumptions of Section \ref{SectionCCII}
uniformly in the above parameters.  Thus, $m$-admissible constants
(as defined in Section \ref{SectionCCII}) can be chosen
to be independent of the above parameters.
In particular, the results of Theorem \ref{ThmMainCCThm} hold
uniformly in the above parameters.

We let $n_0, \eta_1, \xi_1>0$, and $\Phi: B^{n_0}\q( \eta_1\w) \rightarrow \B{Z}{\dt}{x_0}{\xi}$ be as in Theorem \ref{ThmMainCCThm}.
Furthermore, we let $Y_1,\ldots, Y_q$ be the pullbacks of $Z_1,\ldots, Z_q$
as in Theorem \ref{ThmMainCCThm}.
In order to work with our assumption that $\q( X,d\w)$ controls
$\gamma_t$ (see Definition \ref{DefnControlEveryScale}),
we introduce the following two conditions
on $\gh_t$, which will turn out to be equivalent.  All parameters
that follow are assumed to be strictly positive real numbers:
\begin{enumerate}
\item $\sQ_1\q( \rho_1,\tau_1,\q\{\sigma_1^m\w\}_{m\in \N}\w)$ ($\rho_1\leq \rho$, $\tau_1\leq \xi_1$):  For $x\in \Omega'$, define the vector field,
\begin{equation}\label{EqnDefnWhTheta}
\Wh\q( t,x\w) = \frac{d}{d\epsilon}\bigg|_{\epsilon=1} \gh_{\epsilon t}\circ \gh_{t}^{-1}\q( x\w).
\end{equation} 
We suppose,
\begin{itemize}
\item $\Wh\q( t,x\w) =\sum_{l=1}^q c_l\q( t,x\w) Z_l\q( x\w)$, on $\B{Z}{\dt}{x_0}{\tau_1}$,
\item $\sum_{\q|\alpha\w|+\q|\beta\w|\leq m} \CjN{Z^{\alpha} \partial_t^{\beta} c_l}{0}{\Q^N\q( \rho_1\w)\times \B{Z}{\dt}{x_0}{\tau_1}}\leq \sigma_1^m$.
\item Note that we may, without loss of generality, assume that
$c_l\q( 0,x\w)\equiv 0$, as we may replace $c_l\q( t,x\w)$ with $c_l\q( t,x\w)-c_l\q( 0,x\w)$ for every $l$ by using the fact that $\Wh\q( 0,x\w) \equiv 0$.
\end{itemize}

\item $\sQ_2\q( \rho_2, \tau_2,\q\{\sigma_2^m\w\}_{m\in \N}\w)$:
\begin{itemize}
\item $\gh\q(\Q^N\q(\rho_2 \w)\times \B{Z}{\dt}{x_0}{\tau_2} \w)\subseteq \B{Z}{\dt}{x_0}{\xi_1}$,
\item If $\eta'=\eta'\q( \tau_2\w)>0$ is a $2$-admissible constant so small\footnote{Such an $\eta'$ exists by Proposition \ref{PropExtraCCStuff}.}
that,
\begin{equation*}
\Phi\q( B^{n_0}\q( \eta'\w)\w)\subseteq \B{Z}{\dt}{x_0}{\tau_2}\subseteq \Phi\q( B^{n_0}\q( \eta_1\w)\w),
\end{equation*}
then if we define a new map,
\begin{equation*}
\theta_t\q( u\w) = \Phi^{-1}\circ \gamma_t\circ \Phi\q( u\w): \Q^{N}\q( \rho_2\w)\times B^{n_0}\q( \eta'\w)\rightarrow B^{n_0}\q( \eta_1\w),
\end{equation*}
we have,
\begin{equation*}
\CjN{\theta}{m}{\Q^N\q( \rho_2\w)\times B^{n_0}\q( \eta'\w)}\leq \sigma_2^m.
\end{equation*}
\end{itemize}
\end{enumerate}

\begin{rmk}
The particular parameters in $\sQ_1$ and $\sQ_2$ are not important.
All that is important is that, in our primary application, they
may be chosen independent of $j,k$ and $x_0$.  At times we will
keep track of the parameters to make clear the interdependence
between the various constants, but the reader may wish to ignore these
details on a first reading.
\end{rmk}

\begin{prop}
$\sQ_1\Leftrightarrow \sQ_2$ in the following sense:
\begin{itemize}
\item $\sQ_1\q( \rho_1, \tau_1, \q\{\sigma_1^m\w\}_{m\in \N}\w)\Rightarrow$
there exists a $2$-admissible constant $\rho_2=\rho_2\q( \rho_1,\tau_1,\sigma_1^1,N\w)$, and $m+1$-admissible constants $\sigma_2^m=\sigma_2^m\q( \sigma_1^{m+1},N\w)$
such that $\sQ_2\q( \rho_2, \frac{\tau_1}{2}, \q\{\sigma_2^m\w\}_{m\in \N}\w)$
holds.

\item $\sQ_2\q( \rho_2, \tau_2, \q\{\sigma_2^m\w\}_{m\in \N}\w)\Rightarrow$
there exists a $2$-admissible constant $\tau_1=\tau_1\q( \tau_2\w)>0$
and $m$-admissible constants $\sigma_1^m=\sigma_1^m\q( \sigma_2^{m+1},N\w)$,
such that $\sQ_1\q( \rho_2,\tau_1, \q\{\sigma_1^m\w\}_{m\in \N}\w)$ holds.
\end{itemize}
\end{prop}
\begin{proof}
Suppose $\sQ_1\q( \rho_1,\tau_1,\q\{\sigma_1^m\w\}_{m\in \N}\w)$ holds.
Fix $\rho_2>0$ small, to be chosen in a moment.
Suppose
\begin{equation*}
z\in \gh\q(\Q^N\q( \rho_2\w)\times \B{Z}{\dt}{x_0}{\frac{\tau_1}{2}} \w);
\end{equation*}
and so there exists $t_0\in \Q^N\q( \rho_2\w)$ and $y\in \B{Z}{\dt}{x_0}{\frac{\tau_1}{2}}$
such that $z=\gh_{t_0}\q( y\w)$.  We wish to show that $z\in \B{Z}{\dt}{x_0}{\xi_1}$ (by appropriately choosing $\rho_2$).  To do so, it suffices
to show $z\in \B{Z}{\dt}{y}{\frac{\tau_1}{2}}$, as then we would have
$z\in \B{Z}{\dt}{x_0}{\tau_1}\subseteq \B{Z}{\dt}{x_0}{\xi_1}$.
Define a curve $\kappa:\q[0,1\w]\rightarrow \Omega$ by
\begin{equation*}
\kappa\q( s\w) = \gh_{st_0}\q( y\w).
\end{equation*}
Note that $\kappa\q( 0\w) = y$ and $\kappa\q( 1\w) = z$.  Also,
\begin{equation*}
\frac{d}{ds} \kappa\q( s\w) = \frac{1}{s} W\q( st_0, \kappa\q( s\w)\w) = \sum_{l=1}^q \frac{1}{s} c_l\q( st_0, \kappa\q( s\w)\w) Z_l\q( \kappa\q( s\w)\w).
\end{equation*}
Since\footnote{As remarked above, we may assume $c_l\q( 0,x\w)\equiv 0$.} $c_l\q( st_0, \kappa\q( s\w)\w)\big|_{s=0}=0$, the mean value theorem
shows that, for $0<s<1$,
\begin{equation*}
\begin{split}
\q|\frac{1}{s} c_l\q( st_0, \kappa\q( s\w)\w)\w| &\leq \sup_{0\leq s\leq 1} \q|\frac{d}{ds} c_l\q( st_0, \kappa\q( s\w)\w)\w|\\
&\leq \sup_{0\leq s\leq 1} \q[\q|t_0\cdot \grad_t c_l\q( s t_0, \kappa\q( s\w)\w)\w| + \sum_{l=1}^q \q|\q(\frac{1}{s} c_l\q( st_0\w) Z_l c_l \w)\q(st_0, \kappa\q( s\w) \w)\w|\w]\\
&\lesssim \q|t_0\w|\q( \sigma_1^1 + \q(\sigma_1^1\w)^2\w),
\end{split}
\end{equation*}
where we have used that $\q|\frac{1}{s}c_l\q( st_0,\cdot\w)\w|\leq \sigma_1^1$,
and in the last line, we are thinking of $c_l\q( st_0\w) Z_l c_l$ as
the vector field $c_l\q( st_0,\cdot\w) Z_l\q( \cdot\w)$ acting on the
function $c_l\q( s t_0,\cdot\w)$.

Thus, $\frac{d}{ds} \kappa\q( s\w) = \sum_{l=1}^q a_l\q( s\w) Z_l\q( \kappa\q( s\w)\w)$, where $\q|a_l\w|\lesssim \q|t_0\w|$.  By taking $\q|t_0\w|$ sufficiently
small (i.e., by taking $\rho_2$ sufficiently small), we may
insure that $z=\kappa\q( 1\w) \in \B{Z}{\dt}{y}{\frac{\tau_1}{2}}$.

Hence, it makes since to define $\eta'>0$ and $\theta_t$ as in
the definition of $\sQ_2$.  It follows directly
from the definition of $\theta$ that $\theta_0\q( u\w) \equiv u$.
Furthermore, pulling back \eqref{EqnDefnWhTheta} via
the map $\Phi$ shows that,
\begin{equation}\label{EqnDefnWtTheta}
\frac{d}{d\epsilon}\bigg|_{\epsilon=0} \theta_{\epsilon t}\circ \theta_t^{-1}\q( u\w) = \Wt\q( t, u\w),
\end{equation}
where,
\begin{equation*}
\begin{split}
\Wt\q( t,u\w) &= \sum_{l=1}^q \ct_l\q( t,u\w) Y_l\q( u\w),\\
\ct_l\q(t,u\w) &= c_l\q( t,\Phi\q( u\w)\w).
\end{split}
\end{equation*}

We claim,
\begin{equation}\label{EqnToShowCmNWt}
\CjN{\Wt}{m}{\Q^N\q( \rho_2\w)\times B^{n_0}\q( \eta'\w)} \leq C_m,
\end{equation}
where $C_m=C_m\q( \sigma_1^m\w)$ is an $m$-admissible constant.
Indeed, we already have the result for $Y_l$ by \eqref{EqnRescaledSmooth}.
Hence it suffices to bound the $C^m$ norm of $\ct_l$.  However,
applying (a slight modification of) \eqref{EqnRescaledCjN},\footnote{This
slight modification can be proved easily by combining \eqref{EqnRescaledSmooth}
and \eqref{EqnRescaledSpan}.}
\begin{equation}\label{EqnThetaEstCmNorm}
\CjN{\ct_l}{m}{\Q^N\q( \rho_2\w)\times B^{n_0}\q( \eta'\w)}\approx_{m-1} \sum_{\q|\alpha\w|+\q|\beta\w|\leq m} \CjN{Y^{\alpha} \partial_t^{\beta} \ct_l}{0}{\Q^N\q( \rho_2\w)\times B^{n_0}\q( \eta'\w) }.
\end{equation}
Note that,
\begin{equation*}
Y^{\alpha} \partial_t^\beta \ct_l\q( t,u\w) = \q(Z^\alpha \partial_t^{\beta} c_l \w)\q( t,\Phi\q( u\w)\w).
\end{equation*}
Combining this with the fact that,
\begin{equation*}
\Q^N\q( \rho_2\w) \times \Phi\q( B^{n_0}\q( \eta'\w)\w)\subseteq \Q^{N}\q( \rho_1\w) \times \B{Z}{\dt}{x_0}{\tau_1},
\end{equation*}
we have from \eqref{EqnThetaEstCmNorm},
\begin{equation*}
\CjN{\ct_l}{m}{\Q^N\q( \rho_2\w)\times B^{n_0}\q( \eta'\w) } \lesssim_{m-1} \sum_{ \q|\alpha\w|+\q|\beta\w|\leq m} \CjN{Z^{\alpha} \partial_t^{\beta} c_l}{0}{\Q^N\q( \rho_1\w)\times \B{Z}{\dt}{x_0}{\tau_1}} \leq \sigma_1^m,
\end{equation*}
completing the proof of \eqref{EqnToShowCmNWt}.

To complete the proof of $\sQ_1\Rightarrow \sQ_2$, we bound
the $C^m$ norm of $\theta$.  For
$t\in \Q^N\q( \rho_2\w)$, $u\in B^{n_0}\q( \eta'\w)$, and
$\epsilon\in \q[0,1\w]$, define a function,
\begin{equation*}
\omega\q( \epsilon, t,u\w) = \theta\q( \epsilon t, u\w).
\end{equation*}
In light of \eqref{EqnDefnWtTheta}, $\omega$ satisfies
the ODE in the $\epsilon$ variable:
\begin{equation*}
\frac{d}{d\epsilon} \omega\q( \epsilon, t, u\w) = \frac{1}{\epsilon} \Wt\q( \epsilon t, \omega\q( \epsilon, t, u\w)\w), \quad \omega\q( 0, t, u\w) =u.
\end{equation*}
Using that $\Wt\q( 0,u\w) =0$ and that $\Wt$ is smooth, standard
theorems for ODEs show that there exists $\sigma_2^m=\sigma_2^m\q( C_{m+1},m,N\w)$
such that,
\begin{equation*}
\CjN{\omega}{m}{\q[0,1\w]\times \Q^N\q( \rho_2\w)\times B^{n_0}\q( \eta'\w)}\leq \sigma_2^m.
\end{equation*}
Since $\theta\q( t,u\w) = \omega\q( 1,t,u\w)$, this completes the proof of
$\sQ_1\Rightarrow \sQ_2$.

Suppose that $\sQ_2\q( \rho_2, \tau_2, \q\{\sigma_2^m\w\}_{m\in \N}\w)$
holds.  Define,
\begin{equation*}
\Wt\q( t,u\w) = \frac{d}{d\epsilon}\bigg|_{\epsilon=1} \theta_{\epsilon t}\circ\theta_{t}^{-1}\q( u\w).
\end{equation*}
Then,
\begin{equation*}
\CjN{\Wt}{m}{\Q^N\q( \rho_2\w)\times B^{n_0}\q( \eta'\w)} \leq C_m \sigma_2^{m+1},
\end{equation*}
where $C_m$ depends only on $m$.  Using \eqref{EqnRescaledSmooth}
and \eqref{EqnRescaledSpan}, we see that,
\begin{equation}\label{EqnQ2impQ1WtForm}
\Wt\q( t,u\w) = \sum_{l=1}^q \ct_l \q( t,u\w) Y_l\q( u\w),
\end{equation}
and,
\begin{equation*}
\CjN{\ct_l}{m}{\Q^N\q( \rho_2\w)\times B^{n_0}\q( \eta'\w)} \leq C_m',
\end{equation*}
where $C_m'=C_m'\q( \sigma_2^{m+1}\w)$ is an $m$-admissible constant.
Applying, again a slight modification of, \eqref{EqnRescaledCjN},
we have,
\begin{equation}\label{EqnQ2ImpQ1Boundclt}
\sum_{\q|\alpha\w|+\q|\beta\w|\leq m} \CjN{Y^{\alpha} \partial_t^\beta \ct_l}{0}{\Q^N\q( \rho_2\w)\times B^{n_0}\q(\eta'\w)}\approx_{m-1} \CjN{\ct_l}{m}{\Q^N\q( \rho_2\w)\times B^{n_0}\q( \eta'\w)}\leq C_m'.
\end{equation}
Using Proposition \ref{PropExtraCCStuff}, let $\tau_1=\tau_1\q(\eta'\w)>0$
be a $2$-admissible constant such that,
\begin{equation*}
\B{Z}{\dt}{x_0}{\tau_1}\subseteq \Phi\q( B^{n_0}\q(\eta'\w)\w).
\end{equation*}
Pushing \eqref{EqnQ2impQ1WtForm} and \eqref{EqnQ2ImpQ1Boundclt} forward
via the map $\Phi$, and taking $\Wh$ as in the definition of
$\sQ_1$, we see for $x\in \Phi\q( B^{n_0}\q( \eta'\w)\w)$ and $c_l\q( t,x\w)=\ct_l\q( t,\Phi^{-1}\q( x\w)\w)$,
\begin{equation*}
\Wh\q(t,x\w)=\sum_{l=1}^q c_l\q( t,x\w) Z_l\q( x\w),
\end{equation*}
\begin{equation*}
\sum_{\q|\alpha\w|+\q|\beta\w|\leq m} \CjN{Z^\alpha \partial_t^\beta c_l}{0}{\Q^N\q(\rho_2\w)\times \B{Z}{\dt}{x_0}{\tau_1}}\lesssim_{m-1} C_m',
\end{equation*}
completing the proof.
\end{proof}

\begin{defn}
We say $\q( Z,\dt\w)$ controls $\gh$ at the unit scale if either
of the equivalent conditions $\sQ_1$ or $\sQ_2$ holds (for some
choice of the parameters).  If we wish to make the point
$x_0$ explicit, we will say $\q( Z,\dt\w)$ controls $\gh$ at the unit scale
near $x_0$.
\end{defn}

\begin{prop}\label{PropConnectionBetweenControlEveryScaleAndUnit}
Let $\q( X,d\w)$ and $\gamma$ be as in Section \ref{SectionCurves}.
$\q( X,d\w)$ controls $\gamma$ if and only
if $\q( \delta X, \sd\w)$ controls $\gamma_{\delta t}$ at the
unit scale near $x_0$ for every $\delta\in \sA$ and $x_0\in \K$, and the parameters
of $\sQ_1$ (equivalently $\sQ_2$) can be chosen independent
of $x_0\in \K$ and $\delta\in \sA$.
\end{prop}
\begin{proof}
Let $\delta\in \sA$ and $x_0\in \K$.  Define $\gh_t= \gamma_{\delta t}$.
Note that if $W\q( t,x\w)$ is as in Definition \ref{DefnControlEveryScale}
and if $\Wh\q( t,x\w)$ is as in $\sQ_1$, then,
\begin{equation*}
\Wh\q( t,x\w) = W\q( \delta t,x\w).
\end{equation*}
Now the result follows just by comparing $\sQ_1$ and Definition \ref{DefnControlWEveryScale}.
\end{proof}

In light of Proposition \ref{PropConnectionBetweenControlEveryScaleAndUnit},
to work with our assumption that $\q( X,d\w)$ controls $\gamma$, we may
instead work ``at the unit scale.''  For the rest of this section
we work in this perspective, with the understanding
that Proposition \ref{PropConnectionBetweenControlEveryScaleAndUnit}
allows us to transfer our results from the ``unit scale'' to
``every scale.''

\begin{prop}\label{PropControlCompAndInv}
If $\q( Z,\dt\w)$ controls $\gh_{t_1}^1$ and $\gh_{t_2}^2$ at the unit scale,
then $\q( Z,\dt\w)$ controls $\gh_{t_1}^1\circ \gh_{t_2}^2$ at the unit scale.
If $\q( Z,\dt\w)$ controls $\gh_t$ at the unit scale, then $\q( Z,\dt\w)$
controls $\gh_t^{-1}$ at the unit scale.
\end{prop}
\begin{proof}
Both of the above statements are more easily verified using
$\sQ_2$.  Indeed, the $C^m$ norm of,
\begin{equation*}
\Phi^{-1}\circ \gh_{t_1}^1\circ \gh_{t_2}^2\circ \Phi = 
\q(\Phi^{-1} \circ \gh_{t_1}^1\circ \Phi \w)
\circ
\q(\Phi^{-1} \circ \gh_{t_2}^2\circ \Phi \w)
\end{equation*}
can clearly be bounded in terms of the $C^m$ norms of,
\begin{equation*}
\Phi^{-1} \circ \gh_{t_1}^1\circ \Phi 
\text{ and }
\Phi^{-1} \circ \gh_{t_2}^2\circ \Phi,
\end{equation*}
provided one shrinks the parameters $\rho_2$ and $\tau_2$ appropriately
(so that the composition is defined).  A similar proof works
for $\gh_t^{-1}$.  We leave further details to the interested reader.
\end{proof}

\begin{prop}\label{PropScalingOfControl}
If $\q( Z,d\w)$ controls $\gh_t$ at the unit scale, and $c=\q( c_1,\ldots, c_N\w)\in \q[0,1\w]^N$ is a constant, then $\q( Z,\dt\w)$ controls,
\begin{equation*}
\gh^c\q( \q(t_1,\ldots, t_N\w),x\w) := \gh\q(\q( c_1t_1,\ldots, c_Nt_N\w),x \w),
\end{equation*}
at the unit scale.  Moreover, the parameters in the definition of control
at the unit scale may be chosen independent of $c$.
\end{prop}
\begin{proof}
This is obvious from the definition of control (using either $\sQ_1$ or $\sQ_2$).
\end{proof}

So far in this section, we have covered the necessary technical
details to deal with our assumption that
$\q( X,d\w)$ controls $\gamma$.
The other part of our assumption was that the vector fields
with formal degrees $\q( X_j,d_j\w)$ were generated by vector fields
with formal degrees
of the form $\q( X_\alpha, \deg\q( \alpha\w)\w)$,
where
$t^{\alpha} X_\alpha$ was a term in the Taylor series
of $W\q( t,x\w)$, when taken in the $t$ variable, and
$\deg\q( \alpha\w)$ was nonzero in only one component.
We now turn to the various definitions and theorems necessary
to deal with this assumption.  Again, we work
at the unit scale.  We continue to take $\q( Z,\dt\w)$, $\gh_t\q( x\w)$,
and $x_0$ as above.
We assume, throughout the rest of the section,
that $\q( Z,\dt\w)$ controls $\gh$ at the unit scale (near $x_0$).

\begin{defn}\label{DefnZ1ZrMGenerates}
Let $\q\{ Z_1,\ldots, Z_r\w\}$ be a subset of $\q\{ Z_1,\ldots, Z_q\w\}$.
For $M\in \N$, we say $Z_1,\ldots, Z_r$ $M$-generates $Z_1,\ldots, Z_q$
if each $Z_j$ ($r+1\leq j\leq q$) can be written in the form,
\begin{equation*}
Z_j= \ad{Z_{l_1}}\ad{Z_{l_2}}\cdots\ad{Z_{l_m}} Z_{l_{m+1}}, \quad 0\leq m\leq M-1, \quad l\leq l_k\leq r.
\end{equation*}
\end{defn}

Define $\Wh\q( t,x\w)$ as the vector field,
\begin{equation*}
\Wh\q(t,x\w)=\frac{d}{d\epsilon}\bigg|_{\epsilon=1} \gh_{\epsilon t}\circ \gh_{t}^{-1}\q( x\w).
\end{equation*}
Let $Z_{\alpha}\q( x\w)$ be the Taylor coefficients of $W\q( t,x\w)$ when
the Taylor series is taken in the $t$ variable:
\begin{equation*}
\Wh\q( t,x\w) \sim \sum_{\q|\alpha\w|>0} t^\alpha Z_\alpha\q( x\w).
\end{equation*}

\begin{defn}
For $M_1,M_2\in \N$, we say $\gh$ $M_1,M_2$ generates $Z_1,\ldots, Z_q$
if there is a subset $\sF\subseteq \q\{Z_\alpha: \q|\alpha\w|\leq M_1\w\}\cap \q\{Z_1,\ldots, Z_q\w\}$, such that $\sF$ $M_2$-generates $Z_1,\ldots, Z_q$.
\end{defn}

Since we are assuming that $\q( Z,\dt\w)$ controls $\gh$ at the unit
scale, we may pull $\gh$ back under the map $\Phi$
as in $\sQ_2$ to obtain the $C^\infty$ map:
\begin{equation*}
\theta_t\q( u\w)=\Phi^{-1} \circ \gh_t\circ \Phi\q( u\w): \Q^N\q( \rho_2\w)\times B^{n_0}\q( \eta'\w)\rightarrow B^{n_0}\q( \eta\w),
\end{equation*}
as in $\sQ_2$.  As usual, define $Y_1,\ldots Y_q$ to be the pullbacks
of $Z_1,\ldots, Z_q$ under the map $\Phi$.  Define,
\begin{equation*}
\Wt\q( t,u\w) = \frac{d}{d\epsilon}\bigg|_{\epsilon=1} \theta_{\epsilon t}\circ \theta_{t}^{-1} \q( u\w).
\end{equation*}
Let $Y_\alpha\q( u\w)$ be the Taylor coefficients of $\Wt$ expressed
as a Taylor series in the $t$-variable:
\begin{equation*}
\Wt\q( t,u\w) \sim \sum_{\q|\alpha\w|>0} t^{\alpha} Y_{\alpha}.
\end{equation*}
It is immediate to verify that $Y_\alpha$ is the pullback of
$Z_\alpha$ under the map $\Phi$.

\begin{thm}\label{ThmM1M2Generates}
Suppose $\gh$ $M_1,M_2$ generates $Z_1,\ldots, Z_q$.  Then, the set of
vector fields
\begin{equation}\label{EqnGeneratesToShowHor}
\q\{Y_\alpha: \q|\alpha\w|\leq M_1\w\},
\end{equation}
satisfies H\"ormander's conditions of order $M_2$ on $B^{n_0}\q( \eta'\w)$.
More specifically, if we let $V_1,\ldots V_L$ denote those vector fields
in \eqref{EqnGeneratesToShowHor} along with all commutators
of those vector fields in \eqref{EqnGeneratesToShowHor} up to order $M_2$,
then,
\begin{equation*}
\q|\det_{n_0\times n_0} V\q( u\w) \w|_{\infty} \gtrsim_2 1, \quad \forall u\in B^{n_0}\q( \eta'\w),
\end{equation*}
where we have written $V$ to denote the matrix whose columns are $V_1,\ldots, V_L$.
In short, $\theta$ satisfies $\cZ$.
\end{thm}
\begin{proof}
By assumption a subset of $\q\{ Z_\alpha: \q|\alpha\w|\leq M_1\w\}$
$M_2$-generates $Z_1,\ldots, Z_q$.  Pulling this back
under $\Phi$ we see that a subset of \eqref{EqnGeneratesToShowHor}
$M_2$-generates $Y_1,\ldots Y_q$.  I.e., each of $Y_1,\ldots Y_q$
appears in the list $V_1,\ldots V_L$.  Thus,
\begin{equation*}
\q|\det_{n_0\times n_0} V\q( u\w) \w|_{\infty} \geq \q|\det_{n_0\times n_0} Y\q( u\w)\w|_\infty \gtrsim_2 1, \quad \forall u\in B^{n_0}\q( \eta'\w).
\end{equation*}
In the last inequality, we have used \eqref{EqnRescaledSpan}.
\end{proof}

We now combine the above discussion to obtain the main results
of this section.
\begin{thm}\label{ThmCompSatisfycZ}
Suppose $Z_1,\ldots, Z_r$ $M_2$-generates $Z_1,\ldots, Z_q$.
Let $\gh^l\q( t,x\w):\Q^N\q( \rho\w)\times \Omega''\rightarrow \Omega$,
$l=1,\ldots, k$, be a sequence of $k$ $C^\infty$ functions,
such that each $\gh^l$ is controlled at the unit scale by $\q( Z,\dt\w)$.
Let
\begin{equation*}
\Wh_j\q(t,x\w) = \frac{d}{d\epsilon}\bigg|_{\epsilon=1} \gh^j_{\epsilon t}\circ \q(\gh^j_t \w)^{-1}\q( x\w),
\end{equation*}
and define the vector fields $V_{j,\alpha}$ to be the Taylor coefficients (in $t$) of $\Wh_j$:
\begin{equation*}
\Wh_j\q( t\w) \sim \sum_{\alpha} t^{\alpha} V_{j,\alpha}.
\end{equation*}
Suppose there is $M_1\in \N$ such that each $Z_s$ ($1\leq s\leq r$)
appears as some $V_{j,\alpha}$ for $\q|\alpha\w|\leq M_1$.
Define,
\begin{equation*}
\gh_{\q( t_1,\ldots, t_k\w)}\q( x\w) = \gh_{t_1}^1\circ \gh_{t_2}^2\circ \cdots \circ\gh_{t_k}^k\q( x\w).
\end{equation*}
Then, $\gh$ is controlled at the unit scale by $\q( Z,\dt\w)$.
Moreover if we take $\theta$ as in $\sQ_2$, then $\theta$
satisfies $\cZ$, where the constants in the definition of
$\cZ$ are $2$-admissible constants depending on the above parameters ($M_1$, $M_2$).
\end{thm}

\begin{cor}\label{CorCompSatisfycJ}
Take the same setup as Theorem \ref{ThmCompSatisfycZ}; and, more precisely,
assume each $\gh^l$ satisfies $\sQ_2\q( \rho_2,\tau_2,\q\{\sigma_2^m\w\}_{m\in \N}\w)$.
Then, $\theta$ satisfies $\cJ$ (at $0$), where the constants in the definition
of $\cJ$ are $2$-admissible constants depending on the above
parameters ($M_1$, $M_2$, $k$, $\rho_2$, $\tau_2$, $\q\{\sigma_2^m\w\}_{m\in \N}$).
\end{cor}
\begin{proof}[Proof of Corollary \ref{CorCompSatisfycJ} given Theorem \ref{ThmCompSatisfycZ}]
Letting $\gh^1,\ldots, \gh^k$ vary over all choices satisfying
the hypotheses of Corollary \ref{CorCompSatisfycJ} (for a fixed
choice of the parameters), Theorem \ref{ThmCompSatisfycZ}
shows that the set of all such $\theta$, so obtained,
satisfies $\cZu$.  Since $\q\{\sigma_2^m\w\}_{m\in \N}$
is fixed, the set of all such $\theta$ forms a bounded
subset of $C^\infty$.  The corollary
now follows by applying Theorem \ref{ThmUnifEquivCond}
which states that $\cZu\Rightarrow \cJu$ for bounded
subsets of $C^\infty$.
\end{proof}

\begin{proof}[Proof of Theorem \ref{ThmCompSatisfycZ}]
That $\gh$ is controlled at the unit scale by $\q( Z,\dt\w)$
follows from Proposition \ref{PropControlCompAndInv}.
Define,
\begin{equation*}
\Wh\q( t\w) = \frac{d}{d\epsilon}\bigg|_{\epsilon=1} \gh_{\epsilon t}\circ \gh_t^{-1} \sim \sum_{\q|\alpha\w|>0} t^\alpha Z_\alpha.
\end{equation*}
In light of Theorem \ref{ThmM1M2Generates}, we need only show
that each $Z_1,\ldots, Z_r$ appears as some $Z_\alpha$
with $\q|\alpha\w|\leq M_1$.
Since,
\begin{equation*}
\gh_{\q( t_1,\ldots, t_k\w)} = \gh_{t_1}^1\circ \gh_{t_2}^2\circ\cdots \circ \gh_{t_k}^k,
\end{equation*}
it is easy to see that,
\begin{equation*}
\Wh\q( 0,\cdots, 0, t_j, 0,\cdots 0,x\w) = \Wh_j\q( t_j,x\w).
\end{equation*}
Since each $Z_1,\ldots, Z_r$ appears as a Taylor coefficient
of some $\Wh_j$ (with $\q|\alpha\w|\leq M_1$), the same is true for $\Wh$.
This completes the proof.
\end{proof}

Though we will work in slightly greater generality in the sequel,
it is worthwhile to take a moment and see the main
consequence of our assumptions on $\gamma$ in
Theorem \ref{ThmMainThmSecondPass}, in light of
Corollary \ref{CorCompSatisfycJ}.

\begin{prop}\label{PropHowAssumpsAreUsed}
Let $\q( X, d\w)$ and $\gamma$ be as in Theorem \ref{ThmMainThmSecondPass};
i.e., $\q( X,d\w)$ and $\gamma$ satisfy all of the assumptions
of Section \ref{SectionCurves}.  Fix $x_0\in \K$ and
$\delta_1,\delta_2\in \sA$.  Let $\delta_0=\delta_1\vee \delta_2$ ($\delta_0\in \sA$ by our assumption on $\sA$).
Define, $\gh$ as either,
\begin{equation}\label{EqnFirstGtDefn}
\gh_{\q( t_1,t_2,t_3,t_4\w)}\q( x\w) = \gamma_{\delta_1 t_4}^{-1}\circ \gamma_{\delta_2 t_3}\circ \gamma_{\delta_2 t_2}^{-1} \circ \gamma_{\delta_1 t_1}\q( x\w),
\end{equation}
or,
\begin{equation}\label{EqnSecondGtDefn}
\gh_{\q( t_1,t_2,t_3,t_4\w)}\q( x\w) = \gamma_{\delta_1 t_4}\circ \gamma_{\delta_2 t_3}^{-1}\circ \gamma_{\delta_2 t_2} \circ \gamma_{\delta_1 t_1}^{-1}\q( x\w).
\end{equation}
Then, $\q( \delta_0 X, \sd\w)$ controls $\gh$ at the unit scale
near $x_0$ (and the parameters in the definition of control
may be chosen independent of $\delta_1$, $\delta_2$, and $x_0$).  Moreover, if we define $\theta_t\q( u\w)$
as in $\sQ_2$, then $\theta$ satisfies $\cJ$ (at $0$) uniformly
in $\delta_1$, $\delta_2$, and $x_0$.
More specifically, if we let $\sS$ be the set of all such $\theta$
(as $\delta_1$, $\delta_2$, and $x_0$ vary) then
$\sS$ satisfies $\cJu$.
\end{prop}
\begin{proof}\renewcommand{\qedsymbol}{}
Let $\q( Z,\dt\w)= \q( \delta_0 X, \sd\w)$.
By Proposition \ref{PropConnectionBetweenControlEveryScaleAndUnit},
$\gamma_{\delta_0 t}$ is controlled at the unit scale by $\q( Z,d\w)$.
Proposition \ref{PropScalingOfControl} then shows
that $\gamma_{\delta_1 t}$ and $\gamma_{\delta_2 t}$ are controlled
by $\q( Z,\dt\w)$ at the unit scale.
Proposition \ref{PropControlCompAndInv} applies to show
that $\gamma_{\delta_1 t}^{-1}$, $\gamma_{\delta_2 t}^{-1}$, and $\gh$ are all controlled at the unit scale by $\q( Z,\dt\w)$ (uniformly
in the relevant parameters).
The proof will then be completed by combining
Corollary \ref{CorCompSatisfycJ} with the following lemma.
\end{proof}

\begin{lemma}\label{LemmaEachZjAppearsIngammaScaled}
Taking the same setup as in Proposition \ref{PropHowAssumpsAreUsed},
and letting $\q( Z,\dt\w)=\q( \delta_0 X,\sd\w)$.  There is a subset
$Z_1,\ldots, Z_r$ of $Z_1,\ldots, Z_q$ such that $Z_1,\ldots, Z_r$
$M_2$-generates $Z_1,\ldots, Z_q$.  Moreover, each $Z_1,\ldots, Z_r$
appears as a Taylor coefficient (in $t$) of either
\begin{equation*}
\Wh_{\delta_1}\q( t,x\w)=\frac{d}{d\epsilon}\bigg|_{\epsilon=1} \gamma_{\delta_1\q( \epsilon t\w)}\circ\gamma_{\delta_1 t}^{-1}\q( x\w)
\text{ or }
\Wh_{\delta_2}\q( t,x\w)=\frac{d}{d\epsilon}\bigg|_{\epsilon=1} \gamma_{\delta_2\q( \epsilon t\w)}\circ\gamma_{\delta_2 t}^{-1}\q( x\w),
\end{equation*}
and this Taylor coefficient may be chosen to be the coefficient of $t^{\alpha}$,
where $\q|\alpha\w|\leq M_1$.  Here, $M_1,M_2\in \N$ may be chosen
independent of $\delta_1$, $\delta_2$, and $x_0$.
\end{lemma}
\begin{proof}
First we recall the definition of $\q( X,d\w)$.  
We wrote,
\begin{equation*}
W\q( t,x\w) = \frac{d}{d\epsilon}\bigg|_{\epsilon=1} \gamma_{\epsilon t}^{-1}\circ \gamma_{t}\q( x\w) \sim \sum_{\q|\alpha\w|>0} t^{\alpha} X_\alpha.
\end{equation*}
We took a finite subset of 
$$\q\{\q( X_\alpha, \deg\q( \alpha\w)\w): \deg\q( \alpha\w) \text{ is nonzero in only one component}\w\},$$
denote this finite set by $\q( X_1,d_1\w),\ldots, \q( X_r,d_r\w)$,
and assume each $\q( X_1,d_1\w),\ldots, \q( X_r,d_r\w)$ appears
as $\q( X_\alpha, \deg\q( \alpha\w)\w)$ for some $\q|\alpha\w|\leq M_1$.
Take $Z_1,\ldots, Z_r$ to be $\delta_0^{d_1}X_1,\ldots, \delta_0^{d_r}X_r$.

We assumed that $\q( X_1,d_1\w),\ldots, \q(X_r,d_r\w)$ generates the finite
list $\q( X_1,d_1\w), \ldots, \q( X_q,d_q\w)$.  That is,
each $\q( X_j,d_j\w)$ can be written in the form,
\begin{equation}\label{EqnTheXjGenerate}
\begin{split}
X_j &= \ad{X_{l_1}}\ad{X_{l_2}}\cdots \ad{X_{l_{m-1}}} X_{l_m},\\
d_j &= d_{l_1}+d_{l_1}+\cdots+d_{l_m},
\end{split}
\end{equation}
where $1\leq l_k\leq r$ for each $k$, for some $1\leq m=m\q( j\w)$.  We choose
$M_2\in \N$ so that $m\leq M_2$ for every $j$.
Note that if we take $\q( Z,\dt\w)= \q( \delta_0 X,\sd\w)$
then \eqref{EqnTheXjGenerate} shows,
\begin{equation*}
Z_j = \ad{Z_{l_1}}\ad{Z_{l_2}}\cdots \ad{Z_{l_{m-1}}} Z_{l_m},
\end{equation*}
for the same choice of $l_1,\ldots, l_m$.  Hence,
$Z_1,\ldots, Z_r$ $M_2$-generates $Z_1,\ldots, Z_q$.

To complete the proof, we must show that each $Z_j$ ($1\leq j\leq r$)
appears as a Taylor coefficient of either $W_{\delta_1}$ or $W_{\delta_2}$.
Fix $j$ ($1\leq j\leq r$), and suppose that $d_j$ is nonzero in
only the $\mu$ component (by assumption we know it is nonzero
in only one component).  Since $\delta_0=\delta_1\vee \delta_2$,
$\delta_0^\mu\in \q\{\delta_1^\mu, \delta_2^\mu\w\}$.  We suppose
that $\delta_0^\mu=\delta_1^\mu$, the other case is similar.
We will show $Z_j$ is a Taylor coefficient of
$W_{\delta_1}$.  Indeed,
\begin{equation*}
W_{\delta_1}\q( t,x\w) = W\q( \delta_1 t,x\w) \sim \sum_{\q|\alpha\w|>0} t^{\alpha} \delta_1^{\deg\q( \alpha\w)} X_\alpha.
\end{equation*}
Take $\q|\alpha\w|\leq M_1$ such that $\q( X_j,d_j\w) = \q( X_\alpha, \deg\q(\alpha\w)\w)$.  Since $d_j$ is nonzero in only the $\mu$ component, we have,
\begin{equation*}
\delta_1^{\deg\q( \alpha\w)} X_\alpha = \delta_1^{d_j} X_j = \delta_0^{d_j} X_j = Z_j,
\end{equation*}
completing the proof.
\end{proof}

\begin{rmk}\label{RmkPureIsUsed}
Lemma \ref{LemmaEachZjAppearsIngammaScaled} is the only place
where we use that $\q( X_1,d_1\w),\ldots, \q( X_r,d_r\w)$
were chosen with each $d_j$ nonzero in only one component;
i.e., that $\q( X,d\w)$ was generated by ``pure-powers.''
Note, though, that this is essential to
Lemma \ref{LemmaEachZjAppearsIngammaScaled}.
\end{rmk}

	\subsection{Multi-parameter curvature conditions}\label{SectionMultiParamCurvature}
	Throughout the paper, we work mostly with the curvature
conditions $\cZ$ and $\cJ$.  Indeed, using the notation
in Proposition \ref{PropHowAssumpsAreUsed}, the essence
of Section \ref{SectionCurvesII} was that $\theta$
satisfies $\cJ$ uniformly in any relevant parameters (and
this followed by showing that it satisfied $\cZ$ uniformly
in any relevant parameters).
While $\cZ$ seems to be the best suited to theoretical
applications, $\cG$ is perhaps more intuitive.
The reader may, therefore, wonder if there is an analog
to Theorem \ref{ThmEquivConds} in our setup.
Indeed, there is, and this section is devoted to discussing
this fact.  The results in this section will not be used
in the sequel (except briefly in Section \ref{SectionAdjoint}), and the reader may safely skip this section
on a first reading.

The following curvature conditions, we think of as being
``scale invariant.''  We, therefore, denote them
by $\cGs$, $\cYs$, and $\cZs$ 
(we discuss $\cJ$ in Remark \ref{RmkcJs}).
The condition $\cZs$ is precisely the condition
assumed on $\gamma$ in Section \ref{SectionCurves}.
We restate it here.  In what follows,
$\gamma\q( t,x\w): \Q^N\q( \rho\w)\times \Omega''\rightarrow \Omega$
is a $C^\infty$ function with $\gamma\q( 0,x\w)\equiv x$, where $\Omega''$ and $\Omega$ are
as in Section \ref{SectionCurves}.
\begin{itemize}
\item $\cZs$:  Define,
\begin{equation*}
W\q( t\w) = \frac{d}{d\epsilon}\bigg|_{\epsilon=1} \gamma_{\epsilon t}\circ \gamma_t^{-1}\q( x\w).
\end{equation*}
Write $W\q( t\w) \sim \sum_{\alpha} t^\alpha X_\alpha$,
and define 
$$\sV=\q\{\q(X_\alpha,\deg\q( \alpha\w) \w): \deg\q(\alpha\w) \text{ is nonzero in only one component}\w\}.$$
$\cZs$ states that there is a finite subset $\sF\subseteq \sV$ such
that $\sF$ generates a finite list and this finite list controls $\gamma$.
Note that both ``generates a finite list'' and ``controls'' in the above
depend implicitly on the set $\sA\subseteq \q[0,1\w]$, the chosen compact set $\K\Subset \Omega''$, and the multi-parameter dilations $e$.

\item $\cGs$:  Write $\gamma_t\sim \exp\q(\sum_{\q|\alpha\w|>0} t^\alpha \Xh_\alpha \w)$.  $\cGs$ is the same as $\cZs$ except $\sV$ is replaced by
$$\sV=\q\{\q(\Xh_\alpha,\deg\q( \alpha\w) \w): \deg\q(\alpha\w) \text{ is nonzero in only one component}\w\}.$$

\item $\cYs$: For each $j$, $1\leq j\leq N$, define,
\begin{equation*}
W_j\q( t\w) = d\gamma\q( t,\gamma_t^{-1}\q( x\w) \w)\q(\frac{\partial}{\partial t_j}\w)=\frac{d}{d{s_j}}\bigg|_{s=0} \gamma_{t+s}\circ \gamma_t^{-1}\q( x\w).
\end{equation*}
Write $W_j\q( t\w) \sim \sum_{\alpha} t^{\alpha} X_{\alpha,j}$.
$\cYs$ is the same as $\cZs$ except $\sV$ is replaced by
$$\sV=\q\{\q(X_{\alpha,j},\deg\q( \alpha\w)+e_j \w): \deg\q(\alpha\w)+e_j \text{ is nonzero in only one component}\w\},$$
where $e$ denotes the chosen multi-parameter dilations.
\end{itemize}

\begin{prop}\label{PropCurvatureEquivScaleInv}
$\cGs\Leftrightarrow \cYs\Leftrightarrow\cZs$.
\end{prop}
\begin{proof}[Proof sketch]
Using the Campbell-Hausdorff formula, the vector fields
$\Xh_\alpha$ can be written in terms of the
vector fields $X_\alpha$ (respectively, $X_{\alpha,j}$).
This can be seen just as in Proposition 9.6
of \cite{ChristNagelSteinWaingerSingularAndMaximalRadonTransforms}.
It is not hard to see, using the form of this correspondence
(which we have not made precise), that
$\cZs$ and $\cYs$ are both equivalent to $\cGs$.
\end{proof}

\begin{rmk}\label{RmkcJs}
Under the equivalent assumptions $\cZs$, $\cGs$, and $\cYs$,
the assumptions of Proposition \ref{PropHowAssumpsAreUsed} hold.
The conclusion of Proposition \ref{PropHowAssumpsAreUsed}
can be seen as $\cJs$.  
Note that each of $\cZs$, $\cGs$, and $\cYs$ comes with
an intrinsic family of vector fields from which we can construct
the scaling map $\Phi_{x_0,\delta}$ (for each $\delta\in \sA$, $x_0\in \K$), which we use to make sense of
$\cJs$.  There is no such intrinsic family of vector fields
corresponding to $\cJs$, and we therefore cannot discuss any sort
of reverse implication in this sense.  Of course,
if we fix a choice of the list of vector fields $\q(X,d\w)$,
then all four conditions are (in a sense) equivalent,
since the pullback of them under $\Phi$ corresponds
to $\cZu$, $\cGu$, $\cYu$, and $\cJu$, respectively (which
we know to be equivalent).
More specifically, if we define $\theta_t^{x_0,\delta} = \Phi_{x_0,\delta}^{-1} \circ \gamma_{\delta t}\circ \Phi_{x_0,\delta}$
then $\q\{\theta^{x_0,\delta}: x_0\in \K, \delta\in \sA\w\}$
satisfies  $\cZu$, $\cGu$, and $\cYu$, respectively,
if and only if $\gamma$ satisfies 
 $\cZs$, $\cGs$, and $\cYs$ respectively.
We then say that $\gamma$ satisfies $\cJs$ if
$\q\{\theta^{x_0,\delta}: x_0\in \K, \delta\in \sA\w\}$
satisfies $\cJu$.
This idea can also be used to provide another proof
of Proposition \ref{PropCurvatureEquivScaleInv}.
\end{rmk}

	\subsection{Proof of Proposition \ref{PropControlVectsControlCurve}}\label{SectionProofOfPropControlVectsControlCurves}
	In this section, we exhibit the proof
of Proposition \ref{PropControlVectsControlCurve}.
Let $\q( Z,\dt\w)$ be vector fields with single-parameter
formal degrees satisfying the assumptions in
Section \ref{SectionCCII} for a fixed $x_0\in \Omega$; i.e., the
same setup as in Section \ref{SectionCurvesII}.
For each multi-index $\alpha$ with $\q|\alpha\w|\leq L$,
let $V_\alpha$ be a $C^\infty$ vector field.
Though it will not play an essential role in this section,
assign to $V_\alpha$ the formal degree $\q|\deg\q( \alpha\w)\w|_1$.
Let $\gh_t$ be 
given by,
\begin{equation*}
\gh_t\q( x\w) = e^{\sum_{0<\q|\alpha\w|\leq L} t^{\alpha} V_\alpha}x.
\end{equation*}
Proposition \ref{PropControlVectsControlCurve} will follow immediately
by applying the following lemma with $\q( Z,\dt\w) = \q( \delta X,\sd\w)$,
with $\delta\in \sA$, and $x_0\in K$, and taking $\gh_t=\gamma_{\delta t}$.

\begin{lemma}
$\q( Z,\dt\w)$ controls $\gh$ at the unit scale near $x_0$ if and
only if $\q( Z,\dt\w)$ controls $\q( V_\alpha, \q|\deg\q( \alpha\w)\w|_1\w)$
at the unit scale near $x_0$, in the sense that if $\gh$ satisfies
$\sQ_2$ then $\q( V_\alpha, \q|\deg\q( \alpha\w)\w|_1\w)$ satisfies
$\sP_2$ and the parameters of $\sP_2$ can be chosen to be $2$-admissible
constants depending also on the parameters of $\sQ_2$.  Conversely,
if $\q( V_\alpha, \q|\deg\q( \alpha\w)\w|_1\w)$ satisfies $\sP_2$
for each $\alpha$, then $\gh$ satisfies $\sQ_2$ and the parameters
for $\sQ_2$ can be chosen to be $2$ admissible constants
depending also on the parameters for $\sP_2$.
\end{lemma}
\begin{proof}
Let $U_\alpha$ be the pullback of $V_\alpha$ via the map $\Phi$, and
let $\theta_t\q( u\w) = \Phi^{-1}\circ \gh_t\circ \Phi\q( u\w)$.
Note that,
\begin{equation*}
\theta_t\q( u\w) = e^{\sum_{0<\q|\alpha\w|\leq L} t^\alpha U_\alpha}u.
\end{equation*}
Then, the statement of the lemma can be restated as saying,
for some $\rho_2,\eta_2>0$,
\begin{equation*}
\CjN{U_\alpha}{m}{B^{n_0}\q( \eta_2\w)}\lesssim 1, \quad \forall \alpha,\quad \forall m,
\end{equation*}
if and only if,
\begin{equation*}
\CjN{\theta}{m}{B^N\q(\rho_2\w)\times B^{n_0}\q( \eta_2\w)}\lesssim 1, \quad \forall m.
\end{equation*}
This is well known.
\end{proof}

	\subsection{$T^{*}$ is of the same form as $T$}\label{SectionAdjoint}
	In this section, we prove the following proposition.
\begin{prop}\label{PropAdjoint}
If $T$ satisfies all of the assumptions of Theorem \ref{ThmMainThmSecondPass},
then so does $T^{*}$, where $T^{*}$ is the $L^2$ adjoint of $T$.
\end{prop}

Recall, $T$ is of the form,
\begin{equation*}
T\q( f\w) \q( x\w) = \psi_1\q( x\w) \int f\q( \gamma_t\q( x\w)\w) \psi_2\q( \gamma_t\q( x\w)\w) \kappa\q( t,x\w) K\q( t\w)\: dt;
\end{equation*}
see Theorem \ref{ThmMainThmSecondPass} for further details.
Since we are assuming $K\q( t\w)$ has small support (and we are allowed
to choose how small, depending on $\gamma$), we may without loss
of generality assume that $\det \frac{\partial \gamma}{\partial x}\q( t,x\w)\geq \frac{1}{2}$ on the support of $K$ (since $\det \frac{\partial \gamma}{\partial x}\q( 0,x\w)=1$).
Note that,
\begin{equation*}
\begin{split}
T^{*} \q( f\w) \q( x\w) 
&= \overline{\psi_2\q( x\w)} \int \overline{\psi_1\q( \gamma_t^{-1}\q( x\w)\w)\kappa\q( t,\gamma_t^{-1}\q(x\w)\w)} f\q( \gamma_t^{-1}\q( x\w)\w)\q(\det \frac{\partial\gamma}{\partial x}\q( t,\gamma_t^{-1}\q(x\w)\w) \w)^{-1} K\q( t\w) \: dt\\
&= \overline{\psi_2\q( x\w)} \int \overline{\psi_1\q( \gamma_t^{-1}\q( x\w)\w)} \kapt\q( t,\gamma_t^{-1}\q(x\w)\w) f\q( \gamma_t^{-1}\q( x\w)\w) K\q( t\w) \: dt,
\end{split}
\end{equation*}
where $\kapt\q( t,x\w) = \overline{\kappa\q( t,\gamma_t^{-1}\q(x\w)\w)} \q(\det \frac{\partial\gamma}{\partial x}\q( t,\gamma_t^{-1}\q(x\w)\w) \w)^{-1}$.
From here it is immediate to see that Proposition \ref{PropAdjoint}
follows from the following lemma.

\begin{lemma}
$\gamma_t$ satisfies all of the hypotheses of Theorem \ref{ThmMainThmSecondPass}
if and only if $\gamma_t^{-1}$ does.
\end{lemma}
\begin{proof}
We show that $\gamma_t$ satisfies $\cGs$ if and only if $\gamma_t^{-1}$
satisfies $\cGs$, and the result will follow from 
Proposition \ref{PropCurvatureEquivScaleInv}.
Since the result is symmetric in $\gamma_t$ and $\gamma_t^{-1}$ we show
only one direction.

Suppose $\gamma_t$ satisfies $\cGs$.  Thus we assume that our
list of vector fields $\q(X,d\w)$ is generated by vector fields
of the form $\q( \Xh_\alpha, \deg\q( \alpha\w)\w)$, where,
$\gamma_t\q( x\w) \sim \exp\q(\sum_{\q|\alpha\w|>0} t^{\alpha} \Xh_\alpha \w)x$.
However, Lemma 9.3 of
\cite{ChristNagelSteinWaingerSingularAndMaximalRadonTransforms}
shows,
$\gamma_t^{-1}\q( x\w) \sim \exp\q(-\sum_{\q|\alpha\w|>0} t^{\alpha} \Xh_\alpha \w)x$.
Hence we may use the same list of vector fields to show
that $\gamma_t^{-1}$ satisfies $\cGs$.
The result will now follow once we show $\q( X,d\w)$ controls
$\gamma_t^{-1}$.  Rephrasing this,
it suffices to show that $\gh_t^{-1}\q( x\w):= \gamma_{\delta t}^{-1}\q( x\w)$
is controlled by $\q( Z,\dt\w):=\q( \delta X,\sd\w)$ at the unit scale near $x_0$, uniformly for $x_0\in \K$ and $\delta\in \sA$.
However, our assumption on $\gamma$ can be rephrased as saying that
$\gh_t$ is controlled by $\q( Z,\dt\w)$ at the unit scale
near $x_0$, uniformly for $x_0\in \K$ and $\delta\in \sA$.
The result now follows from Proposition \ref{PropControlCompAndInv}.
\end{proof}

\section{The space $L^1_\delta$}\label{SectionL1delta}
In this section, we review the results from 
Sections 7 and 13 of
\cite{ChristNagelSteinWaingerSingularAndMaximalRadonTransforms},
concerning the space $L^1_\delta$ (which is defined below).
The situation we are interested in is as follows.
$\Psi$ will be a $C^\infty$ mapping on the closure $\Qb$
of a finite ball $\Q$ in $\R^d$ mapping to $\R^n$, with $d\geq n$.
Consider a measure $\psi\q( \tau\w) \: d\tau$ in $\R^d$,
where $\psi\in C^1_0\q( \Q\w)$.
We define a measure $\mu$ on $\R^n$ by the integration formula,
\begin{equation*}
\int_{\R^n} f\q( y\w) \:d\mu\q( y\w) = \int_{\Qb} f\q( \Psi\q( \tau\w)\w)\psi\q( \tau\w) \: d\tau;
\end{equation*}
that is, $\mu$ is the transported measure $d\mu=\Psi_{*} \q(\psi d\tau \w)$.

The goal of this section is to discuss the fact that, under appropriate
conditions, $\mu$ is absolutely continuous with respect to Lebesgue
measure and the Radon-Nikodym derivative $h\q( y\w)=\frac{d\mu\q( y\w)}{dy}$
possesses the following level of smoothness.

\begin{defn}
For $0<\delta\leq 1$, $L^1_\delta\q( \R^n\w)$ is the Banach space of
all functions $f\in L^1\q( \R^n\w)$ such that,
\begin{equation}\label{EqnL1deltaN}
\int_{\R^n} \q|h\q( y-z\w) -h\q( y\w)\w|\: dy\leq C\q|z\w|^\delta, \quad \forall z\in \R^n.
\end{equation}
The norm on $L^1_\delta$ is defined to be $\LpN{1}{h}$ plus the smallest
$C$ for which \eqref{EqnL1deltaN} holds.
\end{defn}

\begin{prop}[Proposition 7.2 of \cite{ChristNagelSteinWaingerSingularAndMaximalRadonTransforms}]\label{PropTransportL1delta}
Suppose that, for some multi-index $\alpha$,
\begin{equation}\label{EqnCondForTransport}
\q|\q( \frac{\partial}{\partial\tau}\w)^\alpha \det_{n\times n} \frac{\partial \Psi}{\partial \tau} \q( \tau\w)\w|_\infty \ne 0, \quad\forall \tau\in \Qb.
\end{equation}
Then the transported measure $d\mu = \Psi_{*} \q( \psi d\tau\w)$ is absolutely
continuous with respect to Lebesgue measure, and its Radon-Nikodym
derivative $h$ belongs to $L^1_\delta$ for all $\delta<\q( 2\q|\alpha\w|\w)^{-1}$.
Moreover, the $L^1_\delta$ norm of $h$ can be bounded in terms of
an upper bound for the $C^{\q|\alpha\w|+2}\q( \Qb\w)$ norm of $\Psi$, a lower
bound for the left hand side of \eqref{EqnCondForTransport} in $\Qb$,
an upper bound for the $C^1$ norm of $\psi$, the number $\delta$, and 
an upper bound for $\q|\alpha\w|$.
\end{prop}

Fix constants $C_1,C_2<\infty$ and let $\zeta\in \q( 0,1\w]$.  Consider a non-negative
measure $\Xi$ on $\R^{2n}$ with the following properties.
\begin{itemize}
\item ${\mathrm{supp}} \:\Xi \subseteq \q\{\q( y,z\w) : \q|y\w|,\q|z\w|\leq C_1, \: \q|y-z\w|\leq C_1 \zeta\w\}$.
\item There exist bounded, nonnegative, measurable functions $m_1,m_2$ such that for every
$f\in C^0\q( \R^n\w)$,
\begin{equation*}
\begin{split}
\int\int f\q( y\w) \: d\Xi\q( y,z\w) &= \int f\q( y\w) m_1\q( y\w) \: dy,\\
\int\int f\q( z\w) \: d\Xi\q( y,z\w) &= \int f\q( z\w) m_2\q( z\w) \: dz;\\
\end{split}
\end{equation*}
with,
\begin{equation*}
m_1\q( y\w), m_2\q( z\w) \leq C_2.
\end{equation*}
\end{itemize}

\begin{prop}[See Proposition 13.1 of \cite{ChristNagelSteinWaingerSingularAndMaximalRadonTransforms}]\label{PropIntL1delta}
Suppose $h\in L^1_\delta\q( \R^n\w)$ and $\Xi$ is a measure as
described above.  Then, there exist $\delta',A\in \q(0,\infty\w)$ such that,
\begin{equation*}
\int\q| h\q( y\w) - h\q( z\w) \w|\: d\Xi\q( y,z\w)\leq A\zeta^{\delta'} \q\|h\w\|_{L^1_\delta};
\end{equation*}
where $\delta'$ depends only on $\delta$ and $n$, and $A$ depends only
on $\delta$, $n$, and upper bounds for the constants $C_1,C_2$.
\end{prop}

Proposition \ref{PropIntL1delta} is slightly different than
Proposition 13.1 of \cite{ChristNagelSteinWaingerSingularAndMaximalRadonTransforms}.
Superficially, we have renamed $2^{-\nu}$ in \cite{ChristNagelSteinWaingerSingularAndMaximalRadonTransforms}
to be $\zeta$ here, and taken $j$ and $x$ in \cite{ChristNagelSteinWaingerSingularAndMaximalRadonTransforms} to be $0$.
Furthermore, we have removed any aspect of the ``lifting'' procedure
from \cite{ChristNagelSteinWaingerSingularAndMaximalRadonTransforms}
and we have replaced the quasi-distance $d$ with Euclidean
distance.
Finally, we need not assume $C_1$ small as is done in \cite{ChristNagelSteinWaingerSingularAndMaximalRadonTransforms} (though it would suffice for our
purposes to take $C_1$ small).
In any case, the proof in \cite{ChristNagelSteinWaingerSingularAndMaximalRadonTransforms}
immediately gives Proposition \ref{PropIntL1delta} as we have
stated it.

\section{A general $L^2$ theorem}\label{SectionGenThm}
In this section, we state and prove a general $L^2$ theorem,
which we will use to complete the proof
of Theorem \ref{ThmMainThmSecondPass}.
The goal is to develop a theorem which will imply
\begin{equation}\label{EqnToShowFromGeneral}
\LpOpN{2}{\q(T_j^{*}T_kT_k^{*}T_j\w)^2},
\LpOpN{2}{\q(T_jT_k^{*}T_kT_j^{*}\w)^2}\lesssim 2^{-\epsilon\q|j-k\w|},
\end{equation}
for some $\epsilon>0$,
where $T_j$ and $T_k$ are as in Section \ref{SectionProofOutline}.
We state this result in slightly greater generality than we presently
need, as the more general result will play an essential
role in \cite{SteinStreetMultiParameterSingRadonLp,StreetMultiParameterSingRadonAnal};
in addition, this more general version only requires a small amount
of extra work.
Throughout this section we take $\K\Subset\Omega'\Subset\Omega''\Subset\Omega$,
as in Section \ref{SectionCurves}.
$a>0$ will be a small number to be chosen later, depending on the particular
operators we study.  We will see, though, that $a>0$
may be chosen independent of $j,k$ in \eqref{EqnToShowFromGeneral}.

The setting is as follows.  We are given operators
$S_1,\ldots, S_L$, $R_1$ and $R_2$, and a real number $\zeta\in \q(0,1\w]$.
We will present conditions on these operators
such that there exists $\epsilon>0$ with,
\begin{equation*}
\LpOpN{2}{S_1\cdots S_L \q(R_1-R_2\w)}\lesssim \zeta^{\epsilon}.
\end{equation*}
In Section \ref{SectionProof} we will show that the assumptions
of this section hold uniformly when applied to \eqref{EqnToShowFromGeneral}, in an appropriate sense,
and this will allow us to establish \eqref{EqnToShowFromGeneral},
and complete the proof of Theorem \ref{ThmMainThmSecondPass}.

To describe our assumptions we suppose we are given $C^\infty$
vector fields $Z_1,\ldots, Z_q$ on $\Omega$ with {\it single}-parameter
formal degrees $\dt_1,\ldots, \dt_q$.\footnote{The single-parameter
degrees do not play an essential role in this section, and are only present
to facilitate our applications of the results in this section.}
We assume that $\q( Z,\dt\w)$ satisfies all of the assumptions 
of Theorem \ref{ThmMainCCThm}, uniformly for $x_0\in \K$, for some
fixed $\xi>0$.
Furthermore, we assume that $r$ of the vector fields, $Z_1,\ldots, Z_r$,
$M$-generate $Z_1,\ldots, Z_q$, for some $M>0$ (see
Definition \ref{DefnZ1ZrMGenerates}).

We now turn to defining the operators $S_j$.  We assume, for each $j$,
we are given a $C^\infty$ function $\gh_j:\Q^{N_j}\q( \rho\w)\times \Omega''\rightarrow \Omega$ satisfying $\gh_j\q(0,x\w)\equiv x$.
We assume that each $\gh_j$ is controlled at the unit scale by $\q( Z,\dt\w)$.\footnote{By this we mean that they are controlled by $\q( Z,\dt\w)$ at the unit scale near $x_0$ for every $x_0\in \K$, uniformly in $x_0$.}
As usual, we restrict our attention to $\rho>0$ small, so that
$\gh_{j,t}^{-1}$ makes sense wherever we use it.
We suppose we are given $\psi_{j,1}, \psi_{j,2}\in C_0^{\infty}\q( \R^n\w)$
supported on the interior of $\K$ and $\kappa_j\in C^{\infty}\q(\overline{\Q^{N_j}\q( a\w)}\times \overline{\Omega'} \w)$.  Finally, we suppose we are given
$\vsig_j\in C_0^\infty\q( \Q^{N_j}\q( a\w)\w)$. 
We define,
\begin{equation*}
S_j f\q( x\w) = \psi_{j,1}\q( x\w) \int f\q( \gh_{j,t}\q( x\w)\w) \psi_{j,2}\q( \gh_{j,t}\q( x\w)\w) \kappa_j\q( t,x\w) \vsig_j\q( t\w) \: dt.
\end{equation*}

Note that, under the above assumptions,
\begin{equation}\label{EqnSjBound}
\LpOpN{\infty}{S_j}, \LpOpN{1}{S_j}\lesssim 1;
\end{equation}
and furthermore $S_j^{*}$ is of the same form as $S_j$ with
$\gh_{j,t}$ replaced by $\gh_{j,t}^{-1}$ (c.f. Section \ref{SectionAdjoint}).

\begin{defn}\label{DefnOpControlUnitScale}
If $S_j$ is of the above form, we say that $S_j$ is controlled by
$\q( Z,\dt\w)$ at the unit scale.
\end{defn}

\begin{rmk}\label{RmkControlAdjointAndComp}
Since $\gh_{j,t}$ is controlled at the unit scale if and only
if $\gh_{j,t}^{-1}$ is (Proposition \ref{PropControlCompAndInv})
$S_j$ is controlled at the unit scale if and only if $S_j^{*}$
is.  Furthermore, if $\gh_{j_1,t_1}$ and $\gh_{j_2,t_2}$ are 
controlled at the unit scale, then so is $\gh_{j_1,t_1}\circ \gh_{j_2,t_2}$ (Proposition \ref{PropControlCompAndInv}),
and therefore, if $S_{j_1}$ and $S_{j_2}$ are controlled at the unit
scale, then so is $S_{j_1}S_{j_2}$.
\end{rmk}

We assume, further, that for each $l$, $1\leq l\leq r$, there
is a $j$ ($1\leq j\leq L$), and a multi-index $\alpha$ (with
$\q|\alpha\w|\leq B$, where $B\in \N$ is some fixed constant
which our results are allowed to depend on),
such that,
\begin{equation*}
Z_l\q( x\w) = \frac{1}{\alpha!}\frac{\partial}{\partial t}^\alpha\bigg|_{t=0} \frac{d}{d\epsilon}\bigg|_{\epsilon=1} \gh_{j,\epsilon t}\circ \gh_{j,t}^{-1}\q( x\w).
\end{equation*}
This concludes our assumptions on $S_1,\ldots, S_L$.

\begin{rmk}\label{RmkWhyTheSAssump}
Note that $\gh_1,\ldots, \gh_L$ were chosen so that
\begin{equation}\label{EqnGenL2Defgh}
\gh_{\q( t_1,\ldots, t_L,s_1,\ldots,s_L\w)}\q( x\w) = \gh_{L,t_L}\circ\gh_{L-1,t_{L-1}}\circ\cdots \circ \gh_{1,t_1}\circ \gh_{1,s_1}^{-1}\circ \gh_{2,s_2}^{-1}\circ\cdots\circ \gh_{L,s_L}^{-1}\q( x\w).
\end{equation}
would satisfy the hypotheses (and therefore the conclusions) of
Theorem \ref{ThmCompSatisfycZ} and Corollary \ref{CorCompSatisfycJ}.

\end{rmk}

We now turn to the operators $R_1$ and $R_2$.  It is here where the number
$\zeta$ plays a role.
We assume we are given a $C^\infty$ function $\gt_{t,s}$ (with $\gt_{0,0}\q(x\w)\equiv x$),
which is controlled by $\q( Z,\dt\w)$ at the unit scale:\footnote{In most of our
applications, $\gt_{t,s}$ will be of the form $\gh_{ s\cdot t}$, where
$\gh_t$ is controlled by $\q( Z,\dt\w)$ at the unit scale and $s \cdot t$
scales each coordinate of $t$ by a non-negative integer power of $s$.}
\begin{equation*}
\gt\q(t,s,x\w):\Q^{\Nt}\q( \rho\w)\times \q[-1,1\w] \times \Omega''\rightarrow \Omega, \quad \gt_{0,0}\q( x\w) \equiv x.
\end{equation*}
\begin{rmk}
Here, we are thinking of $\q( t,s\w)$ as playing the role
of the $t$ variable in the definition of control.
\end{rmk}
We suppose we are given 
$\kapt\q( t,s,x\w)\in C^\infty\q( \overline{\Q^{\Nt}\q( a\w)}\times \q[ -1,1\w]\times \Omega'' \w)$, 
$\vsigt\q(t\w) \in L^1\q( \Q^N\q( a\w)\w)$,
and $\psit_1,\psit_2\in C_0^\infty\q( \R^n\w)$
supported on the interior of $\K$.
We define, for $\xi\in \q[-1,1\w]$,
\begin{equation*}
R^\xi f\q( x\w)= \psit_1\q( x\w) \int f\q(\gt_{t,\xi}\q( x\w) \w)\psit_2\q( \gt_{t,\xi}\q( x\w)\w)
\kapt\q( t,\xi,x\w)\vsigt\q( t\w)\: dt.
\end{equation*}
Note that we have,
\begin{equation}\label{EqnRBounded}
\LpOpN{1}{R^{\xi}},\LpOpN{\infty}{R^{\xi}}\lesssim 1.
\end{equation}
We set $R_1= R^{\zeta}$ and $R_2=R^{0}$.


\begin{thm}\label{ThmGenL2Thm}
In the above setup, if $a>0$ is chosen sufficiently small, we have,
\begin{equation*}
\LpOpN{2}{S_1\cdots S_L \q(R_1-R_2\w)}\leq C \zeta^{\epsilon},
\end{equation*}
for some $\epsilon>0$.
\end{thm}

\begin{rmk}\label{RmkWhatConstsDependOn}
It is important that $a$, $C$, and $\epsilon$ may be chosen independent
of any relevant parameters.  
We say merely, that $a$, $C$, and $\epsilon$ can be chosen to
depend only on
the norms of the various functions used to define $S_j$, $R_1$, and $R_2$,
on the parameters $B$ and $M$, on $L$, on the various
dimensions, on the parameters in the definition of control (when using $\sQ_2$),
and on anything that $2$-admissible constants were allowed
to depend on as in Section \ref{SectionCCII}.
The reader wishing to, should have no trouble keeping
track of the various dependencies in our argument.
\end{rmk}

The rest of this section is devoted to the proof
of Theorem \ref{ThmGenL2Thm}.  The reader not immediately interested in the details may safely skip
to Section \ref{SectionProof} to see how Theorem \ref{ThmGenL2Thm}
is used.

\begin{rmk}
Note that in the equivalence between $\sQ_1$ and $\sQ_2$, we used that
we were able to shrink the $t$ variable.  For instance, to
define $\theta$ in the definition of $\sQ_2$, we restricted
attention to $\q|t\w|<\eta'$, where $\eta'\approx 1$.
Thus, we must be able to restrict attention
to $s\leq c$ where $c\gtrsim 1$ when working with $\gt_{t,s}$.  Since,
in the definition $R_1$ and $R_2$ we fix $s$ to be $\zeta$ and $0$,
respectively, it suffices to note that we may restrict attention
to $\zeta\leq c$ in Theorem \ref{ThmGenL2Thm}.  This is clear,
since when $1\geq \zeta\geq c$, Theorem \ref{ThmGenL2Thm}
is trivial.  Henceforth, we assume $\zeta$ is small enough
that all our definitions make sense.
\end{rmk}

Define $S=S_1\cdots S_L$ and $R=R_1-R_2$.  In what follows, $\epsilon>0$ will
be a positive number that may change from line to line.
We wish to show,
\begin{equation*}
\LpOpN{2}{SR}\lesssim \zeta^{\epsilon},
\end{equation*}
and it suffices to show,
\begin{equation*}
\LpOpN{2}{R^{*}S^{*}SR}\lesssim \zeta^{\epsilon}.
\end{equation*}
Using \eqref{EqnRBounded}, we have $\LpOpN{2}{R^{*}}\lesssim 1$,
and therefore it suffices to show,
\begin{equation*}
\LpOpN{2}{S^{*}SR}\lesssim \zeta^{\epsilon}.
\end{equation*}
Continuing in this manner, it suffices to show,
\begin{equation}\label{EqnToShowDiadic}
\LpOpN{2}{\q( S^{*} S\w)^{2^m} R}\lesssim \zeta^{\epsilon},
\end{equation}
for some $m>0$.  Since $\LpOpN{2}{S},\LpOpN{2}{S^*}\lesssim 1$ by
\eqref{EqnSjBound}, it suffices to show,
\begin{equation}\label{EqnToShowL2}
\LpOpN{2}{\q( S^{*}S\w)^n R}\lesssim \zeta^{\epsilon};
\end{equation}
where we have just taken $m$ so large $2^m\geq n$ and applied
\eqref{EqnToShowDiadic}.

Combining \eqref{EqnSjBound} and \eqref{EqnRBounded} we see,
\begin{equation*}
\LpOpN{1}{\q( S^{*}S\w)^n R}\lesssim 1,
\end{equation*}
and so interpolation shows that to prove \eqref{EqnToShowL2}
we need only show,
\begin{equation}\label{EqnTwoShowLinf}
\LpOpN{\infty}{\q( S^{*}S\w)^n R}\lesssim \zeta^{\epsilon}.
\end{equation}
Let $f$ be a bounded measurable function.  Rephrasing \eqref{EqnTwoShowLinf},
we wish to show,
\begin{equation}\label{EqnToShowx0}
\q| \q( S^{*}S\w)^n R f\q( x_0\w)\w|\lesssim \zeta^\epsilon\LpN{\infty}{f},
\end{equation}
for every $x_0\in \K$ (where we have used that fact that $S^{*}g$ is supported
in $\K$ for every $g$).
We now fix $x_0$ and prove \eqref{EqnToShowx0}.  All implicit constants
in what follows can be chosen to be independent of $x_0\in \K$.

Define, $\gh_t$ by \eqref{EqnGenL2Defgh}.  That is,
\begin{equation*}
\gh_{\q( t_1,\ldots, t_L,s_1,\ldots,s_L\w)}\q( x\w) = \gh_{L,t_L}\circ\gh_{L-1,t_{L-1}}\circ\cdots \circ \gh_{1,t_1}\circ \gh_{1,s_1}^{-1}\circ \gh_{2,s_2}^{-1}\circ\cdots\circ \gh_{L,s_L}^{-1}\q( x\w).
\end{equation*}
Thus, $\gh$ is controlled at the unit scale by $\q(Z,\dt\w)$, and
$S^{*}S$ is given by,
\begin{equation*}
S^{*}S f\q( x\w) = \psi_1\q( x\w) \int f\q( \gh_{t}\q( x\w)\w) \psi_2\q(\gh_{t}\q( x\w) \w) \kappa\q(t,x\w) \vsig\q(t\w)\: dt,
\end{equation*}
where $\vsig\in C_0^1\q( \Q^{N}\q( a'\w)\w)$, $a'>0$ is a small number depending
on $a$ (from here on out, $a'>0$ will be a small number (depending on $a>0$)
that may change from line to line), $N=\sum_{j=1}^L 2 N_j$, $\kappa\in C^\infty$, and
$\psi_1,\psi_2\in C_0^\infty$ are supported in the interior of $\K$.
Define, for $\tau=\q( t_1,\ldots, t_n\w)$ ($t_j\in \Q^N\q( a'\w)$),
\begin{equation*}
\Gamma_\tau\q( x\w) = \gh_{t_1}\circ \gh_{t_2}\circ\cdots \circ\gh_{t_n}\q( x\w).
\end{equation*}
So that,
\begin{equation*}
\q( S^{*}S\w)^n f\q( x\w) = \psi_1\q( x\w) \int f\q( \Gamma_{\tau}\q( x\w)\w) \psi_2\q( \Gamma_\tau\q( x\w)\w) \kappa\q( \tau,x\w) \vsig\q( \tau\w)\: d\tau,
\end{equation*}
where the various functions have changed but are of the same basic form as before.
Thus,
\begin{equation*}
\begin{split}
\q( S^{*} S \w)^n R f\q( x\w) = 
&\psi_1\q( x\w) \int 
f\q(\gt_{t,\zeta}\circ \Gamma_\tau \q( x\w) \w) \psit_2\q( \gt_{t,\zeta}\circ \Gamma_\tau\q( x\w) \w)
\kappa\q( t,\zeta,\tau,x\w) \vsig\q( \tau\w)\vsigt\q( t\w)\: d\tau\: d t\\
&-\psi_1\q( x\w) \int 
f\q(\gt_{t,0}\circ \Gamma_\tau \q( x\w) \w) \psit_2\q( \gt_{t,0}\circ \Gamma_\tau\q( x\w) \w)
\kappa\q( t,0,\tau,x\w) \vsig\q( \tau\w)\vsigt\q( t\w)\: d\tau\: d t.
\end{split}
\end{equation*}
Here, $\kappa$ and $\psi_1,\psit_2$ are $C^\infty$, with
$\psi_1$ and $\psit_2$ supported on the interior of $\K$.
Now think of $x_0\in \K$ as fixed.  We wish to establish
\eqref{EqnToShowx0}.  The dependence of
$\kappa$ on $x_0$ is unimportant so we suppress it.  Given 
a bounded measurable function $f$, we wish to
study the integral,
\begin{equation*}
I\q( f\w) = \int  f\q(\gt_{t,\zeta}\circ \Gamma_\tau \q( x_0\w) \w) 
\kappa\q( t,\zeta,\tau\w) \vsig\q( \tau\w)\vsigt\q( t\w)\: d\tau\: d t
- \int  f\q(\gt_{t,0}\circ \Gamma_\tau \q( x_0\w) \w) 
\kappa\q( t,0,\tau\w) \vsig\q( \tau\w)\vsigt\q( t\w)\: d\tau\: d t.
\end{equation*}
Note that $\psi_1\q( x_0\w) I\q( f \psit_2\w)=\q( S^{*}S\w)^n R f\q( x_0\w)$.
Theorem \ref{ThmGenL2Thm} now follows immediately from the
following proposition,
\begin{prop}\label{PropBoundI}
For $a>0$ sufficiently small, there exists $\epsilon>0$ and $C$ such that,
\begin{equation*}
\q|I\q( f\w)\w|\leq C \zeta^{\epsilon}  \sup_{z\in \B{Z}{\dt}{x_0}{\xi}} \q|f\q(z\w)\w|;
\end{equation*}
where $a,\epsilon,$ and $C$ may only depend on the parameters
the constants of the same names were allowed to depend on
in Theorem \ref{ThmGenL2Thm}; see Remark \ref{RmkWhatConstsDependOn}.
\end{prop}

We devote the remainder of this section to the proof
of Proposition \ref{PropBoundI}.
First note that it suffices to consider only $\kappa$
of the form,
\begin{equation*}
\kappa\q( t,\zeta,\tau\w)=\kappa_1\q( \tau\w)\kappa_0\q( t,\zeta\w),
\end{equation*}
since every $\kappa$ may be written as a rapidly converging
sum of such terms.  (This reduction is unnecessary, but simplifies
notation in what follows.)

Proposition \ref{PropControlCompAndInv} shows that
$\q( Z,\dt\w)$ controls $\gt_{t,s}\circ \Gamma_\tau$
at the unit scale near $x_0$,
if $a>0$ is sufficiently small (and therefore $t,s$ and $\tau$ are
sufficiently small), then
$\gt_{t,s}\circ \Gamma_\tau\q(x_0\w)\in \B{Z}{\dt}{x_0}{\xi_1}$ (see $\sQ_2$).  Hence, $I$ only depends on
the values of $f$ on $\B{Z}{\dt}{x_0}{\xi_1}$.
We may, therefore, restrict our attention to $f$ of the form
$g\circ \Phi$, where $g$ is a function on $B^{n_0}\q( \eta_1\w)$;
as the result would follow by taking $g=f\circ\Phi^{-1}$.
$\Phi$ here is, as usual, the scaling map from Theorem \ref{ThmMainCCThm},
and depends on the particular point $x_0$.  Note, in particular,
$\Phi\q( 0\w) =x_0$.

Define,
\begin{equation*}
\thetah_t = \Phi^{-1}\circ \gh_t\circ \Phi,
\end{equation*}
\begin{equation*}
\Theta_\tau = \Phi^{-1}\circ \Gamma_\tau\circ \Phi = \thetah_{t_1}\circ\thetah_{t_2}\circ\cdots\circ\thetah_{t_n},\quad \tau=\q( t_1,\ldots, t_n\w),
\end{equation*}
\begin{equation*}
\thetat_{t,s} = \Phi^{-1}\circ \gt_{t,s}\circ \Phi,
\end{equation*}
We see,
\begin{equation}\label{EqnIofG}
\begin{split}
I\q( g\circ \Phi\w) = 
&\int  g\q(\thetat_{t,\zeta}\circ \Theta_\tau \q( 0\w) \w) 
\kappa_0\q( t,\zeta\w) \kappa_1\q( \tau\w) \vsig\q( \tau\w)\vsigt\q( t\w)\: d\tau\: dt\\
&-\int  g\q(\thetat_{t,0}\circ \Theta_\tau \q( 0\w) \w) 
\kappa_0\q( t,0\w) \kappa_1\q( \tau\w) \vsig\q( \tau\w)\vsigt\q(t\w)\: d\tau\: dt.
\end{split}
\end{equation}
Note that, since $\kappa_0$ is $C^\infty$, we have $\kappa_0\q( t,\zeta\w) = \kappa_0 \q( t,0\w)+ O\q( \zeta\w)$ uniformly in $t$.  Combining this
with \eqref{EqnIofG}, it is easy to see,
\begin{equation}\label{EqnIofGWithError}
\begin{split}
I\q( g\circ \Phi\w) = 
&\int  \q[g\q(\thetat_{t,\zeta}\circ \Theta_\tau \q( 0\w) \w) 
-g\q(\thetat_{t,0}\circ \Theta_\tau \q( 0\w) \w)\w] 
\kappa_0\q( t,0\w) \kappa_1\q( \tau\w) \vsig\q( \tau\w)\vsigt\q( t\w)\: d\tau\: dt\\
&+ O \q( \zeta \sup_{v\in B^{n_0}\q( \eta_1\w)} \q|g\q( v\w)\w|\w).
\end{split}
\end{equation}
Note that the error term in \eqref{EqnIofGWithError} is of the
desired form.  Thus, let $\Ih\q( g\w)$ be the first term
on the right hand side of \eqref{EqnIofGWithError}.  It
suffices to bound $\q|\Ih\q( g\w)\w|$.

Note that, in light of Remark \ref{RmkWhyTheSAssump},
Corollary \ref{CorCompSatisfycJ} shows that
$\thetah$ satisfies $\cJ$ uniformly in any relevant parameters,
i.e.,
there exists a multi-index $\beta$ (with $\q|\beta\w|\lesssim 1$)
such that,\footnote{Actually,
when applied to $\thetah$, $\cJ$ only requires we take a composition of
$n_0$ terms, instead of the $n$ terms used to define $\Theta$.  Since
$n\geq n_0$ the result still holds.}
\begin{equation*}
\q| \q(\frac{\partial}{\partial \tau} \w)^\beta \det_{n_0\times n_0}\frac{\partial \Theta}{\partial \tau} \q( \tau, 0\w)\bigg|_{\tau=0}  \w|\gtrsim 1.
\end{equation*}

Since,
\begin{equation*}
\q\|\Theta\w\|_{C^{\q|\beta\w|+1}}\lesssim 1, \text{ by } \sQ_2,
\end{equation*}
we see that if we take $a>0$ sufficiently small,
\begin{equation*}
\q| \q(\frac{\partial}{\partial \tau} \w)^\beta \det_{n_0\times n_0}\frac{\partial \Theta}{\partial \tau} \q( \tau, 0\w) \w|\gtrsim 1,
\end{equation*}
for all $\tau$ in the support of $\vsig$.

Applying Proposition \ref{PropTransportL1delta},
with $\Psi\q( \tau\w) = \Theta_\tau\q( 0\w)$ and
$\psi\q( \tau\w) = \kappa_1\q( \tau\w) \vsig\q( \tau\w)$,
we see that there exists $\delta\gtrsim 1$
and $h\in L^1_\delta\q( \R^n\w)$ (with $\q\|h\w\|_{L^1_\delta}\lesssim 1$)
such that,\footnote{Of course, $h$ is supported on the range $\Theta_\tau\q( 0\w)$, and therefore supported in $B^{n_0}\q( \eta'\w)$.}
\begin{equation*}
\Ih\q( g\w) = \int \q[ g\q(\thetat_{t,\zeta}\q( u\w)\w)-g\q( \thetat_{t,0}\q( u\w)\w)\w]
h\q( u\w) \kappa_0\q( t,0\w) \vsigt\q( t\w)\: du\: dt.
\end{equation*}
Applying two changes of variables, we have,
\begin{equation*}
\begin{split}
\Ih\q( g\w) = &
\int g\q( v \w) \q(\det \frac{\partial\thetat_{t,\zeta}^{-1}}{\partial v}\q(v\w)\w) h\q(\thetat_{t,\zeta}^{-1}\q( v\w) \w) \kappa_0\q( t,0\w) \vsigt\q( t\w)\: dv\: d t\\
&-\int g\q( v \w) \q(\det \frac{\partial\thetat_{t,0}^{-1}}{\partial v}\q(v\w)\w) h\q(\thetat_{t,0}^{-1}\q( v\w) \w) \kappa_0\q(t,0\w) \vsigt\q(t\w)\: dv\: dt\\
\end{split}
\end{equation*}
Using that $\thetat\q( t,s,u\w)$ is $C^\infty$ (uniformly in any relevant
parameters, by $\sQ_2$), we have
\begin{equation*}
 \q(\det \frac{\partial\thetat_{t,\zeta}^{-1}}{\partial v}\q(v\w)\w)  = \q(\det \frac{\partial\thetat_{t,0}^{-1}}{\partial v}\q(v\w)\w) + O\q( \zeta\w),
\end{equation*}
and therefore,
\begin{equation*}
\begin{split}
\Ih\q( g\w) = &\int g\q( v\w) \q(\det \frac{\partial\thetat_{t,0}^{-1}}{\partial v}\q(v\w)\w) \q(  h\q(\thetat_{t,\zeta}^{-1}\q( v\w) \w) -  h\q(\thetat_{t,0}^{-1}\q( v\w) \w)   \w) \kappa_0\q( t,0\w) \vsigt\q(t\w)\: dv\: dt\\
&+O\q( \zeta\w) \sup_{v\in B^{n_0}\q( \eta'\w)} \q|g\q( v\w)\w|.
\end{split}
\end{equation*}
Using $\int\q|\kappa_0\q(t,0\w)\vsigt\q(t\w)\w|\lesssim 1$, we have,
\begin{equation*}
\q|\Ih\q( g\w)\w| \lesssim \q\{\sup_{\q|t\w|\leq a}\q[ \int_{B^{n_0}\q( \eta'\w) } \q|h\q(\thetat_{1,t_1}^{-1}\q( v\w) \w) -  h\q(\thetat_{2,t_2}^{-1}\q( v\w) \w)   \w|   \: dv\w] + \zeta\w\} \sup_{v\in B^{n_0}\q( \eta'\w)} \q| g\q( v\w)\w|
\end{equation*}
The proof will now be completed by showing, for every $\q|t\w|\leq a$ (think of $t$ as fixed),
\begin{equation}\label{EqnToShowEndOfGenProof}
\int \q| h\q(\thetat_{t,\zeta}^{-1}\q( v\w) \w) -  h\q(\thetat_{t,0}^{-1}\q( v\w) \w)   \w| \: dv \lesssim \zeta^\epsilon,
\end{equation}
for some $\epsilon>0$.
Define a measure $\Xi$ by,
\begin{equation*}
\int k\q( y,z\w) \: d\Xi\q( y,z\w) = \int_{B^{n_0}\q( \eta'\w)} k\q( \thetat_{t,\zeta}^{-1}\q( v\w),\thetat_{t,0}^{-1}\q( v\w)    \w) \: dv;
\end{equation*}
so that the left hand side of \eqref{EqnToShowEndOfGenProof} becomes,
\begin{equation*}
\int \q| h\q( y\w) - h\q( z\w)\w|\:d \Xi\q( y,z\w).
\end{equation*}
Since $\theta_{t,\zeta}^{-1}$ depends smoothly on $\zeta$,
\begin{equation*}
\q| \thetat_{t,\zeta}^{-1}\q( v\w) -\thetat_{t,0}^{-1}\q( v\w)  \w|\lesssim \zeta.
\end{equation*}
Hence $\Xi$ is supported on those $\q(y,z\w)$ such that $\q|y-z\w|\lesssim \zeta$.
Applying Proposition \ref{PropIntL1delta}, we see,
\begin{equation*}
\int \q|h\q( y\w) -h\q( z\w)\w| d\Xi\q( y,z\w)\lesssim \zeta^{\epsilon},
\end{equation*}
for some $\epsilon\gtrsim 1$.  This establishes \eqref{EqnToShowEndOfGenProof}
and completes the proof.

\section{Completion of the proof}\label{SectionProof}
In this section, we complete the proof of Theorem \ref{ThmMainThmSecondPass},
using Theorem \ref{ThmGenL2Thm}.
We use the same notation as in Section \ref{SectionProofOutline},
and we therefore take $T_j$ (for $j\in \lA$) as given by
\eqref{EqnDefnTj}.
The goal is to show that the operator,
\begin{equation*}
\sum_{j\in \lA} T_j
\end{equation*}
converges in the strong operator topology as bounded operators
on $L^2$.  This follows from the Cotlar-Stein lemma and the
following result,
\begin{prop}\label{PropToShowEndOfProof}
\begin{equation*}
\LpOpN{2}{T_k^{*}T_j}, \LpOpN{2}{T_j T_k^{*}}\lesssim 2^{-\epsilon\q|j-k\w|},
\end{equation*}
for some $\epsilon>0$.
\end{prop}

This section is devoted to the proof of Proposition \ref{PropToShowEndOfProof}.
The proof is essentially the same for the two terms, and so we only exhibit
the proof for $\LpOpN{2}{T_k^{*} T_j}$.  Throughout, $\epsilon>0$ may
change from line to line, but will always be independent of $j,k$.
We will show,
\begin{equation}\label{EqnToShowEndOfProof}
\LpOpN{2}{\q(T_j^{*} T_k T_k^{*} T_j \w)^2}\lesssim 2^{-\epsilon \q|j-k\w|},
\end{equation}
and the proof will be complete (with $\epsilon$ replaced by $\epsilon/4$).

\begin{rmk}
In Section \ref{SectionProofOutline}, instead of squaring $T_j^{*}T_kT_k^{*}T_j$,
we raised it to the $n+1$ power.  This will be implicit in this section,
via our application of Theorem \ref{ThmGenL2Thm}.  Actually, the proof
of Theorem \ref{ThmGenL2Thm} will (essentially) raise $T_j^{*}T_kT_k^{*}T_j$
to a power larger than $n+1$. However, the proof could be
easily modified (in this special case) to raise it only to the $n+1$ power.  This is not important
for our purposes, though.
\end{rmk}

We set $j_0=j\wedge k\in \lA$.  Let $\q( X,d\w)$ be the list of vector fields
from Section \ref{SectionCurves}, and define
$\q( Z,\dt\w) = \q( 2^{-j_0} X,\sd\w)$.  We will be using
this $\q( Z,\dt\w)$ in our application of Theorem \ref{ThmGenL2Thm}.
Take $\mu_0$ so that,
\begin{equation*}
\q| j_{\mu_0}-k_{\mu_0}\w| = \q|j-k\w|_\infty.
\end{equation*}
We assume first $j_{\mu_0}\geq k_{\mu_0}$ and remark on changes necessary
to deal with the 
reverse situation at the end of the section.
We assume $j_{\mu_0}>k_{\mu_0}$ as otherwise $\q|j-k\w|=0$
and \eqref{EqnToShowEndOfProof} is trivial (it is easy
to see that $\LpOpN{2}{T_j}, \LpOpN{2}{T_k}\lesssim 1$).

First, we introduce the operators $S_l$
which we will use in our application of Theorem \ref{ThmGenL2Thm}.
We set,
\begin{equation*}
S_1\cdots S_7 = T_j^{*} T_k T_k^{*} T_j T_j^{*} T_k T_k^{*}.
\end{equation*}
Thus our goal is to show,
\begin{equation*}
\LpOpN{2}{S_1\cdots S_7 T_j}\lesssim 2^{-\epsilon\q|j-k\w|}.
\end{equation*}
First we claim that each $S_j$ is controlled at the unit
scale by $\q( Z,\dt\w)$ (see Definition \ref{DefnOpControlUnitScale}).
We prove just the result for $T_j$, the same proof works for $T_k$,
and the claim follows since $S_l$ is controlled at the unit scale
if and only if $S_l^{*}$ is (see Remark \ref{RmkControlAdjointAndComp}).
By a simple change of variables in \eqref{EqnDefnTj}, we obtain,
\begin{equation}\label{EqnTChangeVar}
T_j f\q( x\w) = \psi_1\q( x\w) \int f\q( \gamma_{2^{-j} t}\q( x\w)\w) \psi_2\q( \gamma_{2^{-j}t}\q( x\w)\w) \kappa\q( 2^{-j} t, x\w) \vsig_j\q( t\w) \: dt.
\end{equation}
Proposition \ref{PropConnectionBetweenControlEveryScaleAndUnit}
shows that $\q( 2^{-j_0} X,\sd\w)$ controls $\gamma_{2^{-j_0}t}$ at the
unit scale.
Since $j_0\leq j$ coordinatewise, Proposition \ref{PropScalingOfControl}
shows that $\q( 2^{-j_0} X,\sd\w)$ controls $\gamma_{2^{-j}t}$
at the unit scale.  That $T_j$ is controlled at the unit
scale by $\q( 2^{-j_0} X,\sd\w)$ now follows immediately.

We now turn to the other part of our assumption on the $S_l$:  that there
are vector fields $Z_1,\ldots, Z_r$ which generate $Z_1,\ldots, Z_q$
that appear at the Taylor coefficients of
$$\frac{d}{d\epsilon}\bigg|_{\epsilon=1} \gh_{l,\epsilon t}\circ \gh_{l, t}^{-1},$$
where $\gh_l$ is the function defining $S_l$, for some $l$.
For this purpose, we need only use $S_2=T_k$ and $S_4=T_j$.
In light of \eqref{EqnTChangeVar}, we see
$\gh_{2,t}=\gamma_{2^{-k}t}$, $\gh_{4,t}=\gamma_{2^{-j}t}$.
Now the result follows from Lemma \ref{LemmaEachZjAppearsIngammaScaled},
where we have taken $2^{-j_0}=\delta_0$, $2^{-j}=\delta_1$, and
$2^{-k}=\delta_2$.  Thus, $S_1,\ldots, S_7$ satisfy
all the hypotheses of Theorem \ref{ThmGenL2Thm}.

Before we define $R_1$ and $R_2$, we need to make some preliminary remarks.
First, for the reader only interested in the class of kernels $\sKt$,
decompose $t\in \R^N$ as $t=\q( t_1,t_2\w)$, where $t_1=t^\mu$ (as
defined in Section \ref{SectionClasssKt}).
Note, by the definition of $t^\mu$, when $t$ is scaled (i.e.,
when we consider $2^{j-j_0}t$)
all of the coordinates of $t_1$ are scaled by (at least) a factor
of $2^{c\q( j_{\mu_0}-k_{\mu_0}\w)}=2^{c\q|j-k\w|_\infty}$,
where $c$ is some constant depending on the dilations $e$.
The reader uninterested
in $\sK$ may skip the rest of this paragraph.
Since $j_{\mu_0}>k_{\mu_0}$,
it follows that $\mu_0$ is not $j,\sA$-minimal (see
Section \ref{SectionClasssK}).
Letting $C\geq 1$ be as in Definition \ref{DefnsK},
take $\mu_1$ such that $\mu_0\preceq_{j,C,\sA} \mu_1$
and there does not exist $\mu_2$ with $\mu_1\prec_{j,C,\sA} \mu_2$.
By definition, for every $\mu\in \q[\mu_1\w]_{j,C,\sA}$, we have,
\begin{equation}\label{EqnEndOfProofjmuminuskmu}
j_{\mu}-k_{\mu}\gtrsim j_{\mu_0}-k_{\mu_0} \geq \q|j-k\w|_{\infty}\gtrsim \q|j-k\w|.
\end{equation}
For $t\in \R^N$, decompose $t=\q( t_1, t_2\w)$,
where $t_1=t_1^{\q[\mu_1\w]_{j,C,\sA}}$ and $t_2=t_2^{\q[\mu_1\w]_{j,C,\sA}}$
(as in Section \ref{SectionClasssK}).
By the definition of $t_1$, when $t$ is scaled (i.e., when
we consider $2^{j-j_0}t$)
all of the coordinates of $t_1$ are scaled by (at least)
a factor of $2^{c'\q( j_\mu-k_\mu\w)}$ for some
$\mu\in \q[\mu_1\w]_{j,C,\sA}$ (where $c'$ depends on the dilations $e$). 
By \eqref{EqnEndOfProofjmuminuskmu},
each coordinate of $t_1$ is scaled by at least
a factor of $2^{c\q( j_{\mu_0}-k_{\mu_0}\w)}=2^{c\q|j-k\w|_\infty}$.

Note by the definition of $t=\q( t_1,t_2\w)$ in the previous
paragraph, we have (by our assumptions on $K$),
\begin{equation*}
\int \vsig_j\q( t\w) \: dt_1 =0.
\end{equation*}
In addition, we have from the discussion in the previous paragraph,
if $\q| 2^{j-j_0} t\w|\leq 1$, then $2^{c\q|j-k\w|_\infty}\q|t_1\w|+\q|t_2\w|\leq a$.
I.e., $\dil{\vsig_j}{2^{j-j_0}}\q( t\w)$ is supported
on $2^{c\q|j-k\w|_\infty}\q|t_1\w|+\q|t_2\w|\leq a$.
Letting $\zeta=2^{-c\q|j-k\w|_\infty}$, we see that we may rewrite,
\begin{equation*}
\dil{\vsig_j}{2^{j-j_0}}\q( t_1,t_2\w) = \frac{d \zeta^{-1}t_1}{dt_1} \vsigt\q( \zeta^{-1}t_1,t_2\w),
\end{equation*}
where $\zeta^{-1} t_1$ given by the usual multiplication,
$\frac{d \zeta^{-1} t_1}{dt_1}$ is the Radon-Nikodym derivative
of this multiplication (and therefore is equal to $\zeta^{-m}$ where
$m$ is the dimension of the $t_1$ space),
and $\vsigt\in L^1\q( \Q^N\q( a\w)\w)$, uniformly in any relevant parameters.

Now consider,
\begin{equation*}
\begin{split}
T_j f\q( x\w) = 
&\psi_1\q( x\w)\int f\q(\gamma_{2^{-j_0}\q(t_1,t_2\w)}\q( x\w) \w)\psi_2\q( \gamma_{2^{-j_0}\q(t_1,t_2\w)}\q( x\w)\w) \kappa\q( 2^{-j_0}\q(t_1,t_2\w), x\w) \dil{\vsig_j}{2^{j-j_0}}\q( t_1,t_2\w) \: dt_1\: dt_2\\
&=\psi_1\q( x\w)\int \bigg[ f\q(\gamma_{2^{-j_0}\q(t_1,t_2\w)}\q( x\w) \w)\psi_2\q( \gamma_{2^{-j_0}\q(t_1,t_2\w)}\q( x\w)\w) \kappa\q( 2^{-j_0}\q(t_1,t_2\w), x\w) \\
&\quad -f\q(\gamma_{2^{-j_0}\q(0,t_2\w)}\q( x\w) \w)\psi_2\q( \gamma_{2^{-j_0}\q(0,t_2\w)}\q( x\w)\w) \kappa\q( 2^{-j_0}\q(0,t_2\w), x\w) \bigg]
\dil{\vsig_j}{2^{j-j_0}}\q( t_1,t_2\w) \: dt_1\: dt_2\\
&=\psi_1\q( x\w)\int \bigg[ f\q(\gamma_{2^{-j_0}\q(\zeta t_1,t_2\w)}\q( x\w) \w)\psi_2\q( \gamma_{2^{-j_0}\q(\zeta t_1,t_2\w)}\q( x\w)\w) \kappa\q( 2^{-j_0}\q(\zeta t_1,t_2\w), x\w) \\
&\quad -f\q(\gamma_{2^{-j_0}\q(0 t_1,t_2\w)}\q( x\w) \w)\psi_2\q( \gamma_{2^{-j_0}\q(0 t_1,t_2\w)}\q( x\w)\w) \kappa\q( 2^{-j_0}\q(0 t_1,t_2\w), x\w) \bigg]
\vsigt\q( t_1,t_2\w) \: dt_1\: dt_2.
\end{split}
\end{equation*}
This shows that $T_j$ is of the form $R_1-R_2$; indeed,
we merely take $\gt_{t_1,t_2,s} = \gamma_{2^{-j_0}\q(s t_1,t_2 \w)}$
and $\kapt\q(t_1,t_2,s,x \w)= \kappa\q(2^{-j_0} \q(s t_1,t_2 \w),x\w)$.
It is easy to verify that these choices satisfy the relevant
hypotheses,\footnote{It is immediate to verify that
$\gamma_{2^{-j_0}\q( st_1,t_2\w)}$ satisfies $\sQ_2$,
since $\gamma_{2^{-j_0}\q( t_1,t_2\w)}$ does.}
and we leave the details to the interested reader.

Theorem \ref{ThmGenL2Thm} applies to show,
\begin{equation*}
\LpOpN{2}{S_1\cdots S_7 \q( R_1-R_2\w)}\lesssim \zeta^{-\epsilon} = 2^{-\epsilon'\q|j-k\w|_\infty} \lesssim 2^{-\epsilon''\q|j-k\w|},
\end{equation*}
for some $\epsilon, \epsilon',\epsilon''>0$.
This verifies \eqref{EqnToShowEndOfProof} and completes the proof,
in the case when $j_{\mu_0}\geq k_{\mu_0}$.

When $k_{\mu_0}\geq j_{\mu_0}$ the proof is similar.
To bound,
\begin{equation*}
\LpOpN{2}{\q(T_j^{*} T_k T_k^{*} T_j \w)^2},
\end{equation*}
we use the fact that $\LpOpN{2}{T_j}, \LpOpN{2}{T_k}\lesssim 1$, and
instead bound,
\begin{equation*}
\LpOpN{2}{T_j^{*}T_k T_k^{*} T_j T_j^{*} T_k}.
\end{equation*}
We take, $S_1\cdots S_5 = T_j^{*}T_k T_k^{*} T_j T_j^{*}$
and we define $R_1-R_2=T_k$ in exactly
the same way as we did for $T_j$ above.
The rest of the proof follows in exactly the same manner as above.
This completes the proof of Proposition \ref{PropToShowEndOfProof}
and therefore of Theorem \ref{ThmMainThmSecondPass}.

\section{Kernels revisited}\label{SectionKernelsII}
In this section, we further discuss the class of kernels
$\sK\q( N,e,a,\sA\w)$ defined in Section \ref{SectionKernels}.
None of the results here are used elsewhere in the paper, but
we hope they will give the reader a better understanding
of the class of kernels we use.

\begin{lemma}\label{LemmaSumsConvg}
Every sum of the form \eqref{EqnDefsK} converges in the sense of
distributions.
\end{lemma}
\begin{proof}
Let $C$, $\sA$, $a$, and $\q\{\vsig_j\w\}$ be as in
Definition \ref{DefnsK}.
Fix $f\in C_0^\infty\q( \R^N\w)$.  It suffices to show
\begin{equation*}
\int \dil{\vsig_j}{2^j}\q( t\w) f\q( t\w) \: dt = O\q( 2^{-\epsilon\q|j\w|} \w),
\end{equation*}
for some $\epsilon>0$.
Fix $k=\q(k_1,\ldots, k_\nu\w)\in \sA$.  Note that there is an $\epsilon_0>0$ sufficiently 
small, such that for all but finitely many $j=\q( j_1,\ldots, j_\nu\w)\in \lA$,
there is a $\mu_0=\mu_0\q( j\w)$ such that
\begin{equation}\label{EqnConvgInDistEqn1}
\q|j-k\w|\leq \epsilon_0 \q( j_{\mu_0} - k_{\mu_0}\w).
\end{equation}
We ignore those finitely many $j$ such that \eqref{EqnConvgInDistEqn1} does
not hold, and we also assume $j\ne k$.
Fixing $j$ and $\mu_0$ such that \eqref{EqnConvgInDistEqn1} holds,
let $\mu_1$ be such that $\mu_0\preceq_{j,C,\sA} \mu_1$ and such that
there does not exist $\mu$ with $\mu_1 \prec_{j,C,\sA} \mu$.
Note that there exists $\epsilon_1>0$ (independent of $j$) such that
\begin{equation}\label{EqnConvgInDistEqn2}
0<\q|j-k\w| \leq \epsilon_0 \q( j_{\mu_0}-k_{\mu_0} \w)\leq \epsilon_1 \q( j_\mu - k_\mu \w), \text{ for every }\mu\in \q[\mu_1\w]_{j,C,\sA}. 
\end{equation}
In fact, $\epsilon_1=\frac{\epsilon_0}{C^2}$ will do.
In light of \eqref{EqnConvgInDistEqn2}, we see that $\mu_1$ is not
$j,\sA$-minimal (since $\mu_1\in \q[\mu_1\w]_{j,C,\sA}$).  Thus, we have,
\begin{equation*}
\int \dil{\vsig_j}{2^j}\q( t\w) \: dt_1^{\q[\mu_1\w]_{j,C,\sA} }=0.
\end{equation*}

Furthermore, on the support of $\dil{\vsig}{2^j}\q(t\w)$,
there is a $\mu\in \q[\mu_1\w]_{j,C,\sA}$ 
and $\epsilon_2,\epsilon_3>0$ such that
\begin{equation*}
\q|t_1^{\q[\mu_1\w]_{j,C,\sA}}\w|_\infty \leq a 2^{-\epsilon_2 \q( j_\mu-k_\mu\w)} \leq a2^{-\epsilon_3\q|j-k\w|}\lesssim 2^{-\epsilon_3 \q|j\w|},
\end{equation*}
where in the last inequality we have used that $k$ is fixed.
Thus, we see
\begin{equation*}
\begin{split}
\q|\int \dil{\vsig}{2^j}\q( t\w) f\q( t\w)\: dt\w| &= \q|\int \dil{\vsig}{2^j}\q( t\w) \q[f\q( t_1^{\q[\mu_1\w]_{j,C,\sA} }, t_2^{\q[\mu_1\w]_{j,C,\sA} } \w)-
f\q( 0, t_2^{\q[\mu_1\w]_{j,C,\sA} } \w)\w]\w|\: dt\\
&\lesssim 2^{-\epsilon_3 \q|j\w|} \int \q|\dil{\vsig}{2^j}\q( t\w) \w|\: dt\\
&= 2^{-\epsilon_3 \q|j\w|} \int \q|\vsig\q( t\w)\w| \: dt\\
&\lesssim 2^{-\epsilon_3\q|j\w|},
\end{split}
\end{equation*}
completing the proof.
\end{proof}

Next we turn to investigating the case when $\sA=\q[0,1\w]^\nu$.
In this case, we have the following: 
\begin{lemma}\label{LemmaCancelInFullCase}
Taking the same notation as in Definition \ref{DefnsK}, with
$\sA=\q[0,1\w]^\nu$, the cancellation condition \eqref{EqnDefnsKCancel}
holds if and only if
\begin{equation}
\int \vsig_j\q( t\w) \: dt^\mu = 0, \text{ unless } j_\mu=0,
\end{equation}
where $t^\mu$ is as in the start of Section \ref{SectionClasssKt}.
\end{lemma}
\begin{proof}
This follows from a straightforward application of the definitions.
Indeed, in this case $\lA=\N^\nu$.  The result follows from noting
that if $j=\q( j_1,\ldots, j_\nu\w)\in \N^\nu$, then $\mu$
is $j,[0,1]^\nu$-minimal if and only if $j_\mu=0$.  Furthermore,
if $j_\mu\ne 0$, then it is easy to see that there is no
$\mu'\ne \mu$ such that $\mu \preceq_{j,C,\sA} \mu'$, for any $C\geq 1$.
\end{proof}

\begin{prop}
$\delta_0\in \sK\q( N,e,a,\q[0,1\w]^\nu\w)$.
\end{prop}
\begin{proof}
Fix $\eta\in C_0^\infty\q(\Q^N\q( a' \w)\w)$ with $\int \eta =1$, where $a'>0$ is a small number
to be chosen in a moment.

For $j\in \N^\nu$, define $\vsig_j$ by
\begin{equation}\label{EqnSumForD0}
\dil{\vsig_j}{2^j} = \sum_{\substack{p\in \q\{0,1\w\}^\nu\\ j-p\in \N^\nu}} \q(-1\w)^{\sum_\mu p_\mu} \dil{\eta}{2^{j-p}}.
\end{equation}
It is clear that, for $a'>0$ sufficiently small, $\q\{\vsig_j\w\}\subset C_0^\infty\q( \Q^N\q( a\w)\w)$ is a bounded set.

Fix $j\in \N^\nu$ and $\mu$ such that $j_\mu\ne 0$.  We wish to show
\begin{equation*}
\int \vsig_j\q( t\w) \: dt^\mu=0,
\end{equation*}
and to do so, it suffices to show
\begin{equation}\label{EqnD0IntVanish}
\int \dil{\vsig_j}{2^j} \q( t\w)\: dt^\mu =0.
\end{equation}
Consider the sum \eqref{EqnSumForD0}.  For each choice of $\mu'\ne \mu$ and
$p_{\mu'}\in \q\{0,1\w\}$ (with $j_{\mu'}-p_{\mu'}\geq 0$), the sum \eqref{EqnSumForD0} contains two terms:  one for $p_\mu =1$ and one for $p_\mu=0$.
These two terms have opposite signs, and therefore the sum of their
integrals against $dt^\mu$ is 0.  \eqref{EqnD0IntVanish} follows.

To complete the proof, we must show
$$ \sum_{j\in \N^\nu} \dil{\vsig_j}{2^j} = \delta_0.$$
In light of Lemma \ref{LemmaSumsConvg}, it suffices to show,
for each $m>0$,
\begin{equation}\label{EqnToShowLimitDist}
\sum_{\q|j\w|_\infty \leq m} \dil{\vsig_j}{2^j} =\dil{\eta}{2^{m},\ldots, 2^m},
\end{equation}
since $\dil{\eta}{2^m,\ldots, 2^m}\rightarrow \delta_0$ in distribution as 
$m\rightarrow \infty$.

To see \eqref{EqnToShowLimitDist}, fix $j\in \N^\nu$ with at least one coordinate
less than $m$ and the rest less than or equal to $m$.  Let $\nu_0$
denote the number of coordinates of $j$ that are equal to $m$.
It is easy to see that the coefficient of $\dil{\eta}{2^j}$
in the sum \eqref{EqnToShowLimitDist} is equal to
\begin{equation*}
\sum_{\mu=0}^{\nu-\nu_0} \q( -1\w)^\mu \binom{\nu-\nu_0}{\mu} =0,
\end{equation*}
whereas the coefficient of $\dil{\eta}{2^m,\ldots, 2^m}$ is 1.
This completes the proof of \eqref{EqnToShowLimitDist} and therefore
the proposition.
\end{proof}

Lemma \ref{LemmaCancelInFullCase} will allow us to see product
kernels as the special case
of $\sK\q( N,e,a,\q[0,1\w]^\nu\w)$ when each $e_j$ is nonzero
in precisely one coordinate.  We begin by introducing
the notion of a product kernel.
Our main reference is
\cite{NagelRicciSteinSingularIntegralsWithFlagKernels} and we refer
the reader there for more details.  The definition for product kernels
is recursive on the number of products, $\nu$, with $\nu=1$ corresponding
to the classical Calder\'on-Zygmund kernels.
To define product kernels, suppose we are given a decomposition
of $\R^N$, $\R^N=\R^{N_1}\times \cdots \times \R^{N_\nu}$, into
$\nu$ homogeneous subspaces with given single-parameter dilations
on each subspace.  That is, for each $\mu$ we are given
$e_1^\mu,\ldots, e^{\mu}_{N_{\mu}}\in \q(0,\infty\w)$.  For $\delta\geq 0$
and $t=\q( t_1,\ldots, t_{N_\mu}\w)\in \R^{N_\mu}$, we define
$\delta t=\q(\delta^{e_1^\mu}t_1,\ldots, \delta^{e_{N_\mu}^\mu} t_\mu\w)$.
This, therefore, defines $\nu$ parameter dilations on $\R^N$:  if
$\delta=\q( \delta_1,\ldots, \delta_\nu\w)\in \q[0,\infty\w)^\nu$ and
$t=\q( t^1,\ldots, t^\nu\w)\in \R^{N_1}\times \cdots \times \R^{N_\nu}=\R^N$,
we define
\begin{equation}\label{EqnMultiParamDil}
\delta t = \q(\delta_1t^1,\ldots, \delta_\nu t^\nu \w).
\end{equation}
\begin{defn}[See Definition 2.1.1 of \cite{NagelRicciSteinSingularIntegralsWithFlagKernels}]\label{DefnProdKer}
A product kernel on $\R^N$ relative to the above decomposition is a distribution
$K$ on $\R^N$ which coincides with a $C^\infty$ function away from
the coordinate subspaces $t_\mu=0$ and which satisfies:
\begin{enumerate}
\item (Differential inequalities)  For each multi-index $\alpha=\q(\alpha_1,\ldots, \alpha_\nu \w)$,\footnote{We write $\alpha_\mu=\q(\alpha_\mu^{1},\ldots,\alpha_\mu^{N_\mu} \w)$.}
 there is a constant $C_\alpha$ so that
\begin{equation*}
\q|\partial_{t^1}^{\alpha_1}\cdots \partial_{t^\nu}^{\alpha_\nu} K\q( t\w)\w| \leq C_\alpha \q|t^1\w|^{-\sum \q(1+\alpha_1^j\w)e^1_j} \cdots \q|t^\nu\w|^{-\sum \q(1+\alpha_\nu^j\w)e^\nu_j  },
\end{equation*}
away from the coordinate subspaces.
\item (Cancellation conditions)  These are defined recursively on $\nu$,
the case $\nu=1$ corresponds to the usual Calder\'on-Zygmund operators.
\begin{enumerate}
\item For $\nu=1$, given any $C^\infty$ function $\phi$ supported on the unit ball
with $C^1$-norm bounded by $1$, and any $R>0$,
\begin{equation*}
\int K\q( t\w) \phi\q( Rt\w)\: dt
\end{equation*}
is bounded independently of $\phi$ and $R$.
\item For $\nu>1$, given any $\mu$, $1\leq \mu\leq \nu$, any $C^\infty$ function
$\phi$ supported on the unit ball of $\R^{N_\mu}$
with $C^1$ norm bounded by $1$, and any $R>0$, the distribution
\begin{equation*}
K_{\phi,R}\q( t^1,\ldots, t^{\mu-1}, t^{\mu+1},\ldots,t^{\nu}\w) = \int K\q( t\w) \phi\q( R t^\mu\w) \: d t^\mu
\end{equation*}
is a product kernel on the lower dimensional space which is a product of the
$\R^{N_{\mu'}}$ for $\mu'\ne \mu$, uniformly in $\phi$ and $R$.
\end{enumerate}
\end{enumerate}
\end{defn}

Let $\eh$ denote the multi-parameter dilations induced by
\eqref{EqnMultiParamDil}.  Then we have the following result:
\begin{prop}\label{PropCharacterizeProdKer}
Suppose $K$ is a distribution supported in $\Q^{N}\q( a\w)$.  Then,
$K$ is a product kernel as in Definition \ref{DefnProdKer} if and
only if $K\in \sK\q(N,\eh, a, \q[0,1\w]^\nu\w)$.
\end{prop}
\begin{proof}
This follows easily from Corollary 2.2.2 of \cite{NagelRicciSteinSingularIntegralsWithFlagKernels}.
\end{proof}

We now turn to the notion of flag kernels.  Flag kernels provide an
easy to understand situation where one uses a choice of $\sA$
other than $\q[0,1\w]^{\nu}$.
For flag kernels, we take the same setup as above (flag kernels are 
special cases of product kernels).  Here we think of the factors
$\R^{N_1}\times \cdots \times \R^{N_\nu}$ as being ordered (for the
product kernels, the order of the factors did not matter).
Again, we write $t=\q( t^1,\ldots, t^\nu\w) \in \R^{N_1}\times \cdots\times \R^{N_\nu}$.
\begin{defn}[See Definition 2.3.2 of \cite{NagelRicciSteinSingularIntegralsWithFlagKernels}]\label{DefnFlagKer}
A flag kernel is a distribution $K$ on $\R^N$ which coincides with
a $C^\infty$ function away from $t^\nu=0$ and satisfies
\begin{enumerate}
\item (Differential inequalities)  For each multi-index
$\alpha=\q( \alpha_1,\ldots, \alpha_\nu\w)$\footnote{Again,
we write $\alpha_\mu=\q(\alpha_\mu^1,\ldots, \alpha_\mu^{N_\mu} \w)$.}
there is a constant $C_\alpha$ such that for $t^\nu\ne 0$, we have
\begin{equation*}
\begin{split}
\q|\partial_{t^1}^{\alpha_1} \cdots \partial_{t^\nu}^{\alpha_\nu} K\q( t\w)\w|\leq &C_\alpha \q(\q|t^1\w|+\cdots+\q|t^\nu\w|\w)^{-\sum\q(1+\alpha_1^j\w)e_j^1}\\
&\cdots \q(\q|t^{\nu-1}\w|+\q|t^\nu\w|\w)^{-\sum\q(1+\alpha_{\nu-1}^j\w)e_{j}^{\nu-1}}\q|t^\nu\w|^{-\sum\q(1+\alpha_{\nu}^j\w)e_{j}^{\nu}}.
\end{split}
\end{equation*}
\item (Cancellation conditions)  These are defined recursively on $\nu$.
\begin{enumerate}
\item For $\nu=1$, given any $C^\infty$ function $\phi$ supported on
the unit ball, with $C^1$ norm bounded by $1$, and any $R>0$,
\begin{equation*}
\int K\q( t\w) \phi\q( Rt\w)\: dt
\end{equation*}
is bounded uniformly for $\phi$ and $R$.
\item For $\nu>1$, given any $\mu$, $1\leq \mu\leq \nu$, any $C^\infty$ function
$\phi$ supported on the unit ball of $\R^{N_\mu}$ with $C^1$ norm
bounded by $1$, and any $R>0$, the distribution
\begin{equation*}
K_{\phi,R}\q( t^1,\ldots, t^{\mu-1},t^{\mu+1},\ldots, t^\nu\w) = \int K\q( t\w) \phi\q( Rt^\mu\w) \: dt^\mu
\end{equation*}
is a flag kernel on $\R^{N_1}\times \cdots\times \R^{N_{\mu-1}}\times \R^{N_{\mu+1}}\times \cdots \times \R^{N_\nu}$, uniformly in $\phi$ and $R$.
\end{enumerate}
\end{enumerate}
\end{defn}

We state the following result without proof.  It is a result of Nagel, Ricci,
Stein, and Wainger.
\begin{prop}
Let $\sA=\q\{\q(\delta_1,\ldots,\delta_\nu \w)\in \q[0,1\w]^\nu: \delta_1\geq \delta_2\geq \cdots\geq \delta_\nu\w\}$.
Suppose $K$ is a distribution supported on $\Q^N\q( a\w)$.  Then
$K$ is a flag kernel as above if and only if $K$ 
is an element of $\sK\q( N, \eh, a, \sA\w)$.
Here, $\eh$ are the multi-parameter dilations defined earlier.
\end{prop}

We now present an interesting example of kernels, which are
neither product kernels nor flag kernels, but are still
of the form $\sK\q( N,e,a,\q[0,1\w]^\nu\w)$.
We work in the case $N=3$ and we think of $\R^3$ as being
identified with the three-dimensional Heisenberg group,
$\Ho$ (though our example can be easily generalized to any
stratified Lie group).
If we write $\q( x,y,t\w) \in \R^3=\Ho$ as coordinates, then
the group law of $\Ho$ is given by
\begin{equation}\label{EqnHeisGroupLaw}
\q( x,y,t\w) \q( x',y',t'\w) = \q( x+x', y+y', t+t'+ 2\q(yx'-xy'\w)\w).
\end{equation}
We take $\nu=2$ and consider the two parameter dilations given by
$e_1=\q( 1,0\w)$, $e_2=\q( 0,1\w)$, and $e_3=\q( 1,1\w)$.\footnote{The
resulting kernels are not product kernels, since $e_3$ is nonzero
in more than one component.  These dilations are sometimes referred
to as Zygmund dilations--see \cite{FeffermanPipherMultiparameterOperatorsAnSharpWeightedInequalities}.}
That is, we take the two parameter dilations given by,
\begin{equation}\label{EqnHeisTwoParamDilations}
\q( \delta_1,\delta_2\w) \q( x,y,t\w) = \q(\delta_1 x,\delta_2 y, \delta_1\delta_2 t \w).
\end{equation}
It is easy to see that these dilations are automorphisms of the Heisenberg
group.
It follows from our theory that if $K\in \sK\q( 3,e,a, \q[0,1\w]^2\w)$,
then the operator given by
$f\mapsto f*K$ is bounded on $L^2$ where the convolution is taken
in the sense of $\Ho$.  We will see in Section \ref{SectionExampOnHeis}, though, that when the convolution
is taken in the sense of the usual Euclidean structure on $\R^3$,
then the operator may not be bounded on $L^2$.

\section{Examples}\label{SectionExamples}
In this section, we present a number of examples to help
elucidate the various aspects of Theorem \ref{ThmMainThmSecondPass}.

%
%
%
%
%
%
%

	\subsection{Exponentials of vector fields}\label{SectionExExpOfVect}
	Let $\sA\subseteq\q[0,1\w]^\nu$, $\K\Subset \Omega$, and $e$
be as in the statement of Theorem \ref{ThmMainThmSecondPass}.
For each multi-index, $\alpha$, $0<\q|\alpha\w|\leq M$, 
let $X_\alpha$ be a $C^\infty$ vector field on $\Omega$.
Define a function,
\begin{equation}\label{EqnExFiniteSumFormOfGamma}
\gamma_t\q( x\w) = e^{\sum_{0<\q|\alpha\w|\leq M} t^\alpha X_\alpha}x.
\end{equation}
We say $\gamma$ satisfies $\cV$, if the following conditions
are satisfied:
\begin{itemize}
\item $\q\{\q( X_\alpha, \deg\q( \alpha\w)\w) : \deg\q( \alpha\w) \text{ is nonzero in precisely one component}\w\}$ generates a finite list $\q( X,d\w)$.
See Definition \ref{DefnGeneratesAFiniteList}.
\item The above obtained list $\q( X,d\w)$ controls $\q( X_\alpha,\deg\q( \alpha\w)\w)$ for every $\alpha$ with $\deg\q( \alpha\w)$ nonzero in more
than one component (and thus for every $\alpha$).
\end{itemize}

\begin{prop}\label{PropGammaSatisV}
$\gamma$ satisfies $\cV$ if and only if $\gamma$ satisfies the
assumptions of Theorem \ref{ThmMainThmSecondPass}.
\end{prop}
\begin{proof}
This is a combination of Propositions \ref{PropCurvatureEquivScaleInv} and \ref{PropControlVectsControlCurve}.
\end{proof}

	\subsection{Nilpotent Lie groups}\label{SectionExampNilpotent}
	In this example, we investigate a special case of the previous
example.
In this example we take $\sA=\q[0,1\w]^\nu$, and take $\K\Subset \Omega$
and $e$ as in the previous example.  Recall, we call $\alpha$
a {\it pure power} if $\deg\q( \alpha\w)$ is nonzero in precisely
one component.  Fix $M$ large and for each $\alpha$, $0<\q|\alpha\w|\leq M$, which is a pure power,
let $X_\alpha$ be a $C^\infty$ vector field.  For the non-pure powers,
let $X_\alpha=0$.  Assume furthermore, that the iterated commutators
of the $X_\alpha$ are eventually $0$.  Then, if $\gamma$
is given by \eqref{EqnExFiniteSumFormOfGamma}, $\gamma$
trivially satisfies $\cV$ and therefore satisfies
the assumptions of Theorem \ref{ThmMainThmSecondPass}.

To give a concrete example of the above, let $X_1,\ldots, X_l$
be left invariant vector fields on a nilpotent Lie group.
Consider the operator given by,
\begin{equation}\label{EqnExNilOp}
f\mapsto \psi\q( x\w) \int f\q( e^{t_1 X_1+t_2X_2+\cdots+t_lX_l}x\w)\: K\q( t_1,\ldots, t_l\w)\: dt,
\end{equation}
where $K\q( t_1,\ldots, t_l\w)$ is a product kernel relative to the
decomposition of $\R^l=\R\times \R\times \cdots\times \R$ (and $K$
is supported sufficiently close to $0$).  Then, if $\psi\in C_0^\infty$
has small support,
the operator given by \eqref{EqnExNilOp} is automatically bounded
on $L^2$.

Furthermore, by a simple scaling argument, if the Lie group has an appropriate
family of dilations under which the vector fields are homogeneous (for instance if it is a stratified Lie group),
then one may replace $\psi$ with $1$ on the right hand side of
\eqref{EqnExNilOp} and remove the restriction that $K$ have small (or
even compact) support.\footnote{By this we mean we can take any product kernel
$K$ as in Definition \ref{DefnProdKer}, without restricting our attention
to those kernels with small support.  This is the same as considering
kernels in $\sK\q(N,e,a,\q[0,\infty\w)^\nu\w)$ where $e$ is
chosen as in Proposition \ref{PropCharacterizeProdKer} and
$a>0$ is any real number.  In this
case, $\sK$ is the same as $\sKt$.}

Note that we have not involved terms like $t_1t_2 X$ in the exponential.
These would not be pure powers, and so $\cV$ would not hold vacuously:
one would need an additional condition on the vector fields to ensure
that $\cV$ holds.
In Section \ref{SectionExTransInv} we will discuss a translation invariant operator
in Euclidean space of the form \eqref{EqnExNilOp}, except with
non-pure powers as well, where $\cV$ does not hold, and moreover 
which is {\it not} bounded on $L^2$.

\begin{rmk}
We will see in \cite{StreetMultiParameterSingRadonAnal} that the vector fields
$X_1,\ldots, X_l$ can be replaced with {\it any} real analytic vector fields;
not just left invariant vector fields on a nilpotent
Lie group.  In particular, they could be replaced by left invariant vector fields
on any Lie group.
\end{rmk}

	\subsection{The $L^2$ result of Christ, Nagel, Stein, and Wainger}\label{SectionEXCNSW}
	In this section, we show that the $L^2$ result
of \cite{ChristNagelSteinWaingerSingularAndMaximalRadonTransforms}
(the $p=2$ part of Theorem \ref{ThmCNSWBound}) is a special case
of the main result of this paper (Theorem \ref{ThmMainThmSecondPass}).\footnote{Similar remarks hold for every $1<p<\infty$, and this will be shown in
\cite{SteinStreetMultiParameterSingRadonLp}.}
Thus, we are considering the single-parameter case $\nu=1$.
We take $\sA=\q[0,1\w]$ and fix single parameter dilations $e=\q( e_1,\ldots, e_N\w)$, $e_j\in \q( 0,\infty\w)$, and $\K\Subset \Omega$ as
in Theorem \ref{ThmMainThmSecondPass}.  In
\cite{ChristNagelSteinWaingerSingularAndMaximalRadonTransforms},
$e_j$ is taken to be $1$ for every $j$, but this is not
essential to their work.

The condition assumed in
\cite{ChristNagelSteinWaingerSingularAndMaximalRadonTransforms}
on $\gamma$
was that if $X_\alpha$ was given as in \eqref{EqnDefnXalpha},
then $\q\{X_\alpha\w\}$ satisfies H\"ormander's condition
at every point $x\in \K$.
See $\cZ$ in Section \ref{SectionCNSW}.

\begin{prop}
Under the above hypothesis, $\cZ$, $\gamma$
satisfies the hypotheses of Theorem \ref{ThmMainThmSecondPass}.
\end{prop}
\begin{proof}
Let $\q\{\q(X_1,d_1\w),\ldots, \q( X_r,d_r\w)\w\}$ be a finite
subset of $\q\{\q(X_\alpha, \deg\q( \alpha\w) \w)\w\}$
such that $X_1,\ldots, X_r$ satisfy H\"ormander's condition.
Assume that $X_1,\ldots, X_r$ satisfy H\"ormander's condition
of order $m$.  That is, $X_1,\ldots, X_r$ along with their
commutators up to order $m$ span the tangent space
at each point of $\K$.

We recursively define single-parameter formal degrees on the commutators
of $\q( X_\alpha, \deg\q( \alpha\w)\w)$ in the usual way:  if $Y$
has formal degree $d_1$ and $Z$
has formal degree $d_2$, we assign to $\q[Y,Z\w]$ the formal degree $d_1+d_2$.
Enumerate the list of all commutators of $X_1,\ldots, X_r$ up to
order $m$ along with their above defined formal degrees,
$$\q( X_1,d_1\w),\ldots, \q( X_L,d_L\w);$$
so that $X_1,\ldots, X_L$ span the tangent space at each point.
Let $\q( X_1,d_1\w),\ldots, \q( X_q, d_q\w)$
be an enumeration of all the vector fields that can be written
as a commutator of the $X_\alpha$ and such that their above defined
formal degree is less
than or equal to $\max_{1\leq l\leq L} d_l$.\footnote{We maintain
the notation that $\q( X_1,d_1\w),\ldots, \q( X_L,d_L\w)$ are the
first $L$ of these vector fields with formal degrees.}
It is easy to see that there are only a finite number
of such vector fields.  Note, we are using in an essential
way that the degrees are numbers, unlike in the multi-parameter
case, where they are vectors.

We claim that, for $1\leq i,j\leq q$,
\begin{equation}\label{EqnToShowNSWInteg}
\q[X_i,X_j\w]=\sum_{d_k\leq d_i+d_j} c_{i,j}^k X_k, \quad c_{i,j}^k\in C^\infty.
\end{equation}
Indeed, if $d_i+d_j\leq \max_{1\leq l\leq L} d_l$, then $\q( \q[X_i,X_j\w], d_i+d_j\w)$ is already in the list $\q( X_1,d_1\w),\ldots, \q( X_q,d_q\w)$ by
definition.  On the other hand, if $d_i+d_j> \max_{1\leq l\leq L} d_l$,
we use the fact that,
\begin{equation*}
\q[X_i,X_j\w]=\sum_{k=1}^L c_{i,j}^k X_k,
\end{equation*}
since $X_1,\ldots, X_L$ span the tangent space at each point.
It was shown in Section \ref{SectionSpecialCase} that the list
$\q( X,d\w)= \q( X_1,d_1\w),\ldots, \q( X_q,d_q\w)$
satisfies $\sD\q( \K, \q[0,1\w]\w)$ if it satisfies \eqref{EqnToShowNSWInteg}.

Hence, to complete the proof of the proposition, we need only
show that the list $\q( X,d\w)$ controls $\gamma$.
Let $W\q( t,x\w)$ be as in Definition \ref{DefnControlEveryScale}, so that
the vector fields $X_\alpha\q(x\w)$ are the Taylor coefficients of $W\q(t,x\w)$,
when the Taylor series is taken in the $t$ variable.
Then we have,
\begin{equation}\label{EqnToShowWControlApprox}
W\q( \delta t, x\w) = \sum_{\q|\alpha\w|\leq M} t^{\alpha} \delta^{\deg\q( \alpha\w)} X_\alpha + O\q( \q|\delta t\w|^M\w).
\end{equation}
Note, each term,
\begin{equation}\label{EqnToShowWControlX}
\delta^{\deg\q( \alpha\w)} X_\alpha = \sum_{k=1}^q c_k^{\delta} X_k, \quad c_k^\delta\in C^\infty \text{ uniformly in }\delta.
\end{equation}
Indeed, if $\deg\q( \alpha\w) \leq \max_{1\leq l\leq L}d_l$, then
$\q( X_\alpha, \deg\q( \alpha\w)\w)$ appears in the list
$\q( X_1,d_1\w),\ldots, \q( X_q,d_q\w)$, by definition.
On the other hand, if $\deg\q( \alpha\w)>\max_{1\leq l\leq L}d_l$,
then we use the fact that $X_1,\ldots, X_L$ span the tangent space,
and write,
\begin{equation}\label{EqnXalphaAsXlSum}
X_\alpha = \sum_{l=1}^L c_l X_l, \quad c_l\in C^\infty.
\end{equation}
Multiplying \eqref{EqnXalphaAsXlSum} by $\delta^{\deg\q( \alpha\w)}$
immediately yields \eqref{EqnToShowWControlX}.

Thus, each term of the sum on the right hand side of \eqref{EqnToShowWControlApprox}
has the desired form for the definition of control.  To complete the proof,
we need only show that the term $O\q( \q|\delta t\w|^M\w)$ has
the desired form, provided $M$ is sufficiently large.
Take $M$ so large that,
\begin{equation*}
M\min_{1\leq j \leq N} e_j \geq \max_{1\leq l\leq L} d_l.
\end{equation*}
Then, 
$$\q|\delta t\w|^M=O\q( \delta^{\max_{1\leq l\leq L} d_l } \w).$$
Using one last time that $X_1,\ldots, X_L$ span the tangent space,
we see that we may write
the $O\q( \q|\delta t\w|^M\w)$ term in the form,
\begin{equation*}
\sum_{k=1}^L c_k^{\delta}\q( t,x\w) \delta^{\max_{1\leq l\leq L} d_l }  X_k = \sum_{k=1}^L \ct_k^\delta \delta^{d_k} X_k,
\end{equation*}
with $\ct_k^\delta\in C^\infty$ uniformly in $\delta$.
This completes the proof that $\gamma$ is controlled by $\q( X,d\w)$,
and therefore the proof of the proposition.
\end{proof}

\begin{rmk}
As remarked in Section \ref{SectionSpecialCase}, even in the single-parameter
case, Theorem \ref{ThmMainThmSecondPass} is more general than 
the $L^2$ result from
\cite{ChristNagelSteinWaingerSingularAndMaximalRadonTransforms}.
Indeed, the result in this paper also applies in some situations
when $\gamma$ lies in the leaves of a (possibly singular)
foliation.  To apply the methods of
\cite{ChristNagelSteinWaingerSingularAndMaximalRadonTransforms}
directly to a foliated manifold, one would need the foliation
to be non-singular.  See Section \ref{SectionSpecialCase}
for more details.
\end{rmk}

	\subsection{How the assumptions can fail}\label{SectionExAssumpFail}
	It is perhaps instructive to understand examples of $\gamma$
where our assumptions do not hold.  In this section,
we exhibit $\gamma$ which exemplify the various ways
the assumptions can fail.  We make no claim about whether
or not the associated operators are bounded on
$L^2$, and only claim that our theorem does not apply.

In the single parameter case ($\nu=1$, $\sA=\q[0,1\w]$),
there are three ways in which $\gamma$ might fail
to satisfy our assumptions.
\begin{enumerate}
\item  The vector fields $X_{\alpha}$ (see \eqref{EqnDefnXalpha}) might
fail to generate a finite list.  In the single-parameter 
case, this is essentially equivalent to failing
to generate a locally finitely generated involutive distribution (see
Lemma \ref{LemmaSpecialCaseEquivFiniteGen}).
For instance, 
let $X_1=\partial_x$, $X_2=e^{-\frac{1}{x^2}}\partial_y$,
then the function
$\gamma\q( t,\q(x,y\w)\w):\R\times \R^2\rightarrow \R$ associated\footnote{See
Proposition \ref{PropGetGammaFromW} for the bijective correspondence
between $\gamma$ and the vector field $W$.}
 to
the vector field $W$ given by
\begin{equation*}
W\q( t,\q(x,y\w)\w)  = tX_1+t^2 X_2,
\end{equation*}
has this property.
\item If the $X_{\alpha}$ generate a locally finitely generated involutive
distribution, this distribution foliates the ambient space into leaves.
Our assumptions require that $\gamma_t\q(x\w)$ lies in the leaf of this foliation passing through $x$ for every $t$.
This is not always the case.  For instance, $\gamma_t\q(x\w):\R\times \R\rightarrow \R$ given by,
\begin{equation*}
\gamma\q( t,x\w): x\mapsto x-e^{-\frac{1}{t^2}},
\end{equation*}
does not lie in the appropriate leaf.
In this case, all the $X_{\alpha}$ are $0$ (and
so each leaf is merely a point).
\item Even if the $X_{\alpha}$ generate a locally finitely generated
involutive distribution, and $\gamma_t\q(x\w)$ lies in the appropriate leaf, it may still be that $\gamma$
is not controlled by the list $\q( X,d\w)$.
We work on $\R$.  Define the vector field $W\q( t,x\w)$ by,
\begin{equation*}
W\q( t,x\w) = t e^{-\frac{1}{x^2}} \partial_x+e^{-\frac{1}{t^2}} x\partial_x.
\end{equation*}
Let $\gamma_t$ be function associated to this $W$ (via Proposition \ref{PropGetGammaFromW}).
Note that,
\begin{equation*}
\gamma_t\q( x\w)\text{ is } 
\begin{cases}
\text{negative,}&\text{if }x\text{ is negative,}\\
\text{zero,}&\text{if }x\text{ is zero,}\\
\text{positive,}&\text{if }x\text{ is positive.}
\end{cases}
\end{equation*}
In this case, there is only $1$ $X_{\alpha}$, namely 
$X_{1}=e^{-\frac{1}{x^2}}\partial_x$.  Thus, the leaves
of the foliation are
\begin{equation*}
\q( -\infty, 0\w), \quad \q\{0\w\}, \quad \q( 0,\infty\w).
\end{equation*}
Hence, the proof will be complete if we can show that
$\gamma$ is not controlled by $\q( e^{-\frac{1}{x^2}}\partial_x, 1\w)$.
If $\gamma$ were controlled, it would imply, in particular, that there
exists a $t_0\ne 0$ such that for every $x$ near $0$, we have,
\begin{equation*}
e^{-\frac{1}{t_0^2}} x = c\q( x\w) e^{-\frac{1}{x^2}},
\end{equation*}
with $c\q( x\w)$ bounded uniformly as $x\rightarrow 0$.  This is clearly
impossible.
\end{enumerate}

Note that in each of the above examples, we used functions that
vanished to infinite order at a point.  One might wonder if such
a phenomenon is essential in creating an example where our
assumptions fail in the single parameter case.  Indeed, it is
essential.  It turns out that, when restricting attention
to real analytic $\gamma$, our assumptions
are automatically satisfied in the single parameter case.  This is taken
up in \cite{StreetMultiParameterSingRadonAnal}.

We now turn to our multi-parameter assumptions.  Here, there is a new phenomenon
which appears, even when considering real analytic $\gamma$:  one
may have that the vector fields corresponding to the non-pure powers
are not controlled by the pure powers.  For instance,
consider the operator on $\R$ given by,
\begin{equation}\label{EqnExMultiParamEx}
f\mapsto \psi\q( x\w) \int f\q( x-st\w) \: K\q( s,t\w)\: ds\: dt,
\end{equation}
where $K$ is a product kernel relative to the decomposition
$\q( s,t\w) \in \R\times \R$ (and $K$ is supported near $0$).
Here, there is only one vector field $\partial_x$ given
degree $\q( 1,1\w)$ (we are using the standard two-parameter dilations on 
$\R\times\R$).  Thus, this example does not satisfy our assumptions.
In fact, the methods in Section \ref{SectionExTransInv} can be used to show that there is a $K$ such that \eqref{EqnExMultiParamEx}
is not bounded on $L^2$.

	\subsection{A translation invariant example}\label{SectionExTransInv}
	In this section, we investigate some translation invariant
operators on $\R$, in the two-parameter setting.
In particular, we will investigate operators on $L^2\q( \R\w)$ given by,
\begin{equation}\label{EqnExTransInv}
T_K f\q( x\w) := \int f\q( x-p\q( s,t\w) \w) K\q( s,t\w)\: ds \: dt,
\end{equation}
where $p$ is a polynomial without a constant term and
$K$ is a product kernel on $\R^2$.  That is we are taking the
two-parameter dilations $e=\q( \q(1,0\w),\q(0,1\w)\w)$, so that
$\q( \delta_1,\delta_2\w) \q( s,t\w) = \q( \delta_1 s, \delta_2 t\w)$
and we are considering $K\in \sK\q( 2, e, a_0,\q[0,1\w]^2\w)$
for some small $a_0>0$.

\begin{rmk}
Note we have not included a cutoff function in \eqref{EqnExTransInv}.
Since $K$ has small support, the value of $T_K f\q( x\w)$ only depends
on the values of $f$ near $x$.  It is not hard to see, from this and
the translation invariance of $T_K$,
that $T_K$ is bounded on $L^2$ if and only if $\psi T_K$ is bounded on $L^2$
for some $C_0^\infty$ cut-off function $\psi$ with $\psi\equiv 1$
on a neighborhood of $0$.
\end{rmk}

\begin{rmk}
We have restricted to one-dimensional, two-parameter operators
for simplicity.  The same methods in this section apply to higher dimensional
higher parameter operators.
\end{rmk}

\begin{rmk}
Results like the ones in this section date back to
\cite{NagelWaingerL2BoundednessOfHilbertTransformsMultiParameterGroup}.
Thus, the results in this section are partially expository, though
our perspective is somewhat different.
\end{rmk}

Write
\begin{equation*}
p\q( s,t\w) =\sum_{\q|\alpha\w|\geq 1} c_\alpha \q( s,t\w)^\alpha.
\end{equation*}
Let $a>0$ be the lowest $a$ such that $c_{\q( a,0\w)}\ne 0$, and
$b$ be the lowest $b$ such that $c_{\q( 0,b\w)}\ne 0$.  For simplicity,
we assume $a,b<\infty$, but this it not necessary for what follows.

\begin{thm}\label{ThmExTransInvThm}
$T_K$ is bounded on $L^2$ for every $K$ with sufficiently small support if and only if whenever $c_{\q( e,f\w)}\ne 0$,
we have that $\q(e,f\w)$ lies on or above the line passing through
$\q( a,0\w)$ and $\q( 0,b\w)$.
\end{thm}
\begin{proof}
The if direction follows from Proposition \ref{PropGammaSatisV}.
Indeed, if we set $\gamma_{\q( s,t\w)}\q( x\w) = x-p\q( s,t\w)$,
then,
\begin{equation*}
\gamma_{\q( s,t\w)}\q( x\w) = e^{-p\q( s,t\w) \partial_x}x.
\end{equation*}
To see that Proposition \ref{PropGammaSatisV} applies, one merely
needs to note that $\q( \partial_x, \q( a,0\w)\w)$ and $\q( \partial_x, \q( 0,b\w)\w)$ control $\q( \partial_x,\q(e,f\w)\w)$ if and only if $\q( e,f\w)$
lies on or above the line passing through $\q( a,0\w)$ and $\q( 0,b\w)$.
This is easy to see directly, and is also covered in Example 5.16
of \cite{StreetMultiParameterCCBalls}.

For the only if direction, we will show if there exists $c_{\q( e,f\w)}\ne 0$
and $\frac{1}{a} e+ \frac{1}{b}f <1$,
then there exists a product kernel $K$
with small support such that $T_K$ is unbounded on $L^2$.
Let $e,f$ be such that $c_{\q( e,f\w)}\ne 0$ and $\frac{1}{a}e+\frac{1}{b}f$
is minimal among all such $\q( e,f\w)$.  Furthermore, we assume
that $e$ is minimal among these choices.  Note that, by assumption,
$0\ne e,f$.

First we wish to reduce the question.  Note that product kernels
are a Fr\'echet space (see, for instance, Definition \ref{DefnProdKer}).
Hence, by closed graph theorem, if $T_K$ were bounded for each choice
of $K$, then the map $K\mapsto T$ would be continuous as a map
from the space of product kernels to the space of bounded operators
on $L^2$ (given the uniform operator topology).
Hence, to prove the theorem, it suffices to construct a sequence $\q\{K_j\w\}$
of product kernels such that,
\begin{itemize}
\item $\q\{K_j\w\}$ is a bounded set in the space 
of product kernels.
\item The $K_j$ all have arbitrarily small support. 
\item $\LpOpN{2}{T_{K_j}}\rightarrow \infty$ as $j\rightarrow\infty$.
\end{itemize}

Take $\vsig\in C_0^\infty\q( \q( 0,1\w) \w)$ such that $\int \vsig=0$, and
\begin{equation}\label{EqnExConstvsig}
\int \exp\q({ic_{\q( e,f\w)} s^e t^f}\w) \vsig\q( s\w) \vsig\q( t\w) \: ds\: dt\ne 0.
\end{equation}
Note that such a $\vsig$ clearly exists.  Indeed, take $\eta\in C_0^\infty\q( 0,1\w)$ such that,
\begin{equation*}
\int \eta\q( t\w)\: dt  =0,\quad \int \eta\q( s\w) s^e \: ds \ne 0,\quad \int \eta\q( t\w) t^f\: dt \ne 0. 
\end{equation*}
We let $\dil{\eta}{r}\q( t\w) = r\eta\q( rt\w)$.  Then, we have,
\begin{equation*}
\begin{split}
&\int \exp\q({ic_{e,f}s^et^f}\w) \dil{\eta}{r}\q( s\w) \dil{\eta}{r}\q( t\w) \: ds \:dt = \\
&\quad i\q(\int \eta\q( s\w) s^e \: ds\w) \q(\int \eta\q( t\w) t^f\: dt\w) c_{\q( e,f\w)} r^{-e-f} + O\q( r^{-e-f-1}\w),\text{ as } r\rightarrow \infty.
\end{split}
\end{equation*}
Taking $\vsig=\dil{\eta}{r}$ for $r$ sufficiently large, we obtain
\eqref{EqnExConstvsig}.

For $M$ large, and $\tau>>M$ define,
\begin{equation*}
K_{\tau,M} \q( s,t\w) := \sum_{\substack{0\leq j\leq M\\k=-j\frac{e}{f}}} \dil{\vsig}{2^j \tau^{\frac{1}{a}}}\q( s\w) \dil{\vsig}{2^k\tau^{\frac{1}{b} }}\q( t\w).
\end{equation*}
Here, $j$ is an integer, but $k$ need not be.
Note that $K_{\tau,M}$ is a product kernel\footnote{For the
$K_{\tau,M}$ to be product kernels, it is not necessary that $2^{j} \tau^{\frac{1}{a}}$
be of the form $2^l$ for some $l$ (and similarly for $2^k\tau^{\frac{1}{b}}$).}
uniformly in $\tau$ and $M$.
Furthermore, for $M$ fixed and $\tau>>M$, $K_{\tau,M}$ has small support.
We will show,
\begin{equation}\label{EqnExTransBoundToShow}
\limsup_{M\rightarrow\infty} \limsup_{\tau\rightarrow \infty} \LpOpN{2}{T_{K_{\tau,M}}}=\infty,
\end{equation}
completing the proof.\footnote{Actually, the methods below imply
$\lim_{M\rightarrow\infty} \liminf_{\tau\rightarrow\infty} \LpOpN{2}{T_{K_{\tau,M}}}=\infty,$ but this stronger statement is irrelevant for what
we wish to prove.} 

Let $m_0=\frac{1}{a}e+\frac{1}{b}f$.  To prove \eqref{EqnExTransBoundToShow},
it suffices to show,
\begin{equation*}
\lim_{M\rightarrow\infty} \lim_{\tau\rightarrow\infty} \q|\int \exp\q({ip\q( s,t\w)\tau^{m_0}}\w) K_{\tau,M}\q(s,t\w)\: ds\: dt \w| =\infty,
\end{equation*}
since this is just the absolute value of the multiplier of $T_{K_{\tau,M}}$ evaluated at $\tau^{m_0}$.

Note that,
\begin{equation}\label{EqnExTransInvAftertau}
\begin{split}
\int &\exp\q( {ip\q( s,t\w) \tau^{m_0}} \w)K_{\tau, M}\q( s,t\w)\:ds\:dt\\
&= \sum_{\substack{0\leq j\leq M\\k=-j\frac{e}{f}}} \int \exp\q({ip\q( s,t\w) \tau^{m_0}}\w) \dil{\vsig}{2^{j} \tau^{\frac{1}{a}}}\q(s\w) \dil{\vsig}{2^{k}\tau^{\frac{1}{b}}}\q(t\w) \: ds \: dt\\
&= \sum_{\substack{0\leq j\leq M\\k=-j\frac{e}{f}}} \int \exp\q({i p\q(\tau^{-\frac{1}{a}} s, \tau^{-\frac{1}{b}}t \w)\tau^{m_0}}\w) \dil{\vsig}{2^j}\q( s\w) \dil{\vsig}{2^k}\q( t\w) \: ds\: dt\\
&  \xrightarrow[\tau\rightarrow\infty]{} \sum_{\substack{0\leq j \leq M\\k=-j\frac{e}{f}}} \int \exp\q({i\sum_{\frac{1}{a}g +\frac{1}{b}h =m_0} c_{\q( g,h\w)} \q( s,t\w)^{\q( g,h\w)}}\w) \dil{\vsig}{2^j}\q( s\w) \dil{\vsig}{2^k}\q( t\w) \: ds\: dt\\
&=\sum_{0\leq j\leq M} \int \exp\q({i \sum_{\frac{1}{a}g +\frac{1}{b}h =m_0} c_{\q( g,h\w)} \q( s,t\w)^{\q( g,h\w)} 2^{j\q(h\frac{e}{f}-g\w)}}\w) \vsig\q( s\w) \vsig\q( t\w) \: ds\: dt.
\end{split}
\end{equation}
Let $\q(\sum_{j=0}^M \cdot\w)$ denote the final equation in \eqref{EqnExTransInvAftertau}. 
We will complete the proof by showing,
\begin{equation*}
\lim_{M\rightarrow\infty} \q|\sum_{j=0}^M \cdot\w| =\infty.
\end{equation*}
To do so, it suffices to show,
\begin{equation*}
\lim_{j\rightarrow\infty} \int \exp\q({i \sum_{\frac{1}{a}g +\frac{1}{b}h =m_0} c_{\q( g,h\w)} \q( s,t\w)^{\q( g,h\w)} 2^{j\q(h\frac{e}{f}-g\w)}}\w) \vsig\q( s\w) \vsig\q( t\w) \: ds\: dt\ne 0,
\end{equation*}
i.e., this limit exists and is nonzero. 
Note, when $\q(g,h\w)=\q( e,f\w)$, we have $h\frac{e}{f}-g=0$.  Otherwise,
if $c_{\q( g,h\w)}$ is nonzero and $\frac{1}{a}g +\frac{1}{b}h =m_0$,
we have $h\frac{e}{f}-g<0$, by the minimality of our choice for $e$.
Thus,
\begin{equation*}
\begin{split}
&\lim_{j\rightarrow\infty} \int \exp\q({i \sum_{\frac{1}{a}g +\frac{1}{b}h =m_0} c_{\q( g,h\w)} \q( s,t\w)^{\q( g,h\w)} 2^{j\q(h\frac{e}{f}-g\w)}}\w) \vsig\q( s\w) \vsig\q( t\w) \: ds\: dt\\
&\quad = \int \exp\q( i c_{\q( e,f\w)} \q( s,t\w)^{\q( e,f\w)} \w) \vsig\q( s\w) \vsig\q( t\w)\\
&\quad \ne 0,
\end{split}
\end{equation*}
by our choice of $\vsig$; this completes the proof.
\end{proof}

\begin{rmk}
If we had not restricted our attention to $K$ with small support, we
could have reversed the argument (taking, instead, $\tau$ small) and
shown that
there exists a product kernel $K$ such that $T_K$ is unbounded
if there is a $c_{\q( e,f\w)}\ne 0$ with $\q(e,f\w)$ lying {\it above}
the line passing through $\q( a',0\w)$ and $\q( 0,b'\w)$, where
$a'$ is the maximal $a'$ such that $c_{\q( a',0\w)}\ne 0$, and
similarly for $b'$.
\end{rmk}

\begin{rmk}
Theorem \ref{ThmExTransInvThm} exemplifies the need to distinguish
between pure powers and non-pure powers.  In this case,
the pure powers are those $\alpha\in \q\{\q( \at, 0\w), \q( 0,\bt\w)\w\}$
such that $c_\alpha \ne 0$, while the non-pure powers are remaining
$\alpha$ such that $c_\alpha\ne 0$.
\end{rmk}

\begin{rmk}
The reader familiar with 
\cite{CarberyWaingerWrightSingularIntegralsAndTheNewtonDiagram}
will note that the necessary and sufficient conditions for the boundedness
of $T_K$, where $K\q( s,t\w) = \frac{1}{st}$ are much more complicated than
the conditions in Theorem \ref{ThmExTransInvThm}.  
However, there are similarities between the results in
\cite{CarberyWaingerWrightSingularIntegralsAndTheNewtonDiagram}
and Theorem \ref{ThmExTransInvThm}.  Indeed, 
the assumptions on $p$ in Theorem \ref{ThmExTransInvThm} can
be restated as saying that the vertices of the Newton
diagram of $p$ lie on the coordinate axes.
\end{rmk}

We now turn to contrasting this with the flag kernel setting.
We take all the same notation as above, but set,
\begin{equation*}
\sA=\q\{\q( \delta_1,\delta_2\w)\in \q[0,1\w]^2: \delta_1\leq \delta_2\w\}.
\end{equation*}
We will be considering operators $T_K$ such that $K\in \sK\q( 2,e,a_0,\sA\w)$
for some small $a_0>0$.

We take $p$ and $0<a,b<\infty$ as above.  We assume that $b\leq a$.
We have,
\begin{thm}\label{ThmExTransInvThmFlag}
In the above setting, $T_K$ is bounded on $L^2$ for every $K\in \sK\q( 2,e,a_0,\sA\w)$ (with $a_0>0$ sufficiently small) if and only if for every 
$c_\alpha\ne 0$, $\alpha$ lies on or above the line passing through
$\q( b,0\w)$ and $\q( 0,b\w)$.
\end{thm}
\begin{proof}
The if direction follows from Proposition \ref{PropGammaSatisV}.
Indeed, one needs only verify the straightforward fact that
$\q( \partial_x,\q( 0,b\w)\w)$ controls $\q(\partial_x,\alpha\w)$
if and only if $\alpha$ lies on or above the line passing through
$\q( 0,b\w)$ and $\q( b,0\w)$.  Recall, the definition of
control depends on the set $\sA$.

The only if follows just as in the proof of Theorem \ref{ThmExTransInvThm},
replacing $a$ with $b$ throughout.  It is easy to see, that when $a$
is replaced by $b$, the kernels $K_{\tau,M}$ are uniformly
flag kernels.
\end{proof}

\begin{rmk}
We will see, in \cite{SteinStreetMultiParameterSingRadonLp}, that
the operators in Theorems \ref{ThmExTransInvThm} and
\ref{ThmExTransInvThmFlag} which are bounded on $L^2$
are also bounded on $L^q$ ($1<q<\infty$).
\end{rmk}

	\subsection{Flag kernels versus product kernels}\label{SectionFlagVsProduct}
	As mentioned earlier, if we are considering operators associated
to product kernels,
\begin{equation*}
f\mapsto \psi\q( x\w) \int f\q( \gamma_t\q( x\w) \w) K\q( t\w) \: dt,
\end{equation*}
then Theorem \ref{ThmMainThmSecondPass} holds for a larger
class $\gamma$ when $K$ is assumed to be a flag
kernel (a special case of product kernels).  An example can be seen
by contrasting Theorems \ref{ThmExTransInvThm} and \ref{ThmExTransInvThmFlag}.

In this section, we exhibit another, simpler example of $\gamma$ where
our assumptions hold for flag kernels, but not product kernels.
Unlike Theorems \ref{ThmExTransInvThm} and \ref{ThmExTransInvThmFlag}, we make no claim about the unboundedness of operators
to which our theorem does not apply.

We consider product (and flag) kernels on $\R^2$ supported near the origin.
We take the standard two-parameter dilations on $\R^2$:
\begin{equation*}
\q( \delta_1,\delta_2\w) \q( s,t\w) = \q( \delta_1 s,\delta_2 t\w).
\end{equation*}
Thus, $e=\q( \q(1,0\w),\q(0,1\w)\w)$.
Let $\sA=\q\{\q( \delta_1,\delta_2\w)\in \q[0,1\w]^2 : \delta_1\leq \delta_2\w\}$.
Let $X_{\q( 1,0\w)}=\partial_x$, 
$X_{\q( 2,0\w)} = e^{-\frac{1}{x^2}}\partial_y$, and 
$X_{\q( 0,1\w)} = \partial_y$.
Define,
\begin{equation*}
\gamma_{\q(s,t\w)}\q( x,y\w) = e^{sX_{\q( 1,0\w)} + s^2 X_{\q( 2,0\w)} +tX_{\q( 0,1\w)}}\q( x,y\w).
\end{equation*}
Then, for $K\in \sK\q( 2, e, a, \sA\w)$ (for $a>0$ sufficiently small)
Theorem \ref{ThmMainThmSecondPass} (see also Proposition \ref{PropGammaSatisV})
applies to show that operators of the form
\begin{equation*}
f\mapsto \psi\q( x,y\w) \int f\q( \gamma_{\q( s,t\w)}\q( x,y\w)\w) K\q( s,t\w)\: ds\: dt
\end{equation*}
are bounded on $L^2$, where $\psi$ is a cut off function supported near
$0\in \R^2$.

However, it is easy to see (again via Proposition \ref{PropGammaSatisV})
that the assumptions of Theorem \ref{ThmMainThmSecondPass}
are not satisfied for the full class of product kernels: $\sK\q( 2,e,a, \q[0,1\w]^2\w)$.

\begin{rmk}
The example in this section used functions that vanish to infinite order.
However, we see by contrasting Theorems \ref{ThmExTransInvThm} and \ref{ThmExTransInvThmFlag},
that even when $\gamma_t\q( x\w)$ is real analytic, 
it could be that operator corresponding to each flag kernel
is bounded on $L^2$, while there are operators corresponding
to product kernels which are not bounded on $L^2$.
\end{rmk}

	\subsection{A multi-parameter difficulty}\label{SectionMultiParamDifficult}
	One of the crucial notions that we used was
when a list of vector fields
with formal degrees $\q( X_1,d_1\w),\ldots, \q( X_r,d_r\w)$
``generated a finite list,'' see Definition \ref{DefnGeneratesAFiniteList}.
In the single parameter case ($\nu=1$, $\sA=\q[0,1\w]$), we saw in Lemma
\ref{LemmaSpecialCaseEquivFiniteGen}
that this notion is essentially equivalent to the set $\q\{X_1,\ldots, X_r\w\}$
generating an involutive distribution which is locally finitely generated as
a $C^\infty$ module.
The point of this section is to exhibit an example showing
that the obvious multi-parameter generalization of the above does not hold.

More specifically, let $\nu=2$ and $\sA=\q[0,1\w]^2$.  Suppose
we are given two lists of vector fields with two-parameter formal degrees,
$\q( X_1,d_1\w),\ldots, \q( X_r,d_r\w)$, $\q( Y_1,d_1'\w),\ldots, \q(Y_{r'},d_{r'}'\w)$.
We assume that $d_1,\ldots,d_r$ are all $0$ in the second component,
while $d_1',\ldots, d_{r'}'$ are all $0$ in the first component.
Suppose we wish to show that,
\begin{equation}\label{EqnExToShowFiniteList}
\q( X_1,d_1\w),\ldots, \q( X_r,d_r\w),\q( Y_1,d_1'\w),\ldots, \q(Y_{r'},d_{r'}'\w)
\end{equation}
generates a finite list.
In analogy with the single-parameter case, one might hope that the following
(necessary) condition would be sufficient to prove that \eqref{EqnExToShowFiniteList}
generates a finite list:
\begin{itemize}
\item The involutive distribution generated by $\q\{ X_1,\ldots, X_r\w\}$ is
locally finitely generated as a $C^\infty$ module.
\item The involutive distribution generated by $\q\{ Y_1,\ldots, Y_{r'}\w\}$ is
locally finitely generated as a $C^\infty$ module.
\item The involutive distribution generated by $\q\{ X_1,\ldots, X_r, Y_1,\ldots, Y_{r'}\w\}$ is locally finitely generated as a $C^\infty$ module.
\end{itemize}
This is not the case.  

We exhibit an example where the above hold, but which
does not generate a finite list.  Our vector fields are,
\begin{equation*}
\q(\partial_x, \q( 1,0\w) \w), \q(e^{-\frac{1}{x^2}} \partial_y, \q( 0,1\w)\w), \q( \partial_y, \q( 0,2\w) \w).
\end{equation*}
Let $g\q( x\w) = e^{-\frac{1}{x^2}}$, and let $\dil{g}{m}$ denote the
$m$-th derivative of $g$.  If the above vector fields generated
a
finite list on a neighborhood of $0$, then, in particular there would
exist an $m$ such that
for every $1\geq \delta>0$,
\begin{equation}\label{EqnExFiniteListImp}
\delta \dil{g}{m}\q( x\w) = \sum_{j<m} \delta c_j^\delta\q( x\w) \dil{g}{j}\q( x\w)  + \delta^2 c^{\delta}\q( x\w),
\end{equation}
where $c_j^\delta$ and $c^\delta$ are bounded uniformly in $\delta>0$ and $x>0$.  To
see this we have taken the special case $\q( 1,\delta\w)\in \q[0,1\w]^2$.
\eqref{EqnExFiniteListImp} is clearly impossible.

\begin{rmk}
Once again, we have used functions which vanish to infinite order
at a point.  This is essential, in the sense that the above
phenomenon does not occur in the real analytic category.  We will
see this fact in \cite{StreetMultiParameterSingRadonAnal}.
\end{rmk}

	\subsection{An example on the Heisenberg group}\label{SectionExampOnHeis}
	In this section, we discuss in more detail the example
given at the end of
Section \ref{SectionKernelsII}.
The operators considered in this section
are less singular than those in 
Section \ref{SectionExampNilpotent}.
We work on the three dimensional Heisenberg group\footnote{The methods
discussed here generalize to any stratified nilpotent Lie group.}
$\Ho$.  As a manifold $\Ho$ is diffeomorphic to $\R^3$.
We write $\q( x,y,t\w)$ for coordinates of $\Ho$.
The group law on $\Ho$ is given by
\eqref{EqnHeisGroupLaw}
and we define two parameter dilations on
$\Ho$ by \eqref{EqnHeisTwoParamDilations}.
The Lie algebra of $\Ho$ is three dimensional, and one may take
as a basis,
$X=\partial_x+2y\partial_t$, $Y=\partial_y-2x\partial_t$, $T=\partial_t$.

Given a bounded subset $\q\{\vsig_j\w\}_{j\in \Z^2}\subset C_0^\infty\q( \Q^3\q( 1\w)\w)$ satisfying,
\begin{equation*}
\int \vsig_j\q( x,y,t\w) \: dx\: dt=0, \quad \int \vsig_j\q( x,y,t\w)\: dy\: dt, \forall j,
\end{equation*}
we define a distribution by,
\begin{equation*}
K\q( x,y,t\w) =\sum_{\q( j_1,j_2\w) \in \Z^2} 2^{2j_1+2j_2} \vsig_j\q( 2^{j_1} x, 2^{j_2} y, 2^{j_1+j_2}t\w) = \sum_{j\in \Z^2} \dil{\vsig_j}{2^j}\q( x,y,t\w).
\end{equation*}

\begin{thm}\label{ThmBoundOfToy}
The operator given by,
\begin{equation*}
f\mapsto f*K
\end{equation*}
is bounded $L^2\q( \Ho\w) \rightarrow L^2\q( \Ho\w)$.
\end{thm}

It is not hard to see that the operator in Theorem \ref{ThmBoundOfToy}
is essentially of the form covered by Theorem \ref{ThmMainThmSecondPass}.
The differences are:  we have not restricted to $K$ with small support,
and we have not localized the operator.  However, using the dilation
invariance of the class of operators in Theorem \ref{ThmBoundOfToy}, it is easy to see that Theorem \ref{ThmBoundOfToy}
follows from the result for $K$ with small support and $T$ localized.
Alternatively, our entire proof goes through, essentially unchanged
for the operators in Theorem \ref{ThmBoundOfToy}:  one does not
need to use the small support of $K$ or the localizing functions
in this case.  We leave the details to the reader.

At this point we wish to discuss the proof
of Theorem \ref{ThmMainThmSecondPass} in the context of
Theorem \ref{ThmBoundOfToy}.
Indeed, define the operator,
\begin{equation*}
T_j f = f*\dil{\vsig_j}{2^j}, \quad j\in \Z^2.
\end{equation*}
It is not hard to see that $T_j^{*}T_j$ is essentially of the same
form as $T_j$.  Thus, at first glance, one might think
that a $T^{*}T$ iteration type argument would not be useful
in proving Theorem \ref{ThmBoundOfToy}.  On the contrary, however,
a $T^{*}T$ iteration argument was essential to our proof.
The idea is that when $j,k\in \Z^2$ and $j_1>k_1$ and $j_2<k_2$,
the operator $\q(T_j^{*}T_k \w)^{*} T_j^{*}T_k$
has a less singular Schwartz kernel than $T_j^{*}T_k$.
This is the essence of our proof.

We close this section with one final remark.
If the convolution in Theorem \ref{ThmBoundOfToy}
is replaced by the Euclidean convolution on $\R^3$,
then the operator may be unbounded on $L^2$.\footnote{The closely
related fact that the double Hilbert transform on the surface $z=xy$
in $\R^3$ is not bounded on $L^2$ was noted by Ricci and
Stein \cite{RicciSteinMultiparameterSingularIntegralsAndMaximalFunctions}.}
This further accents the need to distinguish between
pure powers and non-pure powers.  The point is,
because of the dilations \eqref{EqnHeisTwoParamDilations},
$x$ and $y$ are pure powers, while $t$ is a non-pure power.
In the Heisenberg group case, the vector fields
corresponding to the pure powers
are $\q( X,\q( 1,0\w)\w)$, $\q( Y, \q( 0,1\w)\w)$.
The vector field corresponding to the non-pure power
is $\q( T,\q( 1,1\w)\w)= \q(\frac{1}{4} \q[X,Y\w], \q( 1,0\w)+\q( 0,1\w) \w)$
and therefore is controlled by the commutator of $X$ and $Y$.
In the Euclidean case, the vector fields corresponding
to the pure powers are $\q( \partial_x, \q( 1,0\w)\w)$
and $\q( \partial_y, \q( 0,1\w)\w)$; while the vector
field corresponding to the non-pure power
is $\q( \partial_t, \q( 1,1\w)\w)$ which is not controlled
by the previous two.

Finally, to see the operator might not be bounded on $L^2$ if the usual
Euclidean convolution is used, we need only show
that there is a $K$ as above whose Fourier transform is unbounded.
To see this, let $\phi\in C_0^{\infty}\q( \q( -1,1\w)\w)$
be such that $\widehat{\phi}\q(0\w) =1$.
Let $\psi\in C_0^{\infty}\q( \q( -1,1\w)\w)$ be
such that $\int \psi = 0$, $\widehat{\psi}\q( 1\w)> 0$.
Define,
\begin{equation*}
\vsig_{j_1,j_2}\q(x,y,t\w) = 
\begin{cases}
\phi\q( x\w) \phi\q( y\w) \psi\q( t\w) & \text{ if } j_1=-j_2,\\
0 & \text{ otherwise.}
\end{cases}
\end{equation*}
It is easy to see that these $\vsig_j$
satisfy the hypotheses of Theorem \ref{ThmBoundOfToy}.
We have,
\begin{equation*}
\widehat{K}\q( 0,0,1\w) = \sum_{\substack{\q(j_1,j_2\w)\in \Z^2\\ j_1=-j_2}} \widehat{\psi}\q( 1\w) =\infty,
\end{equation*}
completing the proof.

\bibliographystyle{amsalpha}

\bibliography{radon}

\providecommand{\bysame}{\leavevmode\hbox to3em{\hrulefill}\thinspace}
\providecommand{\MR}{\relax\ifhmode\unskip\space\fi MR }
\providecommand{\MRhref}[2]{%
  \href{http://www.ams.org/mathscinet-getitem?mr=#1}{#2}
}
\providecommand{\href}[2]{#2}
\begin{thebibliography}{NRSW11}

\bibitem[Bou89]{BourgainARemarkOnTheMaximalFunctionAssociatedToAnAnalyticVecto%
rField}
J.~Bourgain, \emph{A remark on the maximal function associated to an analytic
  vector field}, Analysis at {U}rbana, {V}ol.\ {I} ({U}rbana, {IL},
  1986--1987), London Math. Soc. Lecture Note Ser., vol. 137, Cambridge Univ.
  Press, Cambridge, 1989, pp.~111--132. \MR{MR1009171 (90h:42028)}

\bibitem[Chr92]{ChristTheStrongMaximalFunctionOnANilpotentGroup}
Michael Christ, \emph{The strong maximal function on a nilpotent group}, Trans.
  Amer. Math. Soc. \textbf{331} (1992), no.~1, 1--13. \MR{MR1104197
  (92j:42018)}

\bibitem[CNSW99]{ChristNagelSteinWaingerSingularAndMaximalRadonTransforms}
Michael Christ, Alexander Nagel, Elias~M. Stein, and Stephen Wainger,
  \emph{Singular and maximal {R}adon transforms: analysis and geometry}, Ann.
  of Math. (2) \textbf{150} (1999), no.~2, 489--577. \MR{MR1726701
  (2000j:42023)}

\bibitem[CWW06]{CarberyWaingerWrightSingularIntegralsAndTheNewtonDiagram}
Anthony Carbery, Stephen Wainger, and James Wright, \emph{Singular integrals
  and the {N}ewton diagram}, Collect. Math. (2006), no.~Vol. Extra, 171--194.
  \MR{MR2264209 (2008c:42011)}

\bibitem[Die60]{DieudonneFoundationsOfModernAnalysis}
J.~Dieudonn{\'e}, \emph{Foundations of modern analysis}, Pure and Applied
  Mathematics, Vol. X, Academic Press, New York, 1960. \MR{MR0120319 (22
  \#11074)}

\bibitem[FP97]{FeffermanPipherMultiparameterOperatorsAnSharpWeightedInequaliti%
es}
R.~Fefferman and J.~Pipher, \emph{Multiparameter operators and sharp weighted
  inequalities}, Amer. J. Math. \textbf{119} (1997), no.~2, 337--369.
  \MR{1439553 (98b:42027)}

\bibitem[Gal79]{GalligoTheoremeDeDivisionEtStabiliteEnGeometrieAnalytiqueLocal%
e}
Andr{\'e} Galligo, \emph{Th\'eor\`eme de division et stabilit\'e en
  g\'eom\'etrie analytique locale}, Ann. Inst. Fourier (Grenoble) \textbf{29}
  (1979), no.~2, vii, 107--184. \MR{MR539695 (81e:32009)}

\bibitem[Lob70]{LobryControlabiliteDesSystemesNonLinearies}
Claude Lobry, \emph{Contr\^olabilit\'e des syst\`emes non lin\'eaires}, SIAM J.
  Control \textbf{8} (1970), 573--605. \MR{MR0271979 (42 \#6860)}

\bibitem[NRS01]{NagelRicciSteinSingularIntegralsWithFlagKernels}
Alexander Nagel, Fulvio Ricci, and Elias~M. Stein, \emph{Singular integrals
  with flag kernels and analysis on quadratic {CR} manifolds}, J. Funct. Anal.
  \textbf{181} (2001), no.~1, 29--118. \MR{MR1818111 (2001m:22018)}

\bibitem[NRSW11]{NagelRicciSteinWaingerSingularIntegralWithFlagKernelsOnHomoge%
neousGroupsI}
Alexander Nagel, Fulvio Ricci, Elias~M. Stein, and Stephen Wainger,
  \emph{{S}ingular integrals with flag kernels on homogeneous groups: {I}},
  2011, to appear in Rev. Mat. Iberoamericana.

\bibitem[NSW85]{NagelSteinWaingerBallsAndMetricsDefinedByVectorFields}
Alexander Nagel, Elias~M. Stein, and Stephen Wainger, \emph{Balls and metrics
  defined by vector fields. {I}. {B}asic properties}, Acta Math. \textbf{155}
  (1985), no.~1-2, 103--147. \MR{MR793239 (86k:46049)}

\bibitem[NW77]{NagelWaingerL2BoundednessOfHilbertTransformsMultiParameterGroup}
Alexander Nagel and Stephen Wainger, \emph{{$L^{2}$} boundedness of {H}ilbert
  transforms along surfaces and convolution operators homogeneous with respect
  to a multiple parameter group}, Amer. J. Math. \textbf{99} (1977), no.~4,
  761--785. \MR{MR0450901 (56 \#9192)}

\bibitem[RS76]{RothschildSteinHypoellipticDifferentialOperatorsAndNilpotentGro%
ups}
Linda~Preiss Rothschild and E.~M. Stein, \emph{Hypoelliptic differential
  operators and nilpotent groups}, Acta Math. \textbf{137} (1976), no.~3-4,
  247--320. \MR{MR0436223 (55 \#9171)}

\bibitem[RS92]{RicciSteinMultiparameterSingularIntegralsAndMaximalFunctions}
F.~Ricci and E.~M. Stein, \emph{Multiparameter singular integrals and maximal
  functions}, Ann. Inst. Fourier (Grenoble) \textbf{42} (1992), no.~3,
  637--670. \MR{1182643 (94d:42020)}

\bibitem[SS11a]{SteinStreetMultiParameterSingRadonLp}
Elias~M. Stein and Brian Street, \emph{Multi-parameter singular {R}adon
  transforms {II}: the ${L}^p$ theory}, 2011, in preparation.

\bibitem[SS11b]{StreetMultiParameterSingRadonAnal}
\bysame, \emph{Multi-parameter singular {R}adon transforms {III}: real analytic
  surfaces}, 2011, in preparation.

\bibitem[Str11]{StreetMultiParameterCCBalls}
Brian Street, \emph{Multi-parameter {C}arnot-{C}arath\'eodory balls and the
  theorem of {F}robenius}, 2011, to appear in Rev. Mat. Iberoamericana.

\bibitem[TW03]{TaoWrightLpImprovingBoundsForAverages}
Terence Tao and James Wright, \emph{{$L\sp p$} improving bounds for averages
  along curves}, J. Amer. Math. Soc. \textbf{16} (2003), no.~3, 605--638
  (electronic). \MR{MR1969206 (2004j:42005)}

\end{thebibliography}

\center{\it{University of Wisconsin-Madison, Department of Mathematics, 480 Lincoln Dr., Madison, WI, 53706}}

\center{\it{street@math.wisc.edu}}

\center{MCS2010: Primary 42B20, Secondary 42B25}

\center{Keywords: Calder\'on-Zygmund theory, singular integrals, singular radon transforms, product kernels, flag kernels.}

\end{document}